\documentclass[dvipdfmx,a4paper,11pt,leqno]{article}   
\usepackage{amssymb, amsmath, amsthm, graphicx, color, epsfig, 
amsfonts, tikz, url}

\usepackage{thmtools} 
\newtheorem{thm}{Theorem}[section]
\newtheorem{lem}[thm]{Lemma}
\newtheorem{cor}[thm]{Corollary}
\newtheorem{prop}[thm]{Proposition}
\newtheorem{rem}{Remark}[section]

\numberwithin{equation}{section}

\newtheorem{pre-prop}[thm]{Pre-Proposition}

\renewcommand{\b}{\beta}
\newcommand{\e}{\varepsilon}
\newcommand{\de}{\delta}
\newcommand{\fa}{\varphi}
\newcommand{\ga}{\gamma}

\newcommand{\la}{\lambda}
\renewcommand{\th}{\theta}
\newcommand{\si}{\sigma}

\renewcommand{\t}{\tau}

\newcommand{\De}{\Delta}
\newcommand{\Ga}{\Gamma}
\newcommand{\La}{\Lambda}
\newcommand{\Om}{\Omega}
\newcommand{\lan}{\langle}
\newcommand{\ran}{\rangle}

\def\R{{\mathbb{R}}}
\def\N{{\mathbb{N}}}
\def\Z{{\mathbb{Z}}}
\def\T{{\mathbb{T}}}

\def\ux{\underline{x}}
\def\uy{\underline{y}}

\allowdisplaybreaks

\newcommand{\vertiii}[1]{{\left\vert\kern-0.25ex\left\vert\kern-0.25ex\left\vert #1 
    \right\vert\kern-0.25ex\right\vert\kern-0.25ex\right\vert}}

\title{Stochastic PDE approach to fluctuating interfaces}
\author{Tadahisa Funaki}
\date{\today \\ \vskip 3mm
{\it Dedicated to the memory of Giuseppe Da Prato}
}

\begin{document}
\maketitle

\begin{abstract}
We propose a new type of SPDEs, singular or with regularized noises, 
motivated by a study of
the fluctuation of the density field in a microscopic interacting particle system.
They include a large scaling parameter $N$, which is the ratio of macroscopic 
to microscopic size, and another scaling parameter $K=K(N)$, which controls 
the formation 
of the interface of size $K^{-1/2}$ in the density field.  They are derived
heuristically from the particle system, assuming the validity of the so-called 
``Boltzmann-Gibbs principle", that is, a combination of the local ensemble average
due to the local ergodicity and its asymptotic expansion.
We study a simple situation where the interface is flat and immobile.
Under making a proper stretch to the normal direction to the
interface, we observe a Gaussian fluctuation of
the interface.  We also heuristically derive a nonlinear SPDE which
describes the fluctuation of the interface.

\footnote{
\hskip -6mm 
Beijing Institute of Mathematical Sciences and Applications, 
No.\ 544 Hefangkou, Huairou District, Beijing 101408, China.
e-mail: funaki@ms.u-tokyo.ac.jp }
\footnote{
\hskip -6mm
Abbreviated title: SPDE approach to fluctuating interfaces}

\end{abstract}

\section{Introduction}  \label{sec:1}

The mesoscopic approach based on stochastic PDEs for the fluctuating interfaces 
was initiated
by Kawasaki and Ohta \cite{KO82} in the physics literature.  Their motivation
was to study the dynamic phase transition.  Starting with the time-dependent
Ginzburg-Landau equation, also known as the dynamic $P(\phi)$-model:
\begin{align}  \label{eq:a} 
\partial_t \Phi(t,x) = \De \Phi(t,x)  -P'(\Phi(t,x)) +\dot{W}(t,x),
\end{align}
with $P(\phi)= -\frac{\t}2 \phi^2 + \frac{g}{4!}\phi^4, \t, g>0$,
they showed the occurrence of the phase separation and derived
a random evolution law for the phase separating interfaces.
Inspired by this, a mathematically rigorous study was developed in
one-dimension with a space-time Gaussian white noise by \cite{F95},
\cite{XZZ}  and in higher dimension with a certain temporal noise by
\cite{Fu99a}, \cite{We-2}; see Section 4 of \cite{F16} for related results.
Recall a seminal paper by Da Prato and Debussche \cite{DD} for
the equation \eqref{eq:a}.

On the other hand, the microscopic approach based on the interacting particle systems
to the fluctuating interfaces in the phase separation phenomena at the rigorous level
is more recent; see \cite{FLS}.  
The present article is a mesoscopic counterpart to \cite{FLS}
and we propose a new type of highly singular SPDEs; see \eqref{eq:2.3-1} below.
Such SPDEs, which are continuum equations with fluctuation term,
can be derived, at least heuristically, from the discrete particle systems 
assuming the validity of the so-called higher-order Boltzmann-Gibbs
principle (see \cite{FLS} and Appendix \ref{Appendix:A}), that is, a combination
of the local ensemble average for the microscopic system due to the local 
ergodicity and the asymptotic expansion.

Kawasaki emphasized in his book \cite{K} that the mesoscopic approach 
based on SPDEs is more fruitful than the microscopic approach.
This is also a basic philosophy in the theory of the fluctuating 
hydrodynamics; see \cite{S14}, \cite{SS15}, \cite{Donev18}, \cite{CF23},
\cite{CFIR}.

In Section \ref{sec:1-A}, we introduce the SPDEs \eqref{eq:2.3-1} which will 
be studied in this paper.  We work on the $d$-dimensional torus $\T^d$.
To explain them, let us assume that the scaled particle density field
$\rho^{N,K}(t,x)$ is given from the microscopic interacting particle system
called the Glauber-Kawasaki dynamics; cf.\ Appendix \ref{Appendix:A}.
We further assume that the dynamics
has two favorable stable states with particle densities $\rho_\pm$,
$\rho_-<\rho_+$.  Here, $N$ is a large parameter describing the ratio of
macroscopic to microscopic size, while $K=K(N)$ is another large parameter
which depends on $N$, but diverges slower than $N$,  and controls the 
formation of the interface separating two stable phases with densities $\rho_\pm$.
In the case where two phases have the same degree of stability, called balanced,
at the level of the hydrodynamic limit which is formulated as a law of large numbers,
it can be shown that $\rho^{N,K}(t,x)$ converges to $\Xi_{\Ga_t}(x)$ as
$N\to\infty$, where $\Xi_{\Ga}(x) = \rho_+$ for $x$ on one side of $\Ga$ 
and $\rho_-$ on the other side of $\Ga$, and 
the hypersurface $\Ga_t$ evolves under the mean 
curvature flow or more generally the anisotropic curvature flow
(cf.\ \cite{F23}).  

We are interested in the 
fluctuation of $\rho^{N,K}(t,x)$ around the limit $\Xi_{\Ga_t}(x)$:
$$
\Phi(t,x) := N^{d/2}\big(\rho^{N,K}(t,x)- \Xi_{\Ga_t}(x)\big),
$$
or, instead of $\Xi_{\Ga_t}(x)$, we take $u^K(t,x)$ which approximates 
$\Xi_{\Ga_t}(x)$; see \eqref{eq:2.5-Phi}.
Assuming the Boltzmann-Gibbs principle, one can formally derive the SPDE
\eqref{eq:2.3-1} for $\Phi$.  
To simplify the problem, we consider
the case that the interface $\Ga_t$ reaches the stationary situation, 
that is, $\Ga_t$ is flat and immobile.

Section \ref{sec:2.2-0} introduces scalings for the SPDE \eqref{eq:2.3-1} 
near the interface to draw out the fluctuation and to clearly observe the shape of
the transition layer by stretching the spatial variable to the normal direction
to the interface; see \eqref{eq:2.stretch} and \eqref{eq:2.2-4}.  We also 
summarize some results which follow from Carr and Pego \cite{CP}, which 
play a crucial role in our analysis.

We then discuss a linear Gaussian fluctuation in Section \ref{sec:2.2}.   
We will see that the fluctuation of the particle density in its value is 
small and negligible, while the fluctuation in the spatial direction becomes
observable by stretching the spatial variable to the normal direction to
the interface by the factor $N^{d/2}K^{-1/4}$.  In this way, the fluctuation of
the density field implies that of the interface.
The Gaussian limit is obtained when $K^{7/4} N^{-d/2} <\!\!< 1$.
In the one-dimensional case, the scaled interface becomes one point (actually two
points on $\T$ in our setting) located at
$\psi(t)$ and it fluctuates as a Brownian motion. In two-dimensional case,
the scaled interface becomes a curve described as a graph $\psi(t,\ux)$
on the interface and we obtain a linear SPDE for $\psi(t,\ux)$ 
to determine the evolution of the curve.
When $d\ge 3$, the fluctuation 
becomes a truly generalized function and we lose the interpretation 
as the fluctuation of the interface.

In Section \ref{sec:2.3}, we will heuristically discuss the nonlinear fluctuation 
limit by taking a proper scaling in $K$, i.e.,  $K=N^{2d/7}$, 
and then $K=N^{2d/5}$.  The nonlinearity has its origin in the Glauber part.
The fluctuation of the particle density field
away from the interface is discussed in Section \ref{sec:awayfromInterface}.
Section \ref{sec:2.4} is for the unbalanced case, in which
the flat interface moves with a constant and very fast velocity of order
$\sqrt{K}$.

Appendix \ref{Appendix:A} discusses the relation between our SPDEs and the
Glauber-Kawasaki dynamics.  Appendix \ref{Appendix:B} is for the extension 
of the results of \cite{CP} in our setting.

The aim of this paper is to propose new problems in a certain
class of singular SPDEs.  For linear SPDEs, our proofs are rigorous;
otherwise, our arguments are at the heuristic level.

\section{SPDE for particle density fluctuation}  \label{sec:1-A}

\subsection{Setting}  \label{sec:2.1}

Let $N\in \N$ and $K=K(N)\ge 1$ be two scaling parameters, both diverging 
to $+\infty$ but $K$ is slower than $N$.  Let $f\in C^\infty(\R)$ be a function 
which has exactly three zeros $f(\rho_-)=f(\rho_*) = f(\rho_+)=0$, 
$\rho_-<\rho^*<\rho_+$ and satisfies $f'(\rho_\pm)<0$ (i.e.\ $\rho_\pm$
are stable and $\rho_*$ is unstable).  Suppose the balance condition 
$\int_{\rho_-}^{\rho_+}f(u)du=0$ holds.
A typical example is $f(u)=u-u^3$ with $\rho_\pm = \pm 1$ and $\rho_*=0$.
In a setting of particle systems, $f\in C^\infty([0,1])$ and 
$0<\rho_-<\rho_*<\rho_+<1$; see Appendix \ref{Appendix:A}.
\begin{figure}[h]
\centering
\includegraphics[width=40mm]{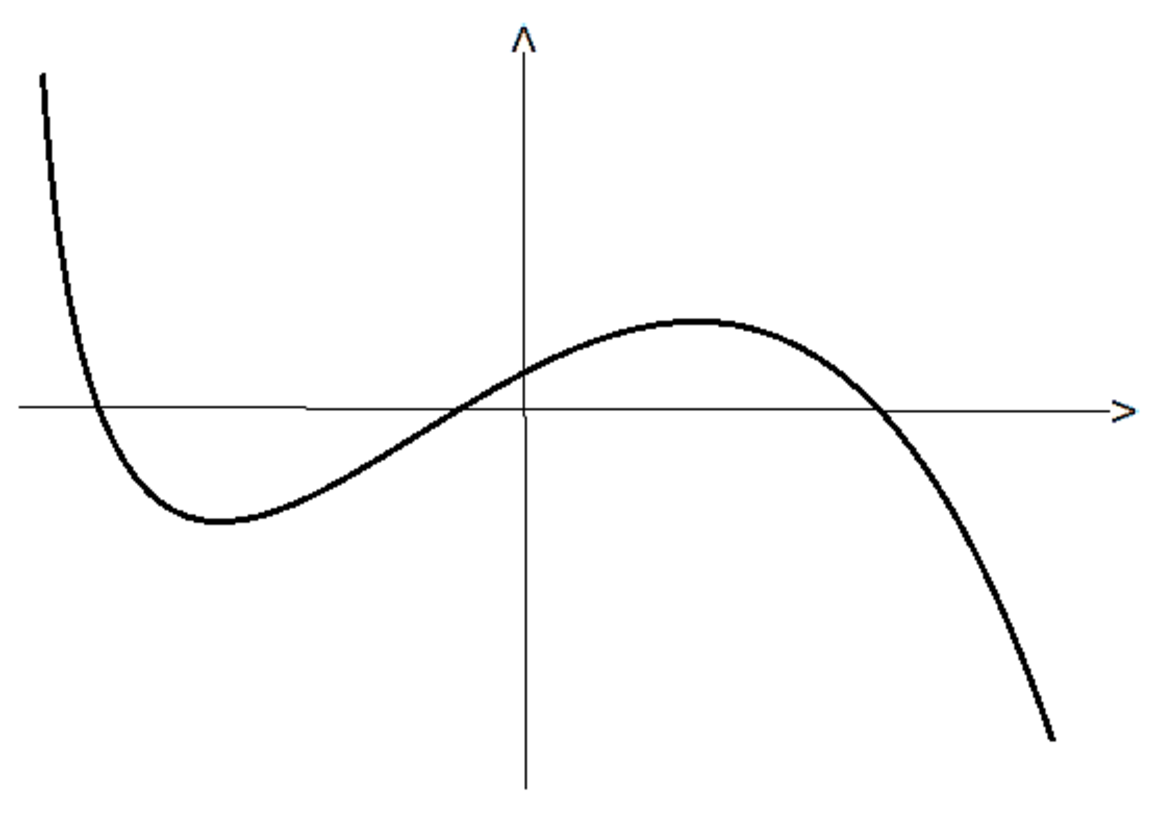}
\vskip -31mm
{\scriptsize \hskip 1mm $f$}
\vskip 8.9mm
{\scriptsize \hskip -3.7mm $\rho_-$ \hskip 10mm $\rho_*$ 
\hskip 13mm $\rho_+$}
\vskip 10mm
\caption{Bistable function $f$}  
  \label{Figure1}
\end{figure}

Let $\T^d$ be the $d$-dimensional torus, that is
$\T^d\equiv [0,1)^d$ with a periodic boundary condition.
Let $v(x_1) \equiv v^K(x_1), x_1 \in \T \equiv \T^1$ be a solution of
\begin{align}  \label{eq:baru} 
\partial_{x_1}^2 v + K f(v) =0, \quad x_1\in \T,
\end{align}
satisfying $\sharp\{x_1\in \T; v(x_1) = \rho_*\}=2$.  Such $v$
exists uniquely except translation; see Proposition \ref{prop:B.11}.  
It takes values in $(\rho_-,\rho_+)$.
To fix the idea, we normalize it as $v(0)=\rho_*$ and $v_{x_1}(0)<0$.  
Let $h_2\in (0,1)$ be uniquely determined by $v(h_2)=\rho_*$
and set $m_1=h_2/2$ and $m_2= (h_2+1)/2$ as in Figure \ref{Figure2}.
Then, $\{h_1=0, h_2\}$ indicates the locations of two transition layers with
width $O(K^{-1/2})$ and $\{m_1,m_2\}$ are the middle points of the layers.
\begin{figure}[h]
\centering
\includegraphics[height=28mm]{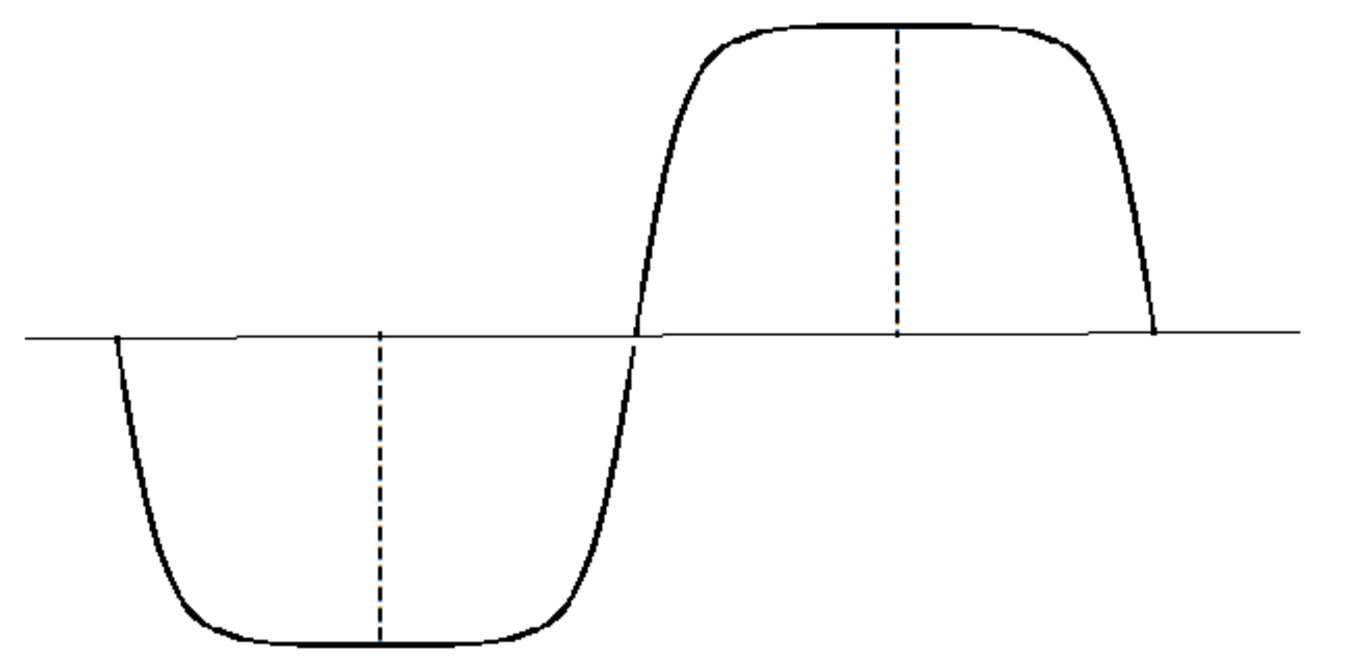}
\vskip -20mm \hskip -36mm
{\scriptsize  $h_1=0$ \hskip 4mm $m_1$ }

\vskip -1mm
{\scriptsize \hskip 20mm $h_2$   \hskip 5mm $m_2$  \hskip 6mm $1$}

\vskip -20.5mm
\hskip 47mm  {\tiny$\cdots\cdots\cdots \rho_+$}

\vskip 9mm
\hskip 58mm {\tiny$\rho_*$}

\vskip 9.5mm
\hskip 26mm {\tiny$\cdots\cdots\cdots\cdots\cdots\cdots\cdots\cdots\cdots \rho_-$}

\caption{Transition profile $v\equiv v^K$}  
  \label{Figure2}
\end{figure}

Decomposing $x\in \T^d$ as $x=(x_1,\ux) \in \T\times \T^{d-1}$, 
$\ux = (x_2,\ldots,x_d)$, we define $u^K(x) := v^K(x_1)$ on $\T^d$.  In particular,
the level sets of $u^K$ are hyperplanes; cf.\ \cite{Savin09}, \cite{Savin10} 
for this form of transition profiles in whole $\R^d$.
Note that $u^K(x)$ is a stationary solution of the 
Allen-Cahn equation on $\T^d$:
\begin{align}  \label{eq:AC}
\partial_t u = \De u + Kf(u),
\end{align}
with $x_1$-directed wave front.
It has an asymptotic behavior as in \eqref{eq:2.7-A} as $K\to\infty$.

\subsection{SPDEs}

Under the above preparation, for $n=1, 2$ and $3$, we consider the following
SPDE for $\Phi= \Phi^N\equiv \Phi^{N,K}(t,x)$, $t\ge 0$, $x\in \T^d$:
\begin{align}  \label{eq:2.3-1} 
\partial_t \Phi(t,x) = \De& \Phi(t,x)  + K F_n^{N}(u^K(x), \Phi(t,x)) \\
&+ \nabla\cdot \big( g_1(u^K(x))\dot{{\mathbb W}}(t,x) \big)
 + \sqrt{K}g_2(u^K(x)) \dot{W}(t,x),   \notag
\end{align}
where 
\begin{align}  \label{eq:F_n^N} 
F_n^{N}(u,\phi) = \sum_{k=1}^n \frac{N^{-(k-1)d/2}}{k!}  f^{(k)}(u) \phi^k,
\end{align}
for $(u,\phi) \in [\rho_-,\rho_+]\times \R$ and 
$g_1, g_2 \in C^\infty([\rho_-,\rho_+])$ are positive functions; see Section 
\ref{sec:1.4} and Appendix \ref{Appendix:A} for the origin of these functions
and scalings $K,\sqrt{K}$ in \eqref{eq:2.3-1} from the particle systems.
In particular when $n=3$,
\begin{align*}
F_3^{N}(u,\phi) = f'(u) \phi + \tfrac12 N^{-d/2} f''(u) \phi^2 + 
\tfrac16 N^{-d} f'''(u) \phi^3.
\end{align*}
Note that $F_n^{N}(u,\phi)$ is a Taylor expansion of
$$
N^{d/2}\big( f(u+N^{-d/2}\phi ) -f(u) \big)
$$
around $u$ up to the $n$th order terms; see Section \ref{sec:1.4}
for further explanation.  In \eqref{eq:2.3-1}, $\nabla\cdot (g_1\dot{{\mathbb W}})$
represents a conservative noise, which comes from the Kawasaki part in the
setting of the particle systems; see Appendix \ref{Appendix:A}.
Two types of noises such as $\nabla\cdot\dot{{\mathbb W}}$ and $\dot{W}$ 
also appear in the study of the stochastic eight vertex model; see (1.2) in
\cite{FNS}.

\subsection{Noises in the SPDE \eqref{eq:2.3-1} }
\label{sec:2.3-Q}

Under the limiting procedure from the particle systems, the noises
$\dot{{\mathbb W}}= \{\dot{W}^{i}\}_{i=1}^d$ and $\dot{W}$,
sometimes denoted by $\dot{W}^{d+1}$, in the SPDE \eqref{eq:2.3-1}
are expected to be mutually independent $d+1$ space-time 
Gaussian white noises on $[0,\infty)\times \T^d$, in particular, they have
the covariance structure:
$$
E[\dot{W}^i(t,x) \dot{W}^j(s,y)] = \de^{ij}\de(t-s)\de(x-y),\quad
t, s \ge 0, \; x, y \in \T^d,\; 1\le i,j \le d+1.
$$
Or, they may be $d+1$ space-time noises which depend on the scaling
parameters $N$ and $K$, and are asymptotically close (in law) to 
the independent space-time Gaussian white noises.

Let us briefly discuss the SPDE \eqref{eq:2.3-1} with the space-time
Gaussian white noises.  When $n=1$, the SPDE \eqref{eq:2.3-1} is linear in $\Phi$
and well-posed.  The solution takes values in $\mathcal{D}'(\T^d)$,
the space of distributions on $\T^d$ for every dimension $d$.  
The case of $n=2$ is inadequate, since we expect the blow-up of the
solution.  When $n=3$, the SPDE \eqref{eq:2.3-1} (with $f'''(u)<0$) 
has a cubic nonlinearity and looks close to the dynamic $P(\phi)$-model 
\eqref{eq:a}, but the difference is that our SPDE has a noise
$\nabla \cdot(g_1 \dot{\mathbb W})$ with bad regularity.  In fact, due to
the luck of regularity, the critical dimension of the SPDE \eqref{eq:2.3-1} is
$d=2$, while it is $d=4$ for \eqref{eq:a}.

The study of  the critical and supercritical SPDEs has recently made
progresses by
the regularization of the noises or the cutoff of their high frequency mode, 
for example, for the two-dimensional KPZ equation (see \cite{CT}, \cite{CSZ}; 
$d=2$ is critical), the one-dimensional KPZ equation with noise 
$\partial_x \dot{W}$ (see \cite{Hai}; $d=0$ is
critical), Dean-Kawasaki equation (see \cite{CFIR}) and others.

Instead of the space-time Gaussian white noises, we may take
independent Gaussian regularized noises $\{\dot{W}^{Q^i}\}_{i=1}^{d+1}$
with mean $0$ and covariance kernels $Q^{i}(x,y)$, respectively, i.e.
\begin{align*}
E[\dot{W}^{Q^i}(t,x) \dot{W}^{Q^j}(s,y)] = \de^{ij}\de(t-s)Q^{i}(x,y),  \quad
t, s \ge 0, \; x, y \in \T^d,\; 1\le i,j \le d+1.
\end{align*}
The covariance kernels $\{Q^i\}_{i=1}^{d+1}$ may depend on the scaling parameters 
$N$ and $K$, and we assume
\begin{align*}
&  \int_{\T^d} \Big\{ \big| Q^{i}(x,x)\big| + 
\Big| \partial_{x_i} \partial_{y_i}Q^{i}(x,y)\Big|_{x=y} \Big| \Big\} dx < \infty,  
\quad 1\le i \le d, \\
&  \int_{\T^d} \big| Q^{d+1}(x,x)\big| dx < \infty.
\end{align*}
Under these conditions, noting that
$g_1(u^K(x)), \partial_{x_1} g_1(u^K(x))$ and $\sqrt{K}g_2(u^K(x))$ are bounded 
on $\T^d$, $\nabla\cdot\big(g_1(u^K(x)) {\mathbb W}^{\mathbb Q}(t,x)\big)$ and
$\sqrt{K}g_2(u^K(x)) W^{Q^{d+1}}(t,x)$ are $L^2(\T^d)$-valued Brownian motions,
where ${\mathbb W}^{\mathbb Q}=\{W^{Q^i}\}_{i=1}^{d}$.
In particular, assuming an additional condition $f'''(u)<0, u\in [\rho_-,\rho_+]$ for
$f(u)$, since $f'''(u^K(x))<0$, the SPDE \eqref{eq:2.3-1} for $n=3$
has a unique global-in-time classical solution for such regularized noises.
Note that $f'''(u)=-6$ for $f(u)=u-u^3$.

In the setting of the particle systems, the noise terms are
martingales and not Gaussian; see Appendix \ref{Appendix:A}.

\subsection{Approximation of $v^K$ by standing wave and the interfaces $\Ga_1, \Ga_2$}  \label{sec:2.4-Q}

Let $U_0$ be a standing wave solution on $\R$ determined from $f$ such that
\begin{align} \label{eq:stand-wave}
\partial_z^2 U_0(z) + f(U_0(z)) =0, \; z \in \R,
\quad U_0(\pm\infty) = \rho_\pm,
\end{align}
normalized as $U_0(0)=\rho_*$.

\begin{figure}[h]
\centering
\includegraphics[width=60mm]{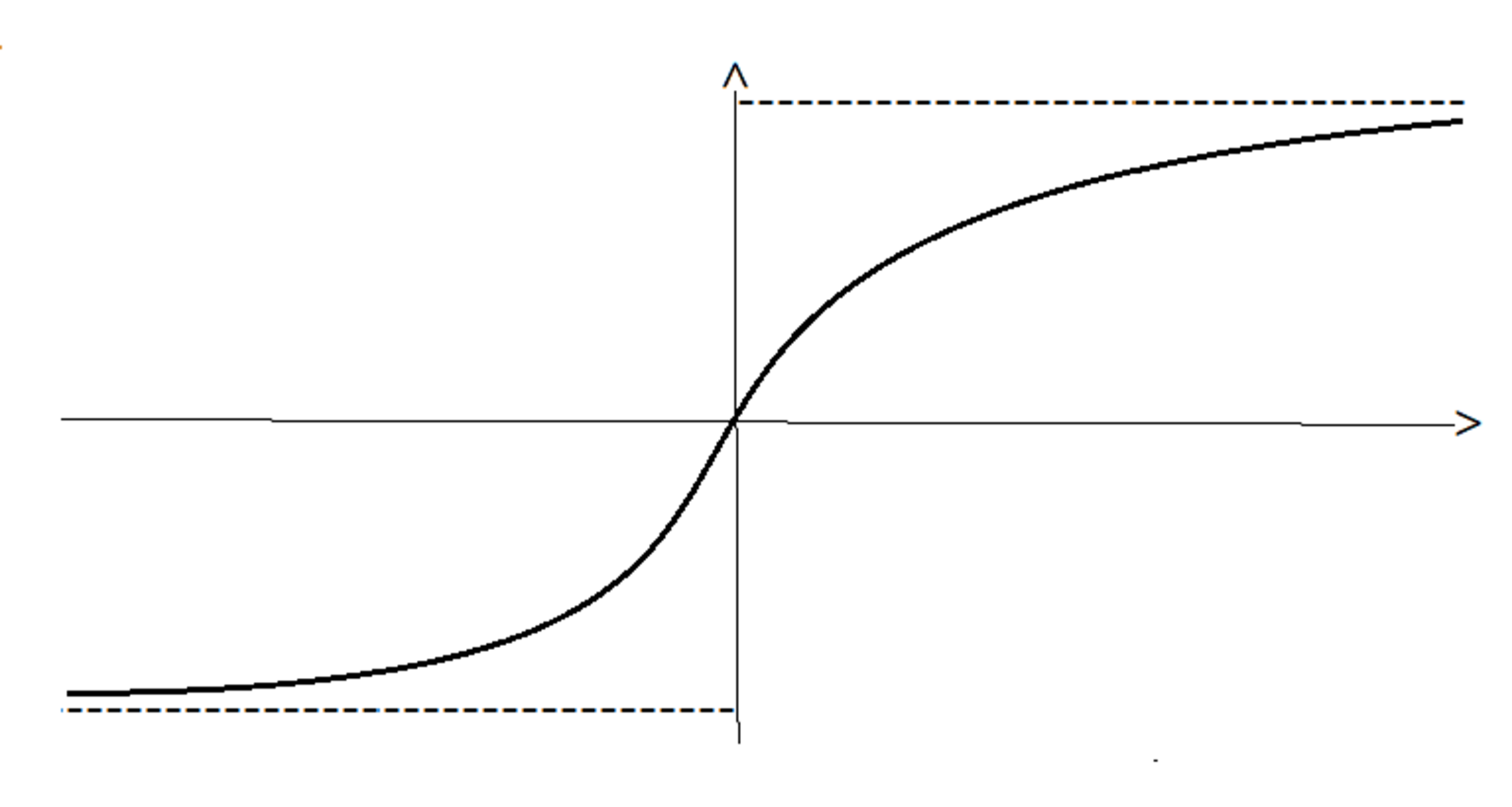}

\vskip -35.5mm
{\scriptsize \hskip 4mm $U_0$}

\vskip -1.2mm
{\scriptsize \hskip -5mm  $\rho_+$ }

\vskip 8mm
{\scriptsize  \hskip 27mm$\rho_*$
\hskip 30mm  $z$ }

\vskip -3mm
{\scriptsize  \hskip 2mm$0$}

\vskip 4.5mm
{\scriptsize \hskip 5mm  $\rho_-$ }

\vskip 3mm
\caption{Standing wave $U_0$}  
  \label{Figure3}
\end{figure}

Recall that we defined $v^K(x_1)$ by \eqref{eq:baru} such that $v^K(0)=\rho_*$ and
$v^K_{x_1}(0)<0$.  Define $\hat v^K(x_1)$ from $U_0$ by
\begin{align}  \label{eq:2.hatv}
\hat v^K(x_1) = \left\{
\begin{aligned}
U_0(-\sqrt{K}x_1),& \quad x_1\in [0,m_1], \\
U_0(\sqrt{K} (x_1-h_2)),& \quad x_1\in [m_1,m_2], \\
U_0(\sqrt{K}(1-x_1)),& \quad x_1\in [m_2,1].
\end{aligned}
\right.
\end{align}
Note that $\hat v^K$ is continuous; recall $h_2=2m_1$ and $h_2+1= 2m_2$.
Then we have
\begin{align} \label{eq:1.vK}
&\| v^K - \hat v^K\|_{L^\infty(\T)} \le C K^{-1/4}, \\
&\| \partial_{x_1}v^K - \partial_{x_1} \hat v^K\|_{L^\infty(\T)} \le C K^{1/4},
\label{eq:1.DvK}
\end{align}
see Lemmas \ref{prop:3.6-CP} and  \ref{prop:3.7-CP}.
We defined $u^K(x) = v^K(x_1)$.  Then, by \eqref{eq:1.vK} and Lemma 
\ref{lem:decay-U0} below, we see
\begin{align}  \label{eq:2.7-A}
\lim_{K\to\infty} u^K(x) = \left\{
\begin{aligned}
\rho_-,& \quad x_1\in (0,h_2), \\
\rho_+,& \quad x_1\in (h_2,1).
\end{aligned}
\right.
\end{align}
Thus, we obtain two interfaces $\Ga_1= \{x\in \T^d; x_1=0\}$ and
$\Ga_2= \{x\in \T^d; x_1=h_2\}$ in $\T^d$ which separates
$\rho_\pm$-phases.  These are flat hyperplanes.  The width of the transition
layers at these interfaces is $O(K^{-1/2})$.

The following exponential decay property of the tail of $U_0(z)$ is well-known;
see Lemma 2.1 of \cite{AHM08}.

\begin{lem}  \label{lem:decay-U0}
There exist $C>0$ and $\la>0$ such that
\begin{align*}
& 0<\rho_+-U_0(z) \le C e^{-\la |z|}, \quad z>0, \\
& 0<U_0(z)-\rho_- \le C e^{-\la |z|}, \quad z<0, \\
& |\partial_z^j U_0(z)| \le C e^{-\la |z|}, \quad z\in \R, \; j=1,2.
\end{align*}
\end{lem}

\subsection{Relation to the particle systems}  \label{sec:1.4}

The SPDE of the type \eqref{eq:2.3-1} arises to describe the fluctuation field 
in the Glauber-Kawasaki dynamics.  The Glauber dynamics governs the
creation and annihilation of particles with the speed-up factor $K$ in time,
while the Kawasaki dynamics determines the motion of particles as
interacting random walks with hard-core exclusion rule and a diffusive
time change factor $N^2$.
The function $f$, which is the ensemble average of the creation rate
minus the annihilation rate of the particles, and the noise $\dot{W}\equiv \dot{W}_G$
arise from the Glauber part, while the Laplacian $\De$ and the noise
$\dot{{\mathbb W}} \equiv \dot{{\mathbb W}}_K$ arise from the (simple) 
Kawasaki part; see Appendix \ref{Appendix:A}.  Note that, before taking
the limit $N\to\infty$,  the noise term $M^N(t)$ is a martingale and not Gaussian.

As it was shown in \cite{F18}, \cite{F23},
the scaled particle density field $\rho=\rho^{N,K}(t,x)$ of the 
Glauber-(simple) Kawasaki dynamics is close as $N\to\infty$, at the level of the 
law of large numbers, to the solution $u=u^K(t,x)$ of the Allen-Cahn equation
\eqref{eq:AC} or more generally \eqref{eq:A.13-Q}, which still contains 
a diverging factor $K=K(N)$.
For $u^K(t,x)$, it is known that the sharp interface limit 
$u^K(t,x) \to \Xi_{\Ga_t}(x)$ holds as $K\to\infty$, where the hypersurface
$\Ga_t$ in $\T^d$ evolves under the mean curvature flow (or more generally under
the anisotropic curvature flow)  and $\Xi_{\Ga}(x)$ is explained in Section \ref{sec:1}.
The function $u^K(x) = v^K(x_1)$ is a stationary solution of
\eqref{eq:AC} and we see that $\Ga_t= \Ga_1\cup\Ga_2$ with $\Ga_1$ and $\Ga_2$ 
defined in Section \ref{sec:2.4-Q}.
Under $u^K(x)$, the interface separating $\rho_\pm$ is flat and immobile.

Then we consider the fluctuation of $\rho^{N,K}$, which is
defined  as a step function in \eqref{eq:A.4-Q} instead of the empirical measure,
around the stationary solution
$u^K$ of \eqref{eq:AC} defined by 
\begin{equation}  \label{eq:2.5-Phi}
\Phi(t,x) = N^{d/2}(\rho^{N,K}(t,x)- u^K(t,x))
\end{equation}
and get the SPDE \eqref{eq:2.3-1}.  In fact, to compute the time derivative 
$\partial_t\Phi$, for the term $\partial_t\rho^{N,K}$ except the noise,
we replace ``the creation rate minus the annihilation rate'' in the Glauber part
by its ensemble average ``$Kf(\rho^{N,K})$'' at the given particle density
with the time change factor $K$ and have $\De\rho^{N,K}$ from the Kawasaki part
with the time change factor $N^2$.  Thus, using \eqref{eq:AC} for
$\partial_t u^K$, we get  for $\partial_t\Phi$ except noise
$$
\De\Phi + N^{d/2} \big( Kf(\rho^{N,K}) -Kf(u^K)\big)
= \De\Phi+ N^{d/2} K \big( f( u^K + N^{-d/2} \Phi)-f(u^K)\big)
$$
and the second term gives rise to the term $K F_n^N(u^K(x), \Phi(t,x))$ 
by the asymptotic expansion up to the $n$th order term.
The procedure of replacing by the ensemble average and taking the
expansion is called the Boltzmann-Gibbs principle.  

The functions
$g_1$ and $g_2$ appear as the ensemble averages under the local ergodicity of
the particle systems.  They have the form $g_1(u) = \sqrt{2\chi(u)}$ and
$g_2(u) = \sqrt{\lan c_0\ran(u)}$, $u\in (0,1)$; see Appendix \ref{Appendix:A}.

\section{Scaling for the SPDE \eqref{eq:2.3-1} near the interfaces}
\label{sec:2.2-0}

\subsection{Stretching around $\Ga_1$ and scaling to observe the fluctuation}

As \eqref{eq:2.hatv} and \eqref{eq:1.vK} suggest, in order
to observe the shape of the transition layer $U_0$, we first stretch the 
spatial variable to the normal direction to the interface $\Ga_1 \cup \Ga_2$ 
by $\sqrt{K}$; see \eqref{eq:3.1-Q}, \eqref{eq:barv-U_0} and 
\eqref{eq:2.stretch} below.  Then, we introduce a further scaling
defined by \eqref{eq:2.2-4}.  Under this scaling, we see that the noise term
behaves as $O(1)$ and, moreover, we will show that the density fluctuation field
is projected to
$U_0'$, which signifies the shift of the layer $U_0$.  In this way, one can grasp
the motion of the transition layer, and the fluctuation of the interface.

For stretching, especially focusing on $\Ga_1=\{x\in\T^d; x_1=0\}$ 
from $\Ga_1\cup\Ga_2$, 
we introduce a new variable $z:=\sqrt{K}x_1
\in \sqrt{K}\T = [-\sqrt{K}/2,\sqrt{K}/2)\subset \R$ regarding $x_1 
\in \T\equiv [-1/2, 1/2)$.  Accordingly, we define
\begin{align}  \label{eq:3.1-Q}
& \bar v^K(z) := v^K(z/\sqrt{K}), \quad z \in \sqrt{K}\T,
\intertext{and} 
&\bar u^K(z,\ux) := \bar v^K(z), \quad 
(z,\ux) \in \sqrt{K}\T\times\T^{d-1}.
\label{eq:3.2-Q}
\end{align}
Note that $\bar v^K$ satisfies the equation
\begin{align}  \label{eq:3.3-Q}
\partial_z^2 \bar v^K + f(\bar v^K)=0, \quad z \in \sqrt{K}\T,
\end{align}
and by \eqref{eq:2.hatv} and \eqref{eq:1.vK}, we have
\begin{align}  \label{eq:barv-U_0}
\lim_{K\to\infty} \|\bar v^K(z) - U_0(-z)\|_{L^\infty([-\sqrt{K}m_1, \sqrt{K}m_1])}
=0.
\end{align}

We observe the solution $\Phi(t,x)=\Phi^{N,K}(t,x)$ of the SPDE \eqref{eq:2.3-1}
on $\T^d$ under the above stretching to the normal direction to the interface
$\Ga_1$ by $\sqrt{K}$:
\begin{align}  \label{eq:2.stretch}
\widetilde{\Psi}(t,z,\ux) \equiv \widetilde{\Psi}^{N,K}(t,z,\ux) 
:= \Phi(t,z/\sqrt{K},\ux), \quad (z,\ux) 
\in \sqrt{K}\T\times\T^{d-1}.
\end{align}
Then we introduce a further scaling \eqref{eq:2.2-4} below
to make the noise term $O(1)$, so that one can observe
the non-trivial fluctuation in the limit.

The SPDE \eqref{eq:2.3-1} with $n=3$ and the space-time Gaussian white noises 
$\dot{{\mathbb W}}$ and $\dot{W}$ is singular even when $d=1$, critical when $d=2$
and supercritical when $d\ge 3$, due to the cubic nonlinear term.
Therefore, we first discuss the linear case, i.e., the SPDE \eqref{eq:2.3-1}
with $n=1$ in Proposition \ref{prop:2.2}.
For the nonlinear case with $n=2,3$, we need to take regularized
noises $\dot{{\mathbb W}}= \dot{{\mathbb W}}^{{\mathbb Q}}$ and
$\dot{W}= \dot{W}^{Q}$ with ${\mathbb Q}$ and $Q$ properly chosen depending
on $N$ and $K$.   In this case,
we only see how the nonlinear terms change under the scalings
\eqref{eq:2.stretch} and \eqref{eq:2.2-4}.  Our argument will be heuristic
for the nonlinear SPDE; see Pre-Proposition \ref{pre-prop:2.2}.
See Remark \ref{rem:3.1-Q} for making the argument rigorous with
regularized noises.

We will rigorously discuss the limit of the linear SPDE \eqref{eq:2.2-5} 
in Section \ref{sec:2.2} and then heuristically the nonlinear SPDE \eqref{eq:2.2-5-P} 
in Section \ref{sec:2.3}.

\begin{prop}   \label{prop:2.2}
Consider the SPDE \eqref{eq:2.3-1} with $n=1$ and independent space-time
Gaussian white noises $\dot{{\mathbb W}}$ and $\dot{W}$.
Then, $\widetilde{\Psi} = \widetilde{\Psi}^{N,K}$ 
defined by \eqref{eq:2.stretch}
satisfies the following SPDE in law
\begin{align} \label{eq:2.2-2} 
\partial_t \widetilde{\Psi}(t,z,\ux) =(- & K \mathcal{A}^K_z + \De_{\ux})
\widetilde{\Psi}(t,z,\ux)  
 \\
&+ K^{3/4}\partial_z \big(g_1(\bar v^K(z))\dot{W}^1(t,z,\ux)\big) + K^{1/4}
g_1(\bar v^K(z)) \nabla_{\ux} \cdot \dot{{\mathbb W}}^2(t,z,\ux)  \notag \\
& + K^{3/4} g_2(\bar v^K(z)) \dot{W}(t,z,\ux),  
\notag
\end{align}
for $(z,\ux) \in \sqrt{K}\T \times \T^{d-1}$, where
\begin{align}\label{eq:2.2-3} 
&\mathcal{A}^K \equiv \mathcal{A}_z^K
= - \partial_z^2 -f'(\bar v^K(z)), \quad z\in \sqrt{K}\T = [-\sqrt{K}/2,\sqrt{K}/2),
\intertext{and}  \label{eq:2.Lap-ux}
& \De_{\ux} = \sum_{i=2}^d \partial_{x_i}^2, \quad \ux=(x_2,\ldots,x_d) \in \T^{d-1},
\end{align}
and $\dot{W}^1(t,z,\ux)$, 
$\dot{{\mathbb W}}^2(t,z,\ux) := \{\dot{W}^i(t,z,\ux)\}_{i=2}^d$ and
$\dot{W}(t,z,\ux)$ are $d+1$ independent space-time
Gaussian white noises on $[0,\infty)\times \sqrt{K}\T\times\T^{d-1}$.

In view of the SPDE \eqref{eq:2.2-2}, we introduce a further scaling
\begin{align} \label{eq:2.2-4} 
\Psi(t,z,\ux) \equiv \Psi^{N,K}(t,z,\ux) 
:= K^{-3/4} \widetilde{\Psi}^{N,K}(t,z,\ux).
\end{align}
Then, $\Psi$ satisfies the SPDE 
\begin{align} \label{eq:2.2-5} 
\partial_t \Psi = & (- K \mathcal{A}_z^K + \De_{\ux})\Psi 
\\
& +  \partial_z \big( g_1(\bar v^K(z))\dot{W}^1(t,z,\ux)) + K^{-1/2}
g_1(\bar v^K(z))\nabla_{\ux} \cdot \dot{{\mathbb W}}^2(t,z,\ux)  \notag  \\
&  + g_2(\bar v^K(z)) \dot{W}(t,z,\ux),  \notag
\end{align}
for $(z,\ux)\in \sqrt{K}\T\times\T^{d-1}$.
Note that the noise terms for $\Psi$ behave as $O(1)$.
\end{prop}

Indeed, at least formally, by noting 
$\partial_{x_1} = \sqrt{K}\partial_z\; (x_1=z/\sqrt{K})$,
the noise terms in \eqref{eq:2.2-2} are obtained from those in
\eqref{eq:2.3-1} by the scaling law of the white noise:
$\dot{W}_\T(t,z/\sqrt{K}) \overset{\text{law}}{=} K^{1/4} \dot{W}_{\sqrt{K}\T}(t,z)$,
where $\dot{W}_{\T}$ and 
$\dot{W}_{\sqrt{K}\T}$ represent the space-time Gaussian white noises on
$\T$ and $\sqrt{K}\T$, respectively.  This scaling law can be shown by
multiplying a test function $H=H(z)$ on $\sqrt{K}\T$ as we will see in the following
proof for $I_3$ and $I_4$.

\begin{proof}
Let $p(t,x,y), t>0, x, y \in \T^d$ be the fundamental solution of $\partial_t-\De$,
that is, $p(t,x,y) = p_1(t,x_1,y_1) p_2(t,\ux,\uy)$, where $p_1(t,x_1,y_1)$ is the
fundamental solution of $\partial_t - \partial_{x_1}^2$ on $\T$ given by
\begin{align} \label{eq:p1-Q}
& p_1(t,x_1,y_1)= \frac1{\sqrt{4\pi t}} \sum_{\ell\in \Z} e^{-\frac{(x_1-y_1-\ell)^2}{4t}},
\quad x_1, \; y_1 \in \T = [-1/2,1/2),
\intertext{and}
& p_2(t,\ux,\uy)= \prod_{i=2}^d p_1(t,x_i,y_i), \quad \ux=(x_i)_{i=2}^d, \; \uy
=(y_i)_{i=2}^d  \in \T^{d-1}.
\label{eq:p2-Q}
\end{align}
Then, the solution of \eqref{eq:2.3-1} with $n=1$, which is in $\mathcal{D}'(\T^d)$,
is expressed in a mild form:
\begin{align*} 
\Phi(t,x) = &\int_{\T^d} p(t,x,y) \Phi(0,y) dy \\
& + K\int_0^t \int_{\T^d} p(t-s,x,y) f'(v^K(y_1)) \Phi(s,y) dsdy\\
& + \int_0^t \int_{\T^d} p(t-s,x,y)  
\nabla_y\cdot\big(  {g_1(v^K(y_1))}{{\mathbb W}}(dsdy)\big) \\
& + K^{1/2} \int_0^t \int_{\T^d} p(t-s,x,y) {g_2(v^K(y_1))}
{W}(dsdy).
\end{align*}
Therefore,
\begin{align} \label{eq:2.tildePsi}
& 
\widetilde{\Psi}(t,z,\ux)  =  \Phi(t,z/\sqrt{K},\ux) \\
= &  
\int_{\T^d} p_1(t,z/\sqrt{K},y_1) p_2(t,\ux,\uy) \widetilde{\Psi}(0,\sqrt{K}y_1,\uy) dy 
\notag \\
& + K\int_0^t \int_{\T^d} p_1(t-s,z/\sqrt{K},y_1) p_2(t-s,\ux,\uy) 
  f'(v^K(y_1)) \widetilde{\Psi}(s,\sqrt{K}y_1,\uy) dsdy   \notag\\
& + \int_0^t \int_{\T^d} p_1(t-s,z/\sqrt{K},y_1) p_2(t-s,\ux,\uy) 
\nabla_y\cdot \big( {g_1(v^K(y_1))}  {{\mathbb W}}(dsdy) \big) \notag \\
& + K^{1/2} \int_0^t \int_{\T^d} p_1(t-s,z/\sqrt{K},y_1) p_2(t-s,\ux,\uy)
 {g_2(v^K(y_1))} {W}(dsdy)  \notag \\
=: & I_1(t,z,\ux) + I_2(t,z,\ux) + I_3(t,z,\ux)+ I_4(t,z,\ux). \notag
\end{align}
Here, we note that the fundamental solution of $\partial_t - K \partial_{x_1}^2$
on $\sqrt{K}\T$ is given by
\begin{align*} 
p_1^K(t,x_1,y_1) := \frac1{\sqrt{K}} p_1(t,x_1/\sqrt{K},y_1/\sqrt{K})
= \frac1{\sqrt{4\pi Kt}} \sum_{\ell\in \Z}e^{-\frac{(x_1-y_1-\sqrt{K}\ell)^2}{4Kt}},
\end{align*}
for $x_1, y_1 \in \sqrt{K}\T = [-\sqrt{K}/2,\sqrt{K}/2)$.

For $I_1(t,z,\ux)$, changing $y_1= w/\sqrt{K}$ in the integral, 
\begin{align*} 
I_1(t,z,\ux)
& = \int_{\sqrt{K}\T\times\T^{d-1}} p_1(t,z/\sqrt{K},w/\sqrt{K}) p_2(t,\ux,\uy) 
\widetilde{\Psi}(0,w,\uy) \frac{dw}{\sqrt{K}}d\uy \\
& = \int_{\sqrt{K}\T\times\T^{d-1}} p_1^K(t,z,w) p_2(t,\ux,\uy) 
\widetilde{\Psi}(0,w,\uy) dwd\uy.
\end{align*}
In particular, $I_1(t,z,\ux)$ satisfies
\begin{align} \label{eq:2.2-I1}
& \partial_t I_1(t,z,\ux) = (K\partial_z^2+  \De_{\ux})I_1(t,z,\ux),\\
& I_1(0,z,\ux) = \widetilde{\Psi}(0,z,\ux).  \notag
\end{align}

For $I_2(t,z,\ux)$, changing $y_1= w/\sqrt{K}$ and recalling $v^K(w/\sqrt{K})
= \bar v^K(w)$,
\begin{align*} 
I_2(t,z,\ux)
& =K\int_0^t \int_{\sqrt{K}\T\times\T^{d-1}} 
p_1(t-s,z/\sqrt{K},w/\sqrt{K}) p_2(t-s,\ux,\uy) 
  f'(\bar v^K(w)) \widetilde{\Psi}(s,w,\uy) ds\frac{dw}{\sqrt{K}}d\uy\\
& =K\int_0^t \int_{\sqrt{K}\T\times\T^{d-1}} p_1^K(t-s,z,w) p_2(t-s,\ux,\uy) 
  f'(\bar v^K(w)) \widetilde{\Psi}(s,w,\uy) dsdwd\uy    
\end{align*}
In particular, $I_2(t,z,\ux)$ satisfies
\begin{align} \label{eq:2.2-I2}
& \partial_t I_2(t,z,\ux) = (K\partial_z^2+  \De_{\ux})I_2(t,z,\ux)
+ K f'(\bar v^K(z)) \widetilde{\Psi}(t,z,\ux),  \\
& I_2(0,z,\ux) = 0.  \notag
\end{align}

For $I_3(t,z,\ux)$, recalling that $\dot{{\mathbb W}} = (\dot{W}^1,
\dot{{\mathbb W}}^2)$ with $\dot{{\mathbb W}}^2= \{\dot{W}^i\}_{i=2}^d$
is a $d$-dimensional space-time Gaussian white noise, we will show that
\begin{align} \label{eq:2.2-I3-0}
I_3(t,z,\ux) \overset{\text{law}}{=}
 \int_0^t & \int_{\sqrt{K}\T\times\T^{d-1}} p_1^K (t-s,z,w) p_2(t-s,\ux,\uy) \\
 & \times   \Big( K^{3/4}\partial_w \big( {g_1(\bar v^K(w))}{W}^1\big) + K^{1/4}
  {g_1(\bar v^K(w))} \nabla_{\uy} \cdot {\mathbb W}^2\Big) (dsdwd\uy).  \notag
\end{align}
In particular, $I_3$ satisfies (in law sense)
\begin{align} \label{eq:2.2-I3}
& \partial_t I_3(t,z,\ux) = (K\partial_z^2+  \De_{\ux})I_3(t,z,\ux) \\
& \hskip 25mm
+ \Big( K^{3/4}\partial_z \big( {g_1(\bar v^K(z))}\dot{W}^1\big) 
+ K^{1/4}{g_1(\bar v^K(z))}
  \nabla_{\ux} \cdot \dot{{\mathbb W}}^2\Big)(t,z,\ux),   \notag   \\
& I_3(0,z,\ux) = 0.  \notag
\end{align}

To show \eqref{eq:2.2-I3-0} precisely, let us denote the right-hand side of
\eqref{eq:2.2-I3-0} by $\tilde I_3(t,z,\ux)$.  Noting that $I_3(t,\cdot,\cdot)
\in \mathcal{D}'(\sqrt{K}\T\times\T^{d-1})$, for a test function
$H=H(z,\ux)\in C^\infty(\sqrt{K}\T\times\T^{d-1})$, we denote
$\lan I_3(t),H\ran \equiv$  $_{ \mathcal{D}'}\! \lan I_3(t),H\ran_{ \mathcal{D}}$
and similar for $\tilde I_3$.  Since these are Gaussian variables with mean $0$,
to show \eqref{eq:2.2-I3-0}, it is sufficient to prove that the covariances 
are the same, that is,
\begin{align} \label{eq:3.16-Q}
& E[\lan I_3(t),H\ran^2] = E[\lan \tilde I_3(t),H\ran^2], \\
& E[\lan I_3(t_1) -I_3(t_2),H\ran \lan I_3(t_2),H\ran]
= E[\lan \tilde I_3(t_1) - \tilde I_3(t_2),H\ran \lan \tilde I_3(t_2),H\ran],
\label{eq:3.17-Q}
\end{align}
for $t\ge 0$ and $t_1>t_2\ge 0$.

First, let us prove \eqref{eq:3.16-Q}.
Recalling that $I_3(t)$ was defined in \eqref{eq:2.tildePsi},
\begin{align*} 
E[\lan I_3(t),H\ran^2]
& = \int_0^t ds\int_{\T^d} dy \bigg| 
\int_{\sqrt{K}\T\times\T^{d-1}} H(z,\ux) g_1(v^K(y_1))  \\
& \hskip 30mm \times \nabla_y \Big(
p_1(t-s,z/\sqrt{K},y_1)
p_2(t-s,\ux,\uy) \Big) dz d\ux \bigg|^2 \\
& = \int_0^t ds\int_{\T^d} g_1(v^K(y_1))^2 dy
\int_{(\sqrt{K}\T\times\T^{d-1})^2} \prod_{k=1}^2 H(z_k,\ux_k) dz_k d\ux_k
\\
& \hskip 10mm \times
\bigg[ \prod_{k=1}^2
(p_1)_{y_1}(t-s,z_k/\sqrt{K},y_1) p_2(t-s,\ux_k,\uy)\\
& \hskip 15mm
+ \sum_{i=2}^d \prod_{k=1}^2 p_1(t-s,z_k/\sqrt{K},y_1) 
(p_2)_{y_i} (t-s,\ux_k,\uy)\bigg]. 
\end{align*}

Here, in the first part in the integrand, by rewriting $y_1= w/\sqrt{K}$,
\begin{align}   \label{eq:p1-y1}
(p_1)_{y_1}(t-s,z/\sqrt{K},y_1) 
& = (p_1)_{y_1}(t-s,z/\sqrt{K},w/\sqrt{K}) \\
& = K (p^K_1)_w(t-s,z,w),  \notag
\end{align}
and, in the second part by rewriting $y_1= w/\sqrt{K}$ again,
\begin{align}   \label{eq:p1}
p_1(t-s,z/\sqrt{K},y_1) = p_1(t-s,z/\sqrt{K},w/\sqrt{K}) 
= \sqrt{K} p_1^K(t-s,z,w).
\end{align}
Therefore, noting $dsdy = ds \frac{dw}{\sqrt{K}}d\uy$ by changing $y_1=w/\sqrt{K}$,
the integral of the first part 
in the right hand side of $E[\lan I_3(t),H\ran^2]$  becomes
\begin{align*} 
& = \int_0^t ds\int_{\sqrt{K}\T\times\T^{d-1}} g_1(\bar v^K(w))^2
\frac{dw}{\sqrt{K}}d\uy
\int_{(\sqrt{K}\T\times\T^{d-1})^2} \prod_{k=1}^2 H(z_k,\ux_k) dz_k d\ux_k
\\
& \hskip 20mm \times
K^2 \prod_{k=1}^2
(p_1^K)_w (t-s,z_k,w)  p_2(t-s,\ux_k,\uy)\\
& = K^{3/2} \int_0^t ds\int_{\sqrt{K}\T\times\T^{d-1}} g_1(\bar v^K(w))^2 dwd\uy
  \\
& \hskip 20mm \times
\bigg( \int_{\sqrt{K}\T\times\T^{d-1}} H(z,\ux)
(p_1^K)_w(t-s,z,w) p_2(t-s,\ux,\uy)dzd\ux\bigg)^2  \\
& = E[\lan \tilde I_3^{(1)}(t),H\ran^2],
\end{align*}
where $\tilde I_3^{(1)}(t)$ is the first term of $\tilde I_3(t)$, that is, the stochastic
integral with respect to $W^1$ in the right-hand side of \eqref{eq:2.2-I3-0}.

The integral of the second part in $E[\lan I_3(t),H\ran^2]$ is similar and 
becomes, by changing $y_1= w/\sqrt{K}$,
\begin{align*} 
& = \int_0^t ds\int_{\sqrt{K}\T\times\T^{d-1}} {g_1(\bar v^K(w))}^2
\frac{dw}{\sqrt{K}}d\uy
\int_{(\sqrt{K}\T\times\T^{d-1})^2} \prod_{k=1}^2 H(z_k,\ux_k) dz_k d\ux_k
 \\
& \hskip 20mm \times
\sum_{i=2}^d K \prod_{k=1}^2 p_1^K(t-s,z_k,w) 
(p_2)_{y_i}(t-s,\ux_k,\uy) \\
& = K^{1/2} \int_0^t ds\int_{\sqrt{K}\T\times\T^{d-1}} g_1(\bar v^K(w))^2  dwd\uy
\\
& \hskip 20mm \times
\sum_{i=2}^d \bigg( \int_{\sqrt{K}\T\times\T^{d-1}} H(z,\ux) 
p_1^K(t-s,z,w)  (p_2)_{y_i}(t-s,\ux,\uy)dzd\ux\bigg)^2  \\
& = E[\lan \tilde I_3^{(2)}(t),H\ran^2],
\end{align*}
where $\tilde I_3^{(2)}(t)$ is the second term of $\tilde I_3(t)$, that is, 
the stochastic integral with respect to ${\mathbb W}^2$ in the right-hand side
of \eqref{eq:2.2-I3-0}.  From these computations, we obtain \eqref{eq:3.16-Q}
noting the independence of $\tilde I_3^{(1)}(t)$ and $\tilde I_3^{(2)}(t)$.

To show \eqref{eq:3.17-Q}, we decompose
\begin{align*} 
I_3(t_1)-I_3(t_2) = I^{(1)}+I^{(2)},
\end{align*}
where
\begin{align*} 
& I^{(1)} =  \int_{t_2}^{t_1} \int_{\T^d} p_1(t_1-s,z/\sqrt{K},y_1) \, p_2(t_1-s,\ux,\uy) 
\, \nabla_y\cdot \big( {g_1(v^K(y_1))}  {{\mathbb W}}(dsdy) \big), \\
& I^{(2)}=  \int_0^{t_2} \int_{\T^d}q(t_1,t_2,s; z/\sqrt{K},y_1; \ux,\uy) 
\, \nabla_y\cdot \big( {g_1(v^K(y_1))}  {{\mathbb W}}(dsdy) \big)
\end{align*}
and
\begin{align*}
q(t_1,t_2,s; z/\sqrt{K},y_1; \ux,\uy) = &
p_1(t_1-s,z/\sqrt{K},y_1) \, p_2(t_1-s,\ux,\uy) \\
& - p_1(t_2-s,z/\sqrt{K},y_1) \, p_2(t_2-s,\ux,\uy).
\end{align*}
Then, we have
\begin{align} \label{eq:IH=0}
 E[\lan I^{(1)},H\ran \lan I_3(t_2),H\ran] = 0,
\end{align}
since the time intervals of the stochastic integrals are disjoint.
For the other part, we have
\begin{align*}
E[\lan I^{(2)},& H\ran \lan I_3(t_2),H\ran]
 = \int_0^{t_2} ds\int_{\T^d} g_1(v^K(y_1))^2 dy \\
& \times
\int_{\sqrt{K}\T\times\T^{d-1}} H(z_1,\ux_1)
 \nabla_y q(t_1,t_2,s; z/\sqrt{K},y_1; \ux,\uy) \, dz_1 d\ux_1  \\
&  \cdot
\int_{\sqrt{K}\T\times\T^{d-1}} H(z_2,\ux_2)
 \nabla_y \big(p_1(t_2-s,z/\sqrt{K},y_1)
p_2(t_2-s,\ux,\uy) \big) \, dz_2 d\ux_2.
\end{align*}
Then, the computation is similar to the above. Using only the
scaling properties \eqref{eq:p1-y1}, \eqref{eq:p1} of $p_1$ and 
the change of variables $y_1= w/\sqrt{K}$, we obtain the corresponding
expectation
\begin{align*}
E[\lan \tilde I^{(2)},H\ran \lan \tilde I_3(t_2),H\ran]
\end{align*}
in terms of $\tilde I_3(t)$, where $\tilde I^{(2)}$ is defined similarly
to the above by taking the integral on $[0,t_2]$ from $\tilde I_3(t_1) - \tilde I_3(t_2)$.
This together with \eqref{eq:IH=0} shows \eqref{eq:3.17-Q} and thus, 
\eqref{eq:2.2-I3-0} is also shown.

For $I_4(t,z,\ux)$, we will show that
\begin{align} \label{eq:2-I5}
I_4(t,z,\ux) \overset{\text{law}}{=}
 K^{3/4} \int_0^t \int_{\sqrt{K}\T\times\T^{d-1}} p_1^K(t-s&,z,w) p_2(t-s,\ux,\uy)\\
& \times  {g_2(\bar v^K(w))} {W}(dsdwd\uy).  \notag
\end{align}
In particular, it satisfies
\begin{align} \label{eq:2.2-I4}
& \partial_t I_4(t,z,\ux) = (K\partial_z^2+  \De_{\ux})I_4(t,z,\ux)
+ K^{3/4} {g_2(\bar v^K(z))} \dot{W}(t,z,\ux), \\
& I_4(0,z,\ux) = 0.  \notag
\end{align}
In fact, this is seen again by computing the covariances.
First, changing $y_1= w/\sqrt{K}$, we have
\begin{align*} 
E[\lan I_4(t),H\ran^2]
& = K \int_0^t ds \int_{\T^d} dy\Big(\int_{\sqrt{K}\T\times\T^{d-1}} H(z,\ux) \\
& \hskip 20mm \times
p_1(t-s,z/\sqrt{K},y_1) p_2(t-s,\ux,\uy)
 {g_2(v^K(y_1))} dzd\ux \Big)^2 \\
& = K^{3/2} \int_0^t ds \int_{\sqrt{K}\T\times\T^{d-1}} dwd\uy
\Big(\int_{\sqrt{K}\T\times\T^{d-1}} H(z,\ux) \\
& \hskip 20mm \times
p_1^K(t-s,z,w) p_2(t-s,\ux,\uy)
 {g_2(\bar v^K(w))} dzd\ux \Big)^2  \\
 & = E[\lan \tilde I_4(t),H\ran^2],
\end{align*}
where $\tilde I_4(t)$ is the right-hand side of \eqref{eq:2-I5}.
This shows \eqref{eq:3.16-Q} for $I_4$ and $\tilde I_4$ in place of
$I_3$ and $\tilde I_3$ .
We can similarly show \eqref{eq:3.17-Q} for $I_4$ and $\tilde I_4$,
and obtain \eqref{eq:2-I5}.

Taking the sum of \eqref{eq:2.2-I1}, \eqref{eq:2.2-I2}, 
\eqref{eq:2.2-I3} and \eqref{eq:2.2-I4}, we obtain the SPDE \eqref{eq:2.2-2}.

The derivation of the SPDE \eqref{eq:2.2-5} is immediate from \eqref{eq:2.2-2}
noting $\widetilde{\Psi}= K^{3/4}\Psi$.
\end{proof}

We now consider the nonlinear SPDE \eqref{eq:2.3-1} with $n=3$.
In particular, we want to see how the nonlinear terms change under
the scalings \eqref{eq:2.stretch} and \eqref{eq:2.2-4}.  The following statement
is made with the space-time Gaussian white noises and, as we have noted,
the argument is at the heuristic level.  See Remark \ref{rem:3.1-Q} 
for the case with Gaussian regularized noises with covariance kernels 
$\{Q^i\}_{i=1}^{d+1}$.

\begin{pre-prop}   \label{pre-prop:2.2}
Consider the SPDE \eqref{eq:2.3-1} with $n=3$.
Then, $\widetilde{\Psi} = \widetilde{\Psi}^{N,K}$ 
defined by \eqref{eq:2.stretch} satisfies the following SPDE in law
\begin{align} \label{eq:2.2-2-P} 
\partial_t \widetilde{\Psi}(t,z,\ux) = & (- K \mathcal{A}^K_z + \De_{\ux})
\widetilde{\Psi}(t,z,\ux)  
+ KN^{-d/2} \tfrac12 f''(\bar v^K(z)) \widetilde{\Psi}(t,z,\ux)^2  \\
& + KN^{-d}\tfrac16 f'''(\bar v^K(z)) \widetilde{\Psi}(t,z, \ux)^3 \notag \\
&+   K^{3/4}\partial_z \big( g_1(\bar v^K(z))\dot{W}^1(t,z,\ux)\big) + K^{1/4}
g_1(\bar v^K(z))\nabla_{\ux} \cdot \dot{{\mathbb W}}^2(t,z,\ux)  \notag \\
& + K^{3/4} g_2(\bar v^K(z)) \dot{W}(t,z,\ux), 
\notag
\end{align}
for $ (z,\ux) \in \sqrt{K}\T \times \T^{d-1}$.
Then, $\Psi$ defined by \eqref{eq:2.2-4}  satisfies the SPDE 
\begin{align} \label{eq:2.2-5-P} 
\partial_t \Psi = & (- K \mathcal{A}_z^K + \De_{\ux})\Psi 
+ K^{7/4} N^{-d/2} \tfrac12 f''(\bar v^K(z)) \Psi(t,z,\ux)^2 \\
& + K^{5/2}N^{-d}\tfrac16 f'''(\bar v^K(z)) \Psi(t,z, \ux)^3  \notag
\\
& +  \partial_z \big( g_1(\bar v^K(z)) \dot{W}^1(t,z,\ux)\big) + K^{-1/2}
g_1(\bar v^K(z)) \nabla_{\ux} \cdot \dot{{\mathbb W}}^2(t,z,\ux)  \notag  \\
&  + g_2(\bar v^K(z)) \dot{W}(t,z,\ux),  \notag
\end{align}
for $(z,\ux)\in \sqrt{K}\T\times\T^{d-1}$.
\end{pre-prop}

\begin{proof}
At least if $\Phi$ is a usual function (i.e.\ if the noises are good),
the solution of \eqref{eq:2.3-1} with $n=3$ is expressed in a mild form:
\begin{align*} 
\Phi(t,x) = &\int_{\T^d} p(t,x,y) \Phi(0,y) dy \\
& + K\int_0^t \int_{\T^d} p(t-s,x,y) f'(v^K(y_1)) \Phi(s,y) dsdy\\
& + KN^{-d/2} \int_0^t \int_{\T^d} p(t-s,x,y) 
\tfrac12 f''(v^K(y_1)) \Phi(t,y)^2 dsdy \\
& + KN^{-d} \int_0^t \int_{\T^d} p(t-s,x,y) 
\tfrac16 f'''(v^K(y_1)) \Phi(t,y)^3 dsdy \\
& + \int_0^t \int_{\T^d} p(t-s,x,y)  
\nabla_y\cdot \big( {g_1(v^K(y_1))} {{\mathbb W}}(dsdy) \big) \\
& + K^{1/2} \int_0^t \int_{\T^d} p(t-s,x,y) {g_2(v^K(y_1))}
{W}(dsdy).
\end{align*}
Therefore,
\begin{align*} 
& 
\widetilde{\Psi}(t,z,\ux)  =  \Phi(t,z/\sqrt{K},\ux) \\
= &  
\int_{\T^d} p_1(t,z/\sqrt{K},y_1) p_2(t,\ux,\uy) \widetilde{\Psi}(0,\sqrt{K}y_1,\uy) dy 
\notag \\
& + K\int_0^t \int_{\T^d} p_1(t-s,z/\sqrt{K},y_1) p_2(t-s,\ux,\uy) 
  f'(v^K(y_1)) \widetilde{\Psi}(s,\sqrt{K}y_1,\uy) dsdy   \notag\\
& + KN^{-d/2} \int_0^t \int_{\T^d} p_1(t-s,z/\sqrt{K},y_1) p_2(t-s,\ux,\uy) 
\tfrac12 f''(v^K(y_1)) \widetilde{\Psi}(t,\sqrt{K}y_1,\uy)^2 dsdy \notag \\  
& + KN^{-d} \int_0^t \int_{\T^d} p_1(t-s,z/\sqrt{K},y_1) p_2(t-s,\ux,\uy) 
\tfrac16 f'''(v^K(y_1)) \widetilde{\Psi}(t,\sqrt{K}y_1,\uy)^3 dsdy \notag \\  
& + \int_0^t \int_{\T^d} p_1(t-s,z/\sqrt{K},y_1) p_2(t-s,\ux,\uy) 
\nabla_y\cdot\big( {g_1(v^K(y_1))} {{\mathbb W}}(dsdy) \big)  \notag \\
& + K^{1/2} \int_0^t \int_{\T^d} p_1(t-s,z/\sqrt{K},y_1) p_2(t-s,\ux,\uy)
 {g_2(v^K(y_1))} {W}(dsdy)  \notag \\
=: & I_1(t,z,\ux) + I_2(t,z,\ux) + I_5(t,z,\ux) + I_3(t,z,\ux)+ I_4(t,z,\ux). \notag
\end{align*}
Here $I_5$ is the sum of the third and fourth terms.

Then, as \eqref{eq:2.2-I2} for the term $I_2$ in the proof of Proposition \ref{prop:2.2},
$I_5(t,z,\ux)$ satisfies
\begin{align} \label{eq:2.2-I3-00-P}
& \partial_t I_5(t,z,\ux) = (K\partial_z^2+  \De_{\ux})I_5(t,z,\ux)
+ KN^{-d/2} \tfrac12 f''(\bar v^K(z)) \widetilde{\Psi}(t,z,\ux)^2  \\
& \hskip 40mm
+ KN^{-d} \tfrac16 f'''(\bar v^K(z)) \widetilde{\Psi}(t,z,\ux)^3, \notag \\
& I_5(0,z,\ux) = 0.  \notag
\end{align}
Therefore, combining with Proposition \ref{prop:2.2}, we obtain 
\eqref{eq:2.2-2-P} and then \eqref{eq:2.2-5-P} noting 
$\widetilde{\Psi}= K^{3/4}\Psi$.
\end{proof}

\begin{rem}  \label{rem:3.1-Q}
Let $\{\dot{W}^{Q^i}\}_{i=1}^{d+1}$ be the independent Gaussian 
regularized noises on $[0,\infty)\times \T^d$ with covariance
kernels $Q^i(x,y), x,y\in \T^d\times \T^d$, respectively; recall Section
\ref{sec:2.3-Q}.  Then, $\widetilde{\Psi} = \widetilde{\Psi}^{N,K}$ 
defined by \eqref{eq:2.stretch} from the SPDE \eqref{eq:2.3-1} with $n=3$
and these regularized noises satisfies the following SPDE in law
\begin{align} \label{eq:2.2-2-P-Q} 
\partial_t \widetilde{\Psi}(t,z,\ux) =& (-  K \mathcal{A}^K_z + \De_{\ux})
\widetilde{\Psi}(t,z,\ux)  
+ KN^{-d/2} \tfrac12 f''(\bar v^K(z)) \widetilde{\Psi}(t,z,\ux)^2  \\
& + KN^{-d}\tfrac16 f'''(\bar v^K(z)) \widetilde{\Psi}(t,z, \ux)^3 \notag \\
&+  \sqrt{K} \partial_z \big( g_1(\bar v^K(z))\dot{W}^{Q^{1,K}}(t,z,\ux) \big)
+ g_1(\bar v^K(z))
\nabla_{\ux} \cdot \dot{{\mathbb W}}^{{\mathbb Q}^{2,K}}(t,z,\ux)  \notag \\
& + \sqrt{K} g_2(\bar v^K(z)) \dot{W}^{Q^{d+1,K}}(t,z,\ux), 
\notag
\end{align}
for $ (z,\ux) \in \sqrt{K}\T \times \T^{d-1}$.  Here, the covariance kernels 
$\{Q^{i,K}(\overline{x},\overline{y})\}_{i=1}^{d+1}$, $\overline{x},\overline{y}
\in \sqrt{K}\T\times\T^{d-1}$ of new noises are determined by
\begin{align} \label{eq:3-AAA} 
Q^{i,K}(\overline{x},\overline{y}) & := Q^i(S_K\overline{x},S_K\overline{y}),
\quad 1\le i \le d+1,
\end{align}
where $S_K$ is a mapping defined by $\sqrt{K}\T\times\T^{d-1}
\ni \overline{x}=(w,\ux) \mapsto (w/\sqrt{K},\ux) \in \T\times \T^{d-1}
\equiv \T^d$.  We denoted $\dot{{\mathbb W}}^{{\mathbb Q}^{2,K}}
= \{\dot{W}^{Q^{i,K}}\}_{i=2}^{d}$.

The case of the space-time Gaussian white noises, discussed in
Proposition \ref{prop:2.2} and Pre-Proposition \ref{pre-prop:2.2},
can be understood as
$Q^i(x,y) = \prod_{j=1}^d \de_0(x_j-y_j)$, $x=(x_j)_{j=1}^d, y=(y_j)_{j=1}^d$,
for all $1\le i \le d+1$.
Indeed, in this case, regarding $\de_0(w/\sqrt{K}) = \sqrt{K}\de_0(w)$,
\begin{align*}
Q^{i,K}(\overline{x},\overline{y}) 
= K^{1/2} \de_0(w-w')  \prod_{j=2}^d \de_0(x_j-y_j), \quad 1\le i \le d+1,
\end{align*}
where $\overline{x}=(w,\ux)$ and $\overline{y}=(w',\uy)$.  Thus, at the level of
noises, we have $\dot{W}^{Q^{i, K}} \overset{\text{law}}{=} K^{1/4} \dot{W}^i$
and this recovers the noises in the SPDEs \eqref{eq:2.2-2} and \eqref{eq:2.2-2-P}.
\end{rem}

\subsection{Linearized operator $\mathcal{A}^K$}

In our SPDEs \eqref{eq:2.2-5} and \eqref{eq:2.2-5-P} for $\Psi$ on 
$\sqrt{K}\T\times \T^{d-1}$,  the Sturm-Liouville
operator $\mathcal{A}^K$ defined by \eqref{eq:2.2-3} appears. 
This is a linearized operator of the stationary Allen-Cahn equation 
\eqref{eq:3.3-Q} around $\bar v^K$ with negative sign.  Since we have a large
parameter $K$ in front of $\mathcal{A}^K$, we need to study its spectral
property.  In fact, Carr and Pego \cite{CP} studied in detail the property of 
the linearized operator of the stationary Allen-Cahn equation \eqref{eq:baru} 
multiplied by $-K^{-1}$:
$$
L^K= - \big(K^{-1} \partial_{x_1}^2+ f'(v^K(x_1))\big)
$$
on $[0,1]$ under the Neumann boundary condition.  Note that $v^K$ (see Figure
\ref{Figure2}) shifted by $m_1$, i.e., $v^K(x_1+m_1)$ satisfies the Neumann condition
at $x_1=0$ and $1$, 
so that one can apply the results of \cite{CP} in our setting on $\T$.

In particular, they proved that the eigenvalues $\{\la_1^K<\la_2^K< \cdots\}$
of $L^K$, being real and simple, satisfy
$$
0=\la_1^K<\la_2^K \le Ce^{-c\sqrt{K}}, \quad \la_3^K \ge \La_1,
$$
for every $K\ge K_0$ for some $K_0\ge 1$, and $C,c,\La_1>0$ are
uniform in $K$; see \eqref{eq:2-CP} in Appendix \ref{Appendix:B}, also 
for the non-negativity of eigenvalues in our setting.
We have two small eigenvalues due to the existence of
two interfaces.
The normalized eigenfunction corresponding to $\la_1^K=0$ is given by
$e^K(x_1) = v_{x_1}^K(x_1)/\|v_{x_1}^K\|_{L^2(\T)}$, as we easily see 
$L^K v_{x_1}^K=0$ by differentiating \eqref{eq:baru} in $x_1$.

Then, by the scaling relation between $L^K$ and our operator
$\mathcal{A}^K= - \partial_z^2 -f'(\bar v^K(z))$ (see Lemma \ref{lem:2.1-a-CP}),
the result for $L^K$ (see Lemma \ref{lem:1.1-CP}) implies
the following proposition for $\mathcal{A}^K$; see Proposition \ref{thm:3-6-0-CP}
in Appendix \ref{Appendix:B}.

\begin{prop}  \label{prop:3-6-0-CP}
For every $t>0$ and $G\in L^2(\R)\cap L^1(\R)$, we have
\begin{align*} 
\lim_{K\to\infty} \| e^{-tK\mathcal{A}^K} G(z) 
- \lan G, e\ran_{L^2(\R)} e(z)\|_{L^2(\sqrt{K}\T)} =0,
\end{align*}
with the first $G$ interpreted as $G|_{\sqrt{K}\T}$, where
\begin{align}  \label{eq:ez-Q}
e(z) := U_0'(-z) /\|U_0'\|_{L^2(\R)}, \quad z\in \R.
\end{align}
\end{prop}

In this proposition, the effect of the transition layer near $\Ga_2$ is asymptotically
negligible due to the tail property of $G$: $G\in L^2(\R)\cap L^1(\R)$.

For $H(z,\ux)\in L^2(\sqrt{K}\T\times \T^{d-1})$ or a restriction of
$H(z,\ux)\in L^2(\R\times \T^{d-1})$ on $\sqrt{K}\T\times \T^{d-1}$,
we define a semigroup
\begin{align} \label{eq:3.TtK}
T_t^KH(z,\ux) = e^{t(-K\mathcal{A}_z^K+\De_{\ux})}H(z,\ux),
\quad (z,\ux)\in \sqrt{K}\T\times \T^{d-1},
\end{align}
which is periodic in $z$ for $t>0$, and set
\begin{align} \label{eq:3.Ht}
H_t(z,\ux) = e(z) e^{t\De_{\ux}}\{ \lan H(\cdot,\ux),e\ran_{L^2(\R)}\},
\quad (z,\ux)\in \R\times \T^{d-1}.
\end{align}
Note that $H_t(z,\ux)$ is a product of functions of $z$ and $\ux$ so that the
variables $z$ and $\ux$ are separate.
The following corollary is immediate from Proposition \ref{prop:3-6-0-CP}.

\begin{cor}  \label{cor:3.3}   
For $t>0$ and $H\in L^2(\R\times\T^{d-1})$ such that $H(\cdot,\ux)\in L^1(\R)$
a.e.\ $\ux\in \T^{d-1}$,  we have
$$
\lim_{K\to\infty}\| T_t^KH-H_t \|_{L^2(\sqrt{K}\T\times\T^{d-1})} =0,
$$
with the first $H$ interpreted as $H|_{\sqrt{K}\T\times \T^{d-1}}$.
\end{cor}

\begin{proof}
Since $e^{t\De_{\ux}}$ is a contraction on $L^2(\T^{d-1})$, the square of
the norm in the statement of the corollary is bounded by
\begin{align*}
&\| e^{-tK\mathcal{A}^K}H(z,\ux) - e(z) \lan H(\cdot,\ux),e\ran_{L^2(\R)}
\|_{L^2(\sqrt{K}\T\times \T^{d-1})}^2  \\
& = \int_{\T^{d-1}} \| e^{-tK\mathcal{A}^K}H(\cdot,\ux) - e(\cdot) 
\lan H(\cdot,\ux),e\ran_{L^2(\R)}
\|_{L^2(\sqrt{K}\T)}^2 d\ux.
\end{align*}
However, since $H(\cdot,\ux)\in L^2(\R)\cap L^1(\R)$ a.e.\ $\ux\in \T^{d-1}$,
by Proposition \ref{prop:3-6-0-CP}, the integrand of the last integral
converges to $0$ for
a.e.\ $\ux\in \T^{d-1}$ as $K\to\infty$.  Moreover, $e^{-tK\mathcal{A}^K}$
is a contraction on $L^2(\sqrt{K}\T)$, which is seen from Lemma 
\ref{lem:2.1-a-CP} and the property of $L^K$, the integrand is bounded by
\begin{align*}
& 2\big(\| H(\cdot,\ux)\|_{L^2(\sqrt{K}\T)}^2+ 
\|e\|_{L^2(\sqrt{K}\T)}^2\lan H(\cdot,\ux),e\ran_{L^2(\R)}^2 \big) \\
& \le 2 \big(\| H(\cdot,\ux)\|_{L^2(\R)}^2
+ \|e\|_{L^2(\R)}^2 \lan H(\cdot,\ux),e\ran_{L^2(\R)}^2 \big),
\end{align*}
which is integrable on $\T^{d-1}$ and independent of $K$.
Therefore, one can apply Lebesgue's convergence theorem to show the
conclusion.
\end{proof}

We expect to have
\begin{align} \label{eq:TKH-partialz}
\lim_{K\to\infty}\|\partial_z(T_t^KH-H_t)\|_{L^2(\sqrt{K}\T\times\T^{d-1})} =0,
\end{align}
under a certain condition for $H$, but at the moment we can only prove the following
weaker estimate; see Section \ref{subsec:B.2}.  

\begin{lem}  \label{lem:3.4}
Assume that $H\in L^2(\R\times\T^{d-1})$ is twice differentiable in $z$ and
$\partial_z^2 H\in L^2(\R\times\T^{d-1})$.  Then, we have
\begin{align}  \label{eq:2.partial_zT}
\sup_{0\le t \le T} \sup_{K\ge 1} \|\partial_z T_t^KH\|_
{L^2(\sqrt{K}\T\times \T^{d-1})}  <\infty.
\end{align}
\end{lem}

The following lemma is well-known for the Sturm-Liouville operator
$\mathcal{A} = - \partial_z^2-f'(U_0(z))$ on the whole line $\R$, though
we will not use this lemma in this paper.

\begin{lem} \label{lem:2.3}  (cf.\ \cite{F95}, Lemma 3.1)
$\mathcal{A}$ is a symmetric and non-negative operator on $L^2(\R,dz)$.
The principal eigenvalue of $\mathcal{A}$ is $\la_1=0$.  It is simple and the
corresponding normalized eigenfunction is
$$
\tilde e(z) := U_0'(z)/ \|U_0'\|_{L^2(\R)}.
$$
The operator $\mathcal{A}$ has a spectral gap, that is, the next eigenvalue 
$\la_2>0$.
\end{lem}

As we noted above, 
differentiating $\partial_z^2 U_0(z) + f(U_0(z)) =0$ in $z$, we get
$\mathcal{A}U_0' =0$ and this implies that $U_0'$ is an eigenfunction of
$\mathcal{A}$ corresponding to $\la_1=0$.

\section{Gaussian fluctuation near the interface}  \label{sec:2.2}

Here, we study the limit as $K\to \infty$ for the linear SPDE \eqref{eq:2.2-5}
derived from the SPDE \eqref{eq:2.3-1} with $n=1$ under the scalings
\eqref{eq:2.stretch} and \eqref{eq:2.2-4},
that is, the SPDE for $\Psi = \Psi^K(t,z,\ux)$:
\begin{align} \label{eq:2.2-5-Q} 
\partial_t \Psi = & (- K \mathcal{A}_z^K + \De_{\ux})\Psi \\
& +  \partial_z \big( {g_1(\bar v^K(z))} \dot{W}^1(t,z,\ux)\big) + K^{-1/2}
{g_1(\bar v^K(z))} \nabla_{\ux} \cdot \dot{{\mathbb W}}^2(t,z,\ux)  \notag \\
&  + {g_2(\bar v^K(z))} \dot{W}(t,z,\ux),  \notag
\end{align}
for $(z,\ux)\in \sqrt{K}\T\times \T^{d-1}$, regarding $\sqrt{K}\T = [-\sqrt{K}/2,
\sqrt{K}/2) \subset \R$.
We assume that the initial value $\Psi(0)=\Psi(0,z,\ux)$ is given on 
$\R\times \T^{d-1}$ and satisfies $\Psi(0)\in L^2(\R\times \T^{d-1})$ and
$\Psi(0,\cdot,\ux)\in L^1(\R)$ a.e.\ $\ux\in \T^{d-1}$.
Compared to the nonlinear SPDE \eqref{eq:2.2-5-P}, one can say that we study
the case $K^{7/4} N^{-d/2} < \!\!< 1$, where the nonlinear terms are negligible
for large $N$.

\subsection{The case of $d\ge 2$}  \label{sec:4.1-A}

Let $T>0$ and let $\dot{\mathbb W} \equiv (\dot{W}^1, \dot{\mathbb W}^2)
= \{\dot{W}^i\}_{i=1}^d$ and $\dot{W}$, denoted by 
$\dot{W}^{d+1}$, be $(d+1)$ independent space-time Gaussian white noises 
on $[0,T]\times\R\times\T^{d-1}$ defined on 
a probability space $(\Om,\mathcal{F},P)$. 

The SPDE \eqref{eq:2.2-5-Q} is considered for $t\in [0,T]$ on this probability 
space with the noises $\dot{\mathbb W}$ and $\dot{W}$ given as above, and
restricted on $[0,T]\times\sqrt{K} \T \times\T^{d-1}$ embedding
$\sqrt{K}\T = [-\sqrt{K}/2,\sqrt{K}/2)$ in $\R$.  To state the theorem,
define the $\mathcal{D}'(\T^{d-1})$-valued process 
$\psi(t,\ux), t\in [0,T], \ux\in \T^{d-1}$ as
\begin{align}  \label{eq:4.1-Q}
\psi(t,\ux) = \psi_0(t,\ux)+\psi_1(t,\ux)+\psi_2(t,\ux),
\end{align}
where
\begin{align}  \notag
\psi_0(t,\ux) & = \int_{\R\times \T^{d-1}} e(w) p_2(t,\ux,\uy) \Psi(0,w,\uy) dwd\uy, \\
\psi_1(t,\ux) & = \int_0^t \int_{\R\times \T^{d-1}} e(w)
 p_2(t-s,\ux,\uy)
\partial_w \big( {g_1(\check U_0(w))}W^1(dsdwd\uy)\big),   
\label{eq:psi-0-2}  \\
\psi_2(t,\ux) & = \int_0^t \int_{\R\times \T^{d-1}} e(w)
{g_2(\check U_0(w))} p_2(t-s,\ux,\uy)W(dsdwd\uy),   \notag
\end{align}
$e(w)$ and $p_2(t,\ux,\uy)$ are defined by \eqref{eq:ez-Q} and \eqref{eq:p2-Q},
respectively, and $\check U_0(w) := U_0(-w)$.  Set
\begin{align}  \label{Psi-psi}
\Psi(t,z,\ux) = \psi(t,\ux) e(z).
\end{align}
Then, for the solution $\Psi^K(t)$ of the SPDE \eqref{eq:2.2-5-Q}
extended as $\Psi^K(t,z,\ux)=0$ for $z\in \R\setminus [-\sqrt{K}/2,\sqrt{K}/2)$,
we have the following theorem.

\begin{thm}  \label{prop:2.6}
For any $t\in (0,T]$ and any test function $H(z,\ux) \in L^2(\R\times \T^{d-1})$,
differentiable twice in $z$ and once in $\ux$,
such that $H(\cdot,\ux)\in L^1(\R)$ a.e.\ $\ux\in \T^{d-1}$,
$\partial_z H, \partial_z^2 H, \partial_{x_i} H \in L^2(\R\times \T^{d-1})$,
$2\le i \le d$, 
$\lan \Psi^K(t),H\ran \equiv\, _{\mathcal{D}'}\!\lan \Psi^K(t),H\ran_{\mathcal{D}}$ 
converges to $\lan \Psi(t),H\ran$ as $K\to\infty$
weakly in $L^2(\Om)$, that is,  $E[\lan \Psi^K(t)- \Psi(t),H\ran J] \to 0$  
holds for all $J\in L^2(\Om)$.
\end{thm}

\begin{rem}
Starting from the SPDE \eqref{eq:2.3-1} with $n=1$, this theorem discusses the 
fluctuation limit only around $\Ga_1$.  Similarly, one can obtain  
the fluctuation limit $\tilde\psi(t,\ux)$ around $\Ga_2$.  Then, one can prove the
independence of
$\psi(t,\ux)$ and $\tilde\psi(t,\ux)$ in the limit, since these are
asymptotically determined from the noises near $\Ga_1$ and $\Ga_2$,
respectively.
\end{rem}

We can write down the SPDE satisfied by $\psi(t,\ux)$ given in \eqref{eq:4.1-Q}
(in law sense).  Let us consider the SPDE on $\T^{d-1}$:
\begin{align}\label{eq:2.2-7-Q}
\partial_t \psi = \De_{\ux} \psi + c_* \dot{W}(t,\ux), \; \ux \in \T^{d-1},
\end{align}
with the initial value
\begin{align}\label{eq:4.4-Q}
\psi(0,\ux) = \int_\R \Psi(0,z,\ux) e(z)dz,
\end{align}
where $c_*$ is determined by 
\begin{align} \label{eq:2.2-c*}
c_* = \Big(\| \partial_w e\, {g_1(\check U_0)}\|_{L^2(\R)}^2
+ \| e\, {g_2(\check U_0)}\|_{L^2(\R)}^2\Big)^{1/2},
\end{align}
and $\dot{W}(t,\ux)$ is a new
space-time Gaussian white noise on $[0,T]\times \T^{d-1}$.
Then, we have

\begin{cor}  \label{cor:4.2}
For each $0< t_1<\cdots< t_n\le T$, the joint distribution of 
$\{\Psi^K(t_k,z,\ux)\}_{k=1}^n$ on $\big(\mathcal{D}'(\R\times \T^{d-1})\big)^n$
converges in law to that of
$\{\Psi(t_k,z,\ux) := \psi(t_k,\ux) e(z)\}_{k=1}^n$,
where $\psi(t,\ux)$ is the solution of the SPDE \eqref{eq:2.2-7-Q} with 
the initial value given by \eqref{eq:4.4-Q}.
\end{cor}

Before giving the proof of Theorem \ref{prop:2.6}, we make a formal argument
to derive the limit $\Psi(t)$ defined by \eqref{Psi-psi}.
The solution of \eqref{eq:2.2-5-Q}  takes values in 
$\mathcal{D}'(\sqrt{K}\T\times\T^{d-1})$ and it is 
given in a mild form:
\begin{align}  \label{eq:M-hd}
\Psi^K(t) = & e^{ -t K \mathcal{A}_z^K + t\De_{\ux}}  \Psi(0) 
+ \int_0^t  e^{ (t-s) (-K \mathcal{A}_w^K+\De_{\uy}) } 
 \Big( \partial_w \big(  {g_1(\bar v^K(w))}W^1(dsdwd\uy)\big) \\
& +K^{-1/2}{g_1(\bar v^K(w))}
\nabla_{\uy} \cdot {\mathbb W}^2(dsdwd\uy)
 + {g_2(\bar v^K(w))} W(dsdwd\uy)\Big).  \notag
\end{align}
Later, we will give another definition \eqref{eq:3.6-G} of the
solution of \eqref{eq:2.2-5-Q} in a weak sense.
By Proposition \ref{prop:3-6-0-CP}, as $K\to\infty$, $e^{ -t K \mathcal{A}_z^K}$
converges to the projection operator to $e(z)$ for $t>0$, that is, for $\Psi(z,\ux)\in 
L^2(\R\times\T^{d-1})$,
\begin{align}  \label{eq:2-proj}
e^{ -t K \mathcal{A}_z^K}\Psi(z,\ux) \to \lan \Psi(\cdot,\ux),e\ran_{L^2(\R)} e(z)
\end{align}
in $L^2(\sqrt{K}\T\times\T^{d-1})$ as $K\to\infty$ for $t>0$; see Corollary
\ref{cor:3.3} with $e^{t\De_{\ux}}$.
Also noting that the stochastic integral with $K^{-1/2}$ in \eqref{eq:M-hd}
is negligible in the limit, and from \eqref{eq:barv-U_0} for 
$\bar v^K(w)$,  we would have
$$
\Psi^K(t,z, \ux) \to (\psi_0(t,\ux)+\psi_1(t,\ux)+\psi_2(t,\ux)) e(z),
$$
where $\psi_0(t,\ux), \psi_1(t,\ux)$ and $\psi_2(t,\ux)$ are defined
in \eqref{eq:psi-0-2}.  This leads to Theorem \ref{prop:2.6}.

\begin{proof}[Proof of Theorem \ref{prop:2.6}]
Precisely, we give the meaning to \eqref{eq:2.2-5-Q}  in a weak sense:
For a test function $G=G(t,z,\ux)$ on $[0,T]\times \sqrt{K}\T\times \T^{d-1}$,
which is $C^1$ in $t$, $C^2$ in $(z,\ux)$ and periodic in $z$ for $t>0$,
\begin{align}  \label{eq:3.6-G}
\lan \Psi^K(t), G(t)\ran = & \lan \Psi(0), G(0)\ran 
+ \int_0^t \lan \Psi^K(s), (\partial_s-K\mathcal{A}_z^K+\De_{\ux}) G(s) \ran ds \\
& - \int_0^t  \int_{\sqrt{K}\T\times \T^{d-1}}  {g_1(\bar v^K(w))}\partial_w 
 G(s,w,\uy) W^1(dsdwd\uy) \notag  \\
& - K^{-1/2} \int_0^t \int_{\sqrt{K}\T\times \T^{d-1}} 
{g_1(\bar v^K(w))} \nabla_{\uy} G(s,w,\uy)
 \cdot {\mathbb W}^2(dsdwd\uy)   \notag  \\
& + \int_0^t  \int_{\sqrt{K}\T\times \T^{d-1}} 
 {g_2(\bar v^K(w))} G(s,w,\uy)W(dsdwd\uy),  \notag
\end{align}
where $\lan \cdot, \cdot\ran = _{\mathcal{D}'(\sqrt{K}\T\times \T^{d-1})}
\lan \cdot, \cdot\ran_{\mathcal{D}(\sqrt{K}\T\times \T^{d-1})}$.
For $t\in (0,T]$ and $H=H(z,\ux)\in L^2(\R\times\T^{d-1})$ restricting on
 $\sqrt{K}\T\times \T^{d-1}$,
take $G(s,z,\ux) = T_{t-s}^KH(z,\ux)$, 
$s\in [0,t]$ recalling \eqref{eq:3.TtK}.
Then, since $(\partial_s-K\mathcal{A}_z^K+\De_{\ux}) G(s)=0$, $s\in (0,t]$,
we obtain from \eqref{eq:3.6-G}
\begin{align*}
\lan \Psi^K(t), H \ran = & \lan \Psi(0), T_t^KH \ran \\
& - \int_0^t  \int_{\sqrt{K}\T\times \T^{d-1}}  {g_1(\bar v^K(w))}\partial_w 
T_{t-s}^KH(w,\uy) W^1(dsdwd\uy)   \\
& - K^{-1/2} \int_0^t \int_{\sqrt{K}\T\times \T^{d-1}} 
 {g_1(\bar v^K(w))} \nabla_{\uy} T_{t-s}^KH(w,\uy)
 \cdot {\mathbb W}^2(dsdwd\uy)     \\
& + \int_0^t \int_{\sqrt{K}\T\times \T^{d-1}} 
 {g_2(\bar v^K(w))} T_{t-s}^KH(w,\uy)W(dsdwd\uy) \\
=: & I_0^K(t) -I_1^K(t) -I_2^K(t) + I_3^K(t).
\end{align*}
Recalling $H_t(z,\ux)$ defined by \eqref{eq:3.Ht} and $\check U_0$ given below
\eqref{eq:psi-0-2}, we set
\begin{align*}
& I_1(t)
 =  \int_0^t  \int_{\R\times \T^{d-1}}  {g_1(\check U_0(w))} \partial_w 
 H_{t-s}(w,\uy) W^1(dsdwd\uy),   \\
& I_3(t)= \int_0^t  \int_{\R\times \T^{d-1}} 
 {g_2(\check U_0(w))}  H_{t-s}(w,\uy) W(dsdwd\uy).
\end{align*}

For $I_1^K(t)$, we will show that it converges to $I_1(t)$ weakly in $L^2(\Om)$ 
as $K\to\infty$.  To this end, consider the operator $\Phi$ defined by the
stochastic integral:
$$
\Phi(F):= \int_0^t \int_{\R\times\T^{d-1}} F(s,w,\uy) W^1(dsdwd\uy)
$$
for $F\in {\mathbb L}^2 := L^2([0,t]\times\R\times\T^{d-1})$.  The operator
$\Phi$ is linear and strongly continuous from ${\mathbb L}^2$ to $L^2(\Om)$
by It\^o isometry:
$$
E[\Phi(F)^2] =  \int_0^t \int_{\R\times\T^{d-1}} F^2(s,w,\uy) dsdwd\uy.
$$
Therefore, $\Phi$ is weakly continuous, i.e., if $F^K\to F$ weakly in 
${\mathbb L}^2$, then $\Phi(F^K)\to \Phi(F)$ weakly in $L^2(\Om)$; see,
e.g., (5.6) in \cite{KR}.  Thus, to show the weak convergence of $I_1^K(t)$
to $I_1(t)$ in $L^2(\Om)$, denoting the integrand of $I_1^K(t)$ by $F^K$
(extending it as $0$ for $w\notin \R\setminus \sqrt{K}\T$) and that of $I_1(t)$
by $F$, it is sufficient to prove that $\lan F^K,J\ran_{{\mathbb L}^2}$
converges to $\lan F,J\ran_{{\mathbb L}^2}$ for any $J\in {\mathbb L}^2$.

First, take $J=J(s,w,\uy) \in {\mathbb L}^2$ such that it has a compact support and
$\partial_w J(s,w,\uy) \in {\mathbb L}^2$.  Then, 
\begin{align*}
\lan F^K,J\ran_{{\mathbb L}^2}
& = \int_0^t ds \int_{\sqrt{K}\T\times \T^{d-1}}  
J(s,w,\uy)  {g_1(\bar v^K(w))}  \partial_w 
 T_{t-s}^KH(w,\uy) dwd\uy  \\
& = \int_0^t ds \int_{\sqrt{K}\T\times \T^{d-1}}  
J(s,w,\uy)  {g_1(\check U_0(w))}  \partial_w 
 T_{t-s}^KH(w,\uy) dwd\uy + o(1)
\end{align*}
as $K\to\infty$, by \eqref{eq:barv-U_0},  $g_1\in C^\infty([\rho_-,\rho_+])$
and noting Lemma \ref{lem:3.4}.  Here, recalling the compact support property 
of $J$, by the integration by parts, the above is rewritten as
\begin{align*}
= - \int_0^t ds \int_{\sqrt{K}\T\times \T^{d-1}}  
\partial_w \Big( J(s,w,\uy) 
 {g_1(\check U_0(w))} \Big) T_{t-s}^KH(w,\uy)  dwd\uy  + o(1),
\end{align*}
for large enough $K$.  However, as $K\to\infty$, this converges to
\begin{align*}
- \int_0^t ds \int_{\R\times \T^{d-1}}  
\partial_w \Big( J(s,w,\uy) 
 {g_1(\check U_0(w))} \Big) H_{t-s}(w,\uy)  dwd\uy 
 = \lan F,J\ran_{{\mathbb L}^2},
\end{align*}
first by Corollary \ref{cor:3.3} for each $s\in (0,t]$ in the integrand 
and then by applying Lebesgue's convergence theorem noting that 
$\|T_{t-s}^KH\|_{L^2(\sqrt{K}\T\times \T^{d-1})}
\le \|H\|_{L^2(\sqrt{K}\T\times \T^{d-1})}$ is bounded in $K$ and $s$. 
Therefore, we obtain $\lan F^K,J\ran_{{\mathbb L}^2} \to
\lan F,J\ran_{{\mathbb L}^2}$ as $K\to\infty$ for $J$ in a dense set of ${\mathbb L}^2$. 
For any $J\in {\mathbb L}^2$ and any $\e>0$, one can find $J_0$ from this dense set
such that $\|J-J_0\|_{{\mathbb L}^2} <\e$.  Then, decomposing
\begin{align*}
\lan F^K - F ,J\ran_{{\mathbb L}^2}
= \lan F^K - F ,J_0\ran_{{\mathbb L}^2}
+ \lan F^K - F ,J-J_0\ran_{{\mathbb L}^2},
\end{align*}
the first term tends to $0$ as $K\to\infty$, while the second term is bounded by
$$
\le \e \|F^K-F\|_{{\mathbb L}^2}
\le \e \{\|F^K \|_{{\mathbb L}^2} + \|F\|_{{\mathbb L}^2}\},
$$
and by Lemma \ref{lem:3.4}, we see that $\|F^K\|_{{\mathbb L}^2}$ is bounded
in $K$ if $\partial_z^2H \in L^2(\R\times \T^{d-1})$.
Thus, we have shown that $I_1^K(t)$ converges to $I_1(t)$ weakly in $L^2(\Om)$.

For $I_2^K(t)$, by It\^o isometry and changing the variable $t-s$ to $s$, we have
\begin{align*}
E[I_2^K(t)^2]  = K^{-1} \int_0^t ds
\Big\|  {g_1(\bar v^K(z))} \nabla_{\ux}T_s^K H(z,\ux)
   \Big\|_{L^2(\sqrt{K}\T\times \T^{d-1}; \R^{d-1})}^2.
\end{align*}
However, since $\partial_{x_i}$ and $T_{s}^K$ commute with each other
for $2\le i\le d$, we have
\begin{align*}
\|\partial_{x_i}T_s^K H \|_{L^2(\sqrt{K}\T\times \T^{d-1})}
& = \|T_s^K \partial_{x_i} H \|_{L^2(\sqrt{K}\T\times \T^{d-1})} \\
& \le \|\partial_{x_i} H \|_{L^2(\sqrt{K}\T\times \T^{d-1})}
\le \|\partial_{x_i} H \|_{L^2(\R\times \T^{d-1})} < \infty,
\end{align*}
from the condition $\partial_{x_i} H \in L^2(\R\times \T^{d-1})$, $2\le i \le d$.
Therefore, $I_2^K(t)\to 0$ strongly in $L^2(\Om)$ as $K\to\infty$..

For $I_3^K(t)$, we have a strong convergence in $L^2(\Om)$.  Indeed, again
by It\^o isometry and changing the variable $t-s$ to $s$, we have
\begin{align*}
&E[|I_3^K(t)- I_3(t)|^2]\\
& = \int_0^t ds
\Big\|  {g_2(\bar v^K(z))} T_s^K H(z,\ux)
 - {g_2(\check U_0(z))}  H_s(z,\ux)
   \Big\|_{L^2(\sqrt{K}\T\times \T^{d-1})}^2   \\
& +\int_0^t ds \int_{\{|z|\ge \sqrt{K}/2\}\times \T^{d-1}}
\Big| {g_2(\check U_0(z))} H_s(z,\ux) \Big|^2 dz d\ux.
\end{align*}
Then, to show that the right-hand side converges to $0$ as $K\to\infty$,
one can use Corollary \ref{cor:3.3}, \eqref{eq:barv-U_0}, 
Lemma \ref{lem:decay-U0}, $g_2\in C^\infty([\rho_-,\rho_+])$ for each 
$s\in (0,t]$ and, as above, we may recall the contraction property of $T_s^K$
to apply Lebesgue's convergence theorem.  For the second integral,
we use Lemma \ref{lem:decay-U0} which shows the exponential decay
property of $e(z)$ for large $|z|$.

Finally for $I_0^K(t)$, by Corollary \ref{cor:3.3} and recalling
$\Psi(0) \in L^2(\R\times \T^{d-1})$ such that
$\Psi(0,\cdot,\ux)\in L^1(\R)$ a.e.\ $\ux\in \T^{d-1}$,
for $t\in (0,T]$, it converges as
$K\to\infty$ to
\begin{align*}
\lan \Psi(0), H_t\ran_{L^2(\R\times \T^{d-1})}
& = \big\lan \Psi(0, z,\ux), e(z) e^{t\De_{\ux}} \{\lan H(\cdot,\ux),e\ran_{L^2(\R)}\}
\big\ran_{L^2(\R\times \T^{d-1})} \\
& = \big\lan e \,e^{t\De_{\ux}} \lan\Psi(0, \cdot,\ux),e\ran_{L^2(\R)}, H
\big\ran_{L^2(\R\times \T^{d-1})} \\
& = \lan e \,\psi_0(t), H\ran_{L^2(\R\times \T^{d-1})},
\end{align*}
where $\psi_0(t)$ is defined in \eqref{eq:psi-0-2}.

Summarizing all these and noting $I_1(t) = - \lan e \,\psi_1(t), H
\ran_{L^2(\R\times \T^{d-1})}$
and $I_3(t) = \lan e \,\psi_2(t), H\ran_{L^2(\R\times \T^{d-1})}$, we have shown that
$\lan \Psi^K(t),H\ran$ converges to $\lan \Psi(t),H\ran$ as $K\to\infty$
weakly in $L^2(\Om)$.
\end{proof}

\begin{rem}
If \eqref{eq:TKH-partialz} is shown, we have the strong convergence for $I_1^K(t)$.
Indeed,
\begin{align*}
&E[|I_1^K(t)- I_1(t)|^2]\\
& = \int_0^t ds
\Big\| {g_1(\bar v^K(z))}\partial_z
 T_s^K H(z,\ux)
 - {g_1(\check U_0(z))}  \partial_z H_s(z,\ux)
   \Big\|_{L^2(\sqrt{K}\T\times \T^{d-1})}^2   \\
& +\int_0^t ds \int_{\{|z|\ge \sqrt{K}/2\}\times \T^{d-1}}
\Big| {g_1(\check U_0(z))}\partial_z
 H_s(z,\ux) \Big|^2 dz d\ux.
\end{align*}
To estimate the right-hand side, we may proceed similar to $I_3^K(t)$ using
\eqref{eq:TKH-partialz}  and \eqref{eq:1.DvK} together.
\end{rem}

To show Corollary \ref{cor:4.2}, we prepare a lemma.

\begin{lem} \label{lem:4.3}
Let $W(t,z,\ux)$  be the Gaussian white noise process (i.e.\ the time
integral of $\dot{W}$) on $\R\times \T^{d-1}$.
For $\fa_1=\fa_1(z)\in L^2(\R)$, define $W(t,\ux;\fa_1)$ by 
$\int_\R W(t,z,\ux) \fa_1(z)dz$, or more precisely, for $\fa_2=\fa_2(\ux)$,
$$
W(t,\fa_2;\fa_1) := \lan W(t),\fa_1\otimes\fa_2\ran_{\R\times\T^{d-1}}.
$$
Then, $W(t,\ux;\fa_1)$ is the Gaussian white noise process on
$\T^{d-1}$ multiplied by $\|\fa_1\|_{L^2(\R)}$.
\end{lem}

\begin{proof}
This is obvious, since $W(t,\fa_2;\fa_1)$ is the Brownian motion multiplied
by $\|\fa_1\otimes\fa_2\|_{L^2(\R\times \T^{d-1})}$
and $\|\fa_1\otimes\fa_2\|_{L^2(\R\times
\T^{d-1})} =\|\fa_1\|_{L^2(\R)} \|\fa_2\|_{L^2(\T^{d-1})}$.
\end{proof}

\begin{proof}[Proof of Corollary \ref{cor:4.2}]
By the observation in Lemma \ref{lem:4.3}, we see that
\begin{align*}
\psi_1(t,\ux) & = \int_0^t \int_{\T^{d-1}} p_2(t-s,\ux,\uy)
 W^1(dsd\uy; - \partial_w e \, {g_1(\check U_0)} ) \\
& \overset{\text{law}}{=}
 \|\partial_w e \, {g_1(\check U_0)}\|_{L^2(\R)} \int_0^t \int_{\T^{d-1}} 
 p_2(t-s,\ux,\uy)
 \widetilde{W}^1(dsd\uy),\\
\psi_2(t,\ux) & = \int_0^t \int_{\T^{d-1}} p_2(t-s,\ux,\uy)W(dsd\uy; 
e{g_2(\check U_0)}) \\
& \overset{\text{law}}{=}
 \|e\, {g_2(\check U_0)}\|_{L^2(\R)} \int_0^t \int_{\T^{d-1}} p_2(t-s,\ux,\uy)
 \widetilde{W}(dsd\uy),
\end{align*}
where $\widetilde{W}^1$ and $\widetilde{W}$ are independent Gaussian
white noise processes on $\T^{d-1}$.  Setting
$$
\psi(t,\ux) = \psi_1(t,\ux) + \psi_2(t,\ux),
$$
it satisfies the SPDE
\begin{align}\label{eq:2.2-7}
\partial_t \psi = \De_{\ux} \psi + c_* \dot{W}(t,\ux), \; \ux \in \T^{d-1};
\quad \psi(0,\ux) = 0,
\end{align}
where $c_*$ is determined by \eqref{eq:2.2-c*}
and $\dot{W}(t,\ux)$ is a new
space-time Gaussian white noise on $[0,T]\times \T^{d-1}$.
Therefore, noting that $\psi_0$ satisfies $\partial_t \psi_0 = \De_{\ux} \psi_0$
with $\psi_0(0,\ux)$ given by \eqref{eq:4.4-Q},
we see that $\psi(t,\ux)$ defined by \eqref{eq:4.1-Q} satisfies 
the SPDE \eqref{eq:2.2-7-Q} in law.

However, Theorem \ref{prop:2.6} implies that $\{\lan \Psi^K(t_k),H\ran\}_{k=1}^n$
converges in law to $\{\lan \Psi(t_k),H\ran\}_{k=1}^n$, where $\Psi(t)$ is
defined by \eqref{Psi-psi} with $\psi(t,\ux)$ in \eqref{eq:4.1-Q}.
Thus we obtain the conclusion.
\end{proof}

\subsection{Interpretation of Theorem \ref{prop:2.6} when $d=2$}
\label{sec:4.2-Inter}

Note that, when $d=2$, the SPDE \eqref{eq:2.2-7-Q} on $\T$ 
is classical, and the solution takes values in  continuous functions on $\T$.

We stated $\Phi= N^{d/2}(\rho^{N,K}-u^K)$ in \eqref{eq:2.5-Phi}, in other words,
the particle density is determined from $\Phi$ and thus from $\Psi$ as
\begin{align*}
\rho^{N,K}(t,x) & = u^K(x) + N^{-d/2} \Phi(t,x) \\
& = u^K(x) + N^{-d/2} K^{3/4} \Psi^K(t,\sqrt{K}x_1,\ux),
\end{align*}
by \eqref{eq:2.stretch} and \eqref{eq:2.2-4}.  However, by Theorem \ref{prop:2.6},
\begin{align*}
\Psi^K(t,z,\ux) &= \psi(t,\ux) e(z) + R^K(t,z,\ux) \\
&= \fa(t,\ux) U_0'(-z) + R^K(t,z,\ux),
\end{align*}
where $\fa(t,\ux) =\psi(t,\ux)/\|U_0'\|_{L^2(\R)}$ and the error term
$R^K(t)$ tends to $0$
as $K\to\infty$ for $t>0$ in the sense that $\lim_{K\to\infty} \lan R^K(t),H\ran=0$
weakly in $L^2(\Om)$.  Thus, noting
$u^K(x) = v^K(x_1) = U_0(-\sqrt{K}x_1) + O(K^{-1/4})$ by \eqref{eq:1.vK},
by Taylor expansion of $U_0$ at $-\sqrt{K}x_1$, we have
\begin{align*}
\rho^{N,K}(t,x) 
& = U_0(-\sqrt{K}x_1) + O(K^{-1/4}) \\
& \hskip 10mm 
+ N^{-d/2} K^{3/4} 
  \big( \fa(t,\ux) U_0'(-\sqrt{K}x_1) + R^K(t,\sqrt{K}x_1,\ux) \big) \\
& = U_0\big(-\sqrt{K} (x_1- N^{-d/2} K^{1/4} \fa(t,\ux))  \big)
+ O\big(\|U_0''\|_{L^\infty} (N^{-d/2} K^{3/4} \|\fa(t)\|_{L^\infty})^2\big)  \\
& \hskip 30mm
+ O(K^{-1/4})+ N^{-d/2} K^{3/4} \cdot R^K(t,\sqrt{K}x_1,\ux) \\
& \sim  \left\{
\begin{aligned}
\rho_+& \quad \text{if } x_1> N^{-d/2} K^{1/4} \fa(t,\ux)\\
\rho_-& \quad \text{if } x_1< N^{-d/2} K^{1/4} \fa(t,\ux),
\end{aligned}
\right.
\end{align*}
if $N^{-d/2}K^{3/4} \ll 1$, i.e., $K \ll N^{2d/3}$, and $N, K \to \infty$. 

This shows that, in a finer scale, the interface is described as 
\begin{align*}
\Ga_1^{N,K} = \{x=(x_1,\ux);
x_1= N^{-d/2} K^{1/4} \fa(t,\ux)\}.
\end{align*}
As $N\to\infty$ such that $K\ll N^{2d}$,
$\Ga_1^{N,K}$ converges to $\Ga_1= \{(0,\ux); \ux\in \T^{d-1}\}$ which is 
immobile.  This corresponds to the law of large numbers.
But, by enlarging the spatial scale to the
normal direction to $\Ga_1$ by $N^{d/2}K^{-1/4}$, one can observe the fluctuation
of the interface, which is described by the height function $\fa(t,\ux)$
at the point $(0,\ux)$ on $\Ga_1$.  The above calculation also implies that
the particle density fluctuates
keeping the shape $U_0$ of the transition layer at the stretched level.

One can say that the fluctuation of the interface to the normal
direction to $\Ga_1$ behaves as
$$
N^{-d/2} K^{1/4}\psi(t,\ux)/\|U_0'\|_{L^2(\R)}.
$$
The constant $\|U_0'\|_{L^2(\R)}^2$ is sometimes called the surface tension.

\subsection{The case of $d=1$}

When $d=1$, the SPDE \eqref{eq:2.2-5-Q} for $\Psi= \Psi^K(t,z), 
z\in \sqrt{K}\T$, is written as
\begin{align} \label{eq:2.2-6} 
\partial_t \Psi(t,z) = - K \mathcal{A}^K \Psi(t,z) 
+  \partial_z \big(  {g_1(\bar v^K(z))}\dot{W}^1(t,z) \big)
 + {g_2(\bar v^K(z))} \dot{W}(t,z),
\end{align}
with the initial value $\Psi(0)\in L^2(\R)$ restricted on $\sqrt{K}\T$.
The argument in Section \ref{sec:4.1-A} works as it is, by dropping the
variable $\ux \in \T^{d-1}$.  In particular, as $K\to\infty$, $\Psi^K(t,z)$
converges to 
$$
(\psi_0+\psi_1(t)+\psi_2(t)) e(z),
$$
for $t>0$, where
\begin{align*}
\psi_0 & =\lan\Psi(0),e \ran_{L^2(\R)}, \\
\psi_1(t) & = \int_0^t \int_\R e(w) 
\partial_w \big( {g_1(\check U_0(w))}W^1(dsdw)\big), \\
\psi_2(t) & = \int_0^t \int_\R e(w){g_2(\check U_0(w))} W(dsdw).
\end{align*}
However, $\psi_1(t)$ and $\psi_2(t)$ are independent Brownian motions 
with covariances
\begin{align*}
E[\psi_1(t)^2] & = t \| \partial_w e(w)\, {g_1(\check U_0(w))})\|_{L^2(\R)}^2, \\
E[\psi_2(t)^2] & = t \| e(w)\, {g_2(\check U_0(w))}\|_{L^2(\R)}^2.
\end{align*}
Therefore, we see
\begin{align*}
\psi_1(t) + \psi_2(t) \overset{\text{law}}{=} c_* B_t,
\end{align*}
where $B_t$ is a one-dimensional Brownian motion and $c_*$ is the same
constant as in \eqref{eq:2.2-c*}.  We state the following theorem at the level
of Corollary \ref{cor:4.2}.

\begin{thm}  \label{prop:2.4}
If $d=1$, any finite-dimensional distribution of $\Psi^K(t,z)$ in $t>0$ converges 
to that of 
$$
\Psi(t,z) = \Big(\lan \Psi(0),e\ran_{L^2(\R)}+c_* B_t\Big) e(z)
$$ 
as $K\to\infty$.
\end{thm}

When $d=1$, $\De_{\ux}$ does not appear and $\dot{W}(t,\ux)$ in \eqref{eq:2.2-7}
can be interpreted as $\dot{B}_t$, and therefore we obtain Theorem
\ref{prop:2.4} in a sense directly from Corollary \ref{cor:4.2}.

As we discussed in Section \ref{sec:4.2-Inter},
Theorem \ref{prop:2.6} for $d=2$ and similarly Theorem \ref{prop:2.4} for $d=1$
imply that the interface fluctuates according to the solution
of the SPDE \eqref{eq:2.2-7-Q} when $d=2$ or as a Brownian motion
multiplied by $c_*$ when $d=1$, and the fluctuation occurs as the
spatial shift preserving the shape $U_0(z)$ of the transition layer of the interface.

\section{Nonlinear fluctuation near the interface}  \label{sec:2.3}

Recall that we started with the SPDE \eqref{eq:2.3-1} (with $n=1, 2$ or $3$) on 
$\T^d$ for $\Phi(t,x)$, then obtained the linear SPDE \eqref{eq:2.2-5} 
and the nonlinear SPDE \eqref{eq:2.2-5-P} on $\sqrt{K}\T\times \T^{d-1}$
for $\Psi(t,z,\ux) \equiv \Psi^{N,K}(t,z,\ux):= K^{-3/4} \Phi(t,z/\sqrt{K},\ux)$ 
under the scalings \eqref{eq:2.stretch} and \eqref{eq:2.2-4}. 

This section studies the SPDE \eqref{eq:2.2-5-P}, whose derivation was
heuristic, and therefore the argument in this section is also heuristic.
Let us choose $K$ as $K^{7/4}N^{-d/2}=1$ in \eqref{eq:2.2-5-P}, i.e., $K=N^{2d/7}$.
Then, since $K^{5/2}N^{-d} = N^{-2d/7} \to 0$ and $K^{-1/2}\to 0$,
dropping two terms with these factors, we would have the SPDE
\begin{align} \label{eq:2.3-5} 
\partial_t \Psi = & (- K \mathcal{A}_z^K + \De_{\ux})\Psi 
+ \tfrac12 f''(\bar v^K(z)) \Psi(t,z,\ux)^2  \\
& +  \partial_z\big({g_1(\bar v^K(z))} \dot{W}^1(t,z,\ux)\big)
 + {g_2(\bar v^K(z))} \dot{W}(t,z,\ux),  \notag
\end{align}
for $(z,\ux)\in \sqrt{K}\T \times \T$.
As in \eqref{eq:M-hd}, this may be rewritten in a mild form
\begin{align}  \label{eq:M-hd-KPZ}
\Psi(t) =& e^{ -t K \mathcal{A}_z^K + t\De_{\ux}}  \Psi(0) 
+ \int_0^t  e^{ (t-s) (-K \mathcal{A}_w^K+\De_{\uy}) } \\
& \hskip 10mm 
\times  \Big(  \partial_w \big( {g_1(\bar v^K(w))}W^1(dsdwd\uy) \big)
 + {g_2(\bar v^K(w))} W(dsdwd\uy)\Big)    \notag \\
& + \int_0^t  e^{ (t-s) (-K \mathcal{A}_w^K+\De_{\uy}) } 
 \tfrac12 f''(\bar v^K(w)) \Psi(s,w,\uy)^2 ds.  \notag 
\end{align}
The last term is new and we may consider the limit of this term only.

Since all terms in the right-hand side of \eqref{eq:M-hd-KPZ} contain
$e^{ -t K \mathcal{A}_z}$ or $e^{ -(t-s) K \mathcal{A}_w}$, by 
Proposition \ref{prop:3-6-0-CP} or Corollary \ref{cor:3.3},
we expect that $\Psi=\Psi^{N,K}(t)$ is projected to $e(z)$ in $z$-variable,
behaving as $\lan \Psi^{N,K}(t,\cdot,\ux), e\ran_{L^2(\R)} e(z)$ and converging to
$\psi(t,\ux) e(z)$ for some $\psi(t,\ux)$.  In particular, we expect that
the last term is projected to
\begin{align*}  
e(z) \int_0^t  \Big\lan e^{ (t-s)\De_{\uy}} 
 \tfrac12 f''(\bar v^K(w)) \Psi^{N,K}(s,w,\uy)^2, e(w)\Big\ran_{L^2(\R)} ds 
\end{align*}
and, by \eqref{eq:barv-U_0}, it behaves as
\begin{align*}  
e(z) & \int_0^t  \Big\lan e^{ (t-s)\De_{\uy}} 
 \tfrac12 f''(\check U_0(w)) \psi(s,\uy)^2 e(w)^2, e(w)\Big\ran_{L^2(\R)} ds \\
& = c_2 \, e(z) \int_0^t  \int_\R p_2(t-s,\ux,\uy) \psi(s,\uy)^2 d\uy  ds,
\end{align*}
where
\begin{align*}
c_2& = \tfrac12 \int_\R f''(\check U_0(w)) e(w)^3 dw  \\
& = \frac1{2\|U_0'\|_{L^2(\R)}^3} \int_\R f''(U_0(w)) U_0'(w)^3 dw,
\end{align*}
by the change of variable $w\mapsto -w$.  Therefore, since we are assuming that
the left-hand side of \eqref{eq:M-hd-KPZ} converges to $\psi(t,\ux)e(z)$, 
we would obtain the nonlinear equation for $\psi(t,\ux)$:
\begin{align}\label{KPZ}
\partial_t \psi = \De_{\ux} \psi + c_* \dot{W}(t,\ux) + c_2\psi^2, 
\quad \ux \in \T^{d-1},
\end{align}
where $c_*$ and $\dot{W}(t,\ux)$ are the same as in \eqref{eq:2.2-7}.

\begin{lem}
Indeed, we have $c_2=0$.  
\end{lem}

\begin{proof}
The above integral is rewritten as
\begin{align*}
\int_\R f''(U_0(w)) U_0'(w)^3 dw
& = \int_\R \big(f'(U_0(w)) \big)' U_0'(w)^2 dw\\
& = - 2\int_\R f'(U_0(w))  U_0'(w)U_0''(w) dw \\
& = 2\int_\R f'(U_0(w)) f(U_0(w)) U_0'(w)dw \\
& = \int_\R \big(f^2(U_0(w))\big)' dw\\
& = f^2(\rho_+)-f^2(\rho_-) =0
\end{align*}
since $U_0''(z)= -f(U_0(w))$ by \eqref{eq:stand-wave} and $f(\rho_\pm)=0$.
\end{proof}

\begin{rem}
We have shown $c_2=0$.  However, if $c_2\not=0$, the SPDE \eqref{KPZ} 
considered on $\R$ (i.e.\ the case of $d=2$)
has an instantaneous blow-up of the solution everywhere; 
see \cite{FKN}.  Note that the drift term $\psi^2$ satisfies the Osgood's
condition $\int_1^\infty 1/\psi^2 d\psi <\infty$.  They assumed the
non-decreasing property of the drift term in \cite{FKN}.  
Though $\psi^2$ is not non-decreasing, 
if $c_2>0$, we can use the comparison argument for SPDEs noting that
$\psi^21_{\{\psi\ge 0\}} \le \psi^2$.  If $c_2<0$, we may consider $-\psi$.
\end{rem}

Since $c_2=0$, the nonlinear equation \eqref{KPZ}  becomes linear.  Then,
under the next order scaling $K^{5/2}N^{-d}=1$, i.e.\ $K=N^{2d/5}$, 
keeping the third order term in the SPDE \eqref{eq:2.2-5-P}, similar to the above
and neglecting the second order term, we expect to have the SPDE
\begin{align}  \label{eq:2-cube}
\partial_t \psi = \De_{\ux} \psi + c_* \dot{W}(t,\ux) + c_3\psi^3, \quad 
\ux \in \T^{d-1},
\end{align}
where
\begin{align*}
c_3& = \tfrac16 \int_\R f'''(\check U_0(w)) e(w)^4 dw \\
& = \frac1{6\|U_0'\|_{L^2(\R)}^4} \int_\R f'''(U_0(w)) U_0'(w)^4 dw.
\end{align*}

\begin{lem}  \label{lem:5.2}
For simplicity, assume $f''>0$ (i.e., convex) on $(\rho_-,\rho_*)$ and
$f''<0$ (i.e., concave) on $(\rho_*,\rho_+)$.  Then, $c_3<0$.  In particular,
when $d=2$, the SPDE \eqref{eq:2-cube} has a global-in-time solution.
\end{lem}

The SPDE \eqref{eq:2-cube} becomes an SDE when $d=1$. 
It is well-posed in the classical sense when $d=2$ and singular when 
$d=3, 4$, the same equation as the dynamic $P(\phi)$-model \eqref{eq:a}
(with $\t=0$).
 
\begin{proof}[Proof of Lemma \ref{lem:5.2}]
The integral in $c_3$ is given by
\begin{align*}
\int_\R f'''(U_0(w)) U_0'(w)^4 dw
& = \int_\R \big(f''(U_0(w)) \big)' U_0'(w)^3 dw\\
& = - 3\int_\R f''(U_0(w))  U_0'(w)^2 U_0''(w) dw \\
& = 3\int_\R f''(U_0(w)) f(U_0(w)) U_0'(w)^2 dw <0,
\end{align*}
since $f''(\rho)f(\rho)< 0$ for $\rho \in (\rho_-,\rho_+)\setminus \{\rho_*\}$.
We again used $U_0''(z)= -f(U_0(w))$.
\end{proof}

\section{Fluctuation away from the interface}  \label{sec:awayfromInterface}

Let us study the fluctuation of the density field away from the interface
$\Ga=\Ga_1\cup\Ga_2$, that is,  the behavior of $\Phi(t,x)$, which is governed by 
the SPDE \eqref{eq:2.3-1}, for $x\in \T^d$ such that $\text{dist}(x,\Ga)\ge \de>0$
(or $\text{dist}(x,\Ga) >\!\!> 1/\sqrt{K}$ is sufficient).  
Again, the argument is heuristic.  For such $x$, it holds 
$f'(u^K(x)) \cong f'(\rho_\pm)<0$  by \eqref{eq:2.7-A} and
writing $c := - f'(\rho_+)$ or  $-f'(\rho_-)>0$, the second term of \eqref{eq:2.3-1} 
behaves as $-cK\Phi(t,x)$ by neglecting the higher order terms.
So, we would have the SPDE for $\Phi=\Phi^K$
\begin{align}   \label{eq:6.1-Q}
\partial_t \Phi(t,x) = \De\Phi(t,x) -cK \Phi(t,x)
+ c_1 \nabla\cdot \dot{{\mathbb W}}(t,x)
 + c_2\sqrt{K} \dot{W}(t,x),   \quad x\in \T^d,
\end{align}
where $c_1 = g_1(\rho_+)$ (or $g_1(\rho_-)$) and
$c_2 = g_2(\rho_+)$ (or $g_2(\rho_-)$) again by \eqref{eq:2.7-A}.
We may consider the SPDE \eqref{eq:6.1-Q} on $\R^d$, since we consider
the equation by localizing around the point away from the interface.

To study the limit, we scale down $\Phi=\Phi^K$ as
$$
\Psi^K(t,x) = K^{-1/4} \Phi^K(t,x).
$$
Then, $\Psi= \Psi^K (\in \mathcal{D}'(\T^d))$ satisfies the SPDE
\begin{align}    \label{eq:6.2-Q}
\partial_t \Psi(t,x) = \De\Psi(t,x) -cK \Psi(t,x)
+ {c_1}K^{-1/4} \nabla\cdot \dot{{\mathbb W}}(t,x)
 + c_2K^{1/4} \dot{W}(t,x), \; x \in \T^d.
\end{align}
This SPDE is linear in $\Psi$ and the solution $\Psi$ is Gaussian.

\begin{prop}  \label{prop:2.1}
Suppose $d=1$ and $\sup_{K\ge 1} \|\Psi^K(0)\|_{L^2(\T)}<\infty$.  Then,
the solution $\Psi^K$ of the SPDE \eqref{eq:6.2-Q} has a decomposition
$\Psi^K=\Psi_1^K + \Psi_2^K$ with $\Psi_1^K(t) \in\mathcal{D}'(\T)$ 
and $\Psi_2^K(t) \in C(\T)$ a.s.   As $K\to\infty$, $\Psi_1^K(t)$ converges to $0$
in the sense that $\lan \Psi_1^K(t),\fa\ran \to 0$ in $L^2(\Om)$ for
any test function $\fa\in C^1(\T)$ and $t>0$, while  $\Psi_2^K(t,x)$ 
converges to $\Psi_2(t,x)$ in law for $t>0$ and $x\in \T$.
For each $t>0$ and $x\in \T$, the limit $\Psi(t,x)$ is an $\R$-valued 
Gaussian random variable with mean $0$ and variance $\si^2$ given by
\begin{align}    \label{eq:6.3-Q}
\si^2 = \frac{c_2^2}{\sqrt{8\pi}}
\int_0^\infty \frac{e^{-2cu}}{\sqrt{u}}du.
\end{align}
Furthermore, if $t_1\not= t_2$ or $x_1\not= x_2$,
$\Psi(t_1,x_1)$ and $\Psi(t_2,x_2)$ are independent.
\end{prop}

The noise $K^{-1/4}{c_1} \nabla\cdot \dot{{\mathbb W}}(t,x)$
is smaller than the other and vanishes in the limit.  The effect of $\De$ is
also lost (except appearing in $\si^2$) in the limit, and this causes the
independence of the limit $\Psi(t,x)$.

\begin{proof}
For a while, we consider in a general dimension $d$.
By Duhamel's formula, $\Psi^K(t,x)$ is expressed as
\begin{align*} 
\Psi^K(t,x) = & e^{-cKt} e^{t\De}\Psi^K(0,x) \\
& + c_1K^{-1/4} \int_0^t \int_{\T^d} e^{-cK(t-s)} p(t-s,x,y) 
\nabla\cdot {{\mathbb W}}(dsdy)\\
& + {c_2}K^{1/4} \int_0^t \int_{\T^d} e^{-cK(t-s)} p(t-s,x,y) 
{W}(dsdy) \\
=: & I_0^K(t,x) +  I_1^K(t,x) +  \Psi_2^K(t,x),
\end{align*}
where $p(t,x,y)$ is the heat kernel on $\T^d$; recall the above of 
\eqref{eq:p1-Q}.  We set $\Psi_1^K(t,x): = I_0^K(t,x) +  I_1^K(t,x)$.

First, since $c>0$ and $\|e^{t\De}\Psi^K(0)\|_{L^2(\T^d)}  \le
\|\Psi^K(0)\|_{L^2(\T^d)}$ is bounded in $K$, we have
$I_0^K(t) \to 0$ in $L^2(\T^d)$ as $K\to\infty$ for $t>0$.

Next, we show $\lan I_1^K(t),\fa\ran \to 0$ in $L^2(\Om)$ as $K\to\infty$
for every $\fa=\fa(x)\in C^1(\T^d)$ and $t\ge 0$.  Indeed,
\begin{align*} 
E[ \lan I_1^K(t,\cdot),\fa\ran^2]
& = c_1^2 K^{-1/2} \int_0^t \int_{\T^d} e^{-2cK(t-s)} 
\Big| \int_{\T^d} \nabla_y p(t-s,x,y) \fa(x)dx \Big|^2 dsdy\\
& = c_1^2 K^{-1/2} \int_0^t  e^{-2cKs} ds
\int_{\T^d\times\T^d} \nabla\fa(x_1)\cdot \nabla\fa(x_2) p(2s,x_1,x_2) dx_1dx_2\\
& \le C_\fa K^{-1/2} \int_0^t  e^{-2cKs} ds  
\le \frac{C_\fa}{2cK} K^{-1/2} \longrightarrow 0,
\quad K\to\infty,
\end{align*}
where $C_\fa= c_1^2 \|\nabla\fa\|_{L^\infty(\R^d)}\|\nabla\fa\|_{L^1(\R^d)}$.
In the above calculation, we used It\^o isometry to get the first line.
Then, we rewrote the spatial integral in $y$ and $x$ as
$\int_{\T^d} dy \Big| \int_{\T^d}  p(t-s,x,y) \nabla_x\fa(x)dx \Big|^2$
and got the second line by integrating first in $y$. The third line was obtained
by noting $\int_{\T^d}p(2s,x_1,x_2) dx_2=1$. 

We now assume $d=1$ for the calculation of $\Psi_2^K$ which is in 
the function space $C(\T)$ if $d=1$.  
Taking $0\le t_2 \le t_1, x_1, x_2 \in \T^1$, we compute the
covariance of $\Psi_2^K(t,x)$ as
\begin{align*} 
& E[ \Psi_2^K(t_1,x_1) \Psi_2^K(t_2,x_2)]  \\
& = c_2^2 K^{1/2} \int_0^{t_2} \int_{\T^d} e^{-cK(t_1-s)} p(t_1-s,x_1,y)
\cdot e^{-cK(t_2-s)} p(t_2-s,x_2,y) dsdy\\
& = c_2^2 K^{1/2} \int_0^{t_2} e^{-cK(t_1+t_2-2s)} p(t_1+t_2-2s,x_1,x_2)
ds\\
& = c_2^2 K^{1/2} \int_0^{t_2}  e^{-cK(t_1-t_2+2s)} p(t_1-t_2+2s,x_1,x_2)
ds.
\end{align*}
The last line follows by the change of variable $t_2-s$ to $s$.
If $t_2<t_1$, using a rough estimate
$p(t_1-t_2+2s,x_1,x_2) \le C/\sqrt{s}$ in the case of $d=1$, the above is bounded by
\begin{align*} 
C e^{-cK(t_1-t_2)} K^{1/2} \int_0^{t_2} \frac1{\sqrt{s}} e^{-2cKs}ds \to 0,
\quad K\to\infty.
\end{align*}
Note that, by the change of variable $Ks =u$,
$$
K^{1/2} \int_0^{t_2} \frac1{\sqrt{s}} e^{-2cKs}ds
= K^{1/2} \int_0^{Kt_2} \frac{\sqrt{K}}{\sqrt{u}} e^{-2cu}\frac{du}K
$$
and this is bounded in $K$.

If $t_1=t_2=t$, the above covariance is equal to
\begin{align*} 
& c_2^2 K^{1/2} \int_0^t  e^{-2cKs} p(2s,x_1,x_2) ds.
\end{align*}
However, by \eqref{eq:p1-Q}, this is rewritten and bounded by
\begin{align*} 
& = c_2^2 K^{1/2} \int_0^t  e^{-2cKs} \frac1{ \sqrt{8\pi s}} 
  \sum_{\ell\in \Z} e^{-\frac{(x_1-x_2-\ell)^2}{8s}} ds\\
& = \frac{c_2^2}{\sqrt{8\pi}} \sum_{\ell\in \Z} \int_0^{Kt} \frac{e^{-2cu}}{\sqrt{u}}
  e^{-\frac{K (x_1-x_2-\ell)^2}{8u}}du  \\
& \le \frac{c_2^2}{\sqrt{8\pi}} \sum_{\ell\in \Z} \int_0^\infty \frac{e^{-2cu}}{\sqrt{u}}
  e^{-\frac{K (x_1-x_2-\ell)^2}{8u}}du.
\end{align*}
If $x_1\not= x_2$, this tends to $0$ as $K\to\infty$  by Lebesgue's convergence
theorem.  If $x_1=x_2$, the above covariance is equal to
\begin{align*}
\frac{c_2^2}{\sqrt{8\pi}} \sum_{\ell\in \Z} \int_0^{Kt} \frac{e^{-2cu}}{\sqrt{u}}
  e^{-\frac{K \ell^2}{8u}}du.
\end{align*}
However the sum of the terms from $\ell\not=0$ vanish in the limit
 (as we saw above), 
and therefore, if $t>0$, the limit of the above is given by $\si^2$ in
\eqref{eq:6.3-Q}.  This shows the conclusion.
\end{proof}

\begin{rem}
Fixing any $x_0$ such that $\text{dist}(x_0,\Ga) \ge \de>0$,
we may study the behavior near $x_0$ by stretching around $x_0$:
$$
z/\sqrt{K} := x-x_0 \in \R^d.
$$
Note that we stretch in all directions, unlike in Section \ref{sec:2.2}.
Then, differently from Proposition \ref{prop:2.1}, the limit $\Psi(t,z)$ can have 
the dependence.
\end{rem}

\section{Unbalanced case}  \label{sec:2.4}

Let us consider the case that the balance condition is not satisfied, i.e.\
$\int_{\rho_-}^{\rho_+}f(u)du \not= 0$.  Then, the stationary equation 
\eqref{eq:baru} of the Allen-Cahn equation requires a modification, since
the stationary solution $u^K(t,x)$ has a moving front as explained below.  

Let us consider the traveling wave solution
$U_0(z),z\in \R$ with speed $c\in \R$:
\begin{align}  \label{eq:2.4-1}
& \partial_z^2U_0 + c \partial_z U_0(z) + f(U_0(z)) = 0, \quad z\in \R,\\
& U_0(\pm\infty) = \rho_\pm, \; U_0(0) = \rho_*.  \notag
\end{align}
This equation uniquely determines $c\in \R$ and an increasing solution $U_0$
except for translation.
The speed $c=0$ in the balanced case, while $c\not=0$ in the unbalanced case.
Note that $U(t,z) := U_0(z-ct)$ is a solution of
\begin{align*}
\partial_t U = \partial_z^2 U + f(U), \quad t\ge 0, \, z\in \R.
\end{align*}

In this section, for simplicity, we discuss on $\R\times \T^{d-1}$ instead 
of $\T^d$.  Then, the stationary solution $u^K(x)$ introduced in
Section \ref{sec:2.1} is replaced by
\begin{align*}
u^K(t,x) := U_0(\sqrt{K}x_1-cKt) = U_0(\sqrt{K}(x_1-c \sqrt{K}t)),
\end{align*}
for $x=(x_1,\ux) \in \R\times \T^{d-1}$.  In this expression,
$c\sqrt{K}$ represents the speed of the moving interface $\Ga_t:= \{x\in
\R\times \T^{d-1}; x_1= c\sqrt{K}t\}$.  Note that $x_1-c \sqrt{K}t$
describes the signed distance of $x$ from $\Ga_t$.  In \cite{FMST},
we introduced a shorter time scale $1/\sqrt{K}$, but here we keep the
original time scale so that $\Ga_t$ moves fast.  Then,
\begin{align*}
\partial_t u^K(t,x) & = -c K \partial_zU_0 = K(\partial_z^2 U_0+f(U_0))\\
& = \partial_{x_1}^2 u^K(t,x) + K f(u^K(t,x)) = \De u^K+ K f(u^K).
\end{align*}
This corresponds to the Allen-Cahn equation \eqref{eq:AC}, but considered
on $\R\times \T^{d-1}$.  The function $u^K(t,x)$ keeps the shape $U_0$ of 
the wave front under the moving frame.  In particular, when $c=0$, 
we replace $u^K(x)$ on $\T^d$  by $U_0(\sqrt{K}x_1)$ on 
$\R\times\T^{d-1}$.  This is natural in view of \eqref{eq:2.hatv}
(with positive sign inside $U_0$) and  \eqref{eq:1.vK}.

As an extension of Lemma \ref{lem:2.3} on $\R$, we have

\begin{lem}  \label{lem:2.7}
The Sturm-Liouville operator is modified as $\mathcal{A} = -\big(\partial_z^2
+ c\partial_z+ f'(U_0(z))\big)$.  The function $U_0'(z)$ is the eigenfunction 
of $\mathcal{A}$ corresponding to the eigenvalue $0$,
and $\mathcal{A}$ is symmetric in the space $L^2(\R, e^{cz}dz)$.
\end{lem}

\begin{proof}
Differentiating \eqref{eq:2.4-1} in $z$, we obtain $\mathcal{A} U'_0(z)=0$.
For $\fa, \psi\in C_0^\infty(\R)$, we have
\begin{align*}
\int_\R (\partial_z^2+ c\partial_z) \fa \cdot \psi \, e^{cz}\, dz
 = \int_\R \fa \big( \partial_z^2(\psi e^{cz}) -c\partial_z(\psi e^{cz})\Big) dz.
\end{align*}
Here,
\begin{align*}
\partial_z^2(\psi e^{cz}) -c\partial_z(\psi e^{cz})
& = \partial_z(\psi' e^{cz} + c \psi e^{cz}) -c(\psi' e^{cz} + c \psi e^{cz}) \\
& = (\psi'' e^{cz} + 2c \psi' e^{cz} + c^2 \psi e^{cz}) -c(\psi' e^{cz} + c \psi e^{cz})\\
& = \psi'' e^{cz} + c \psi' e^{cz}
= (\partial_z^2 + c\partial_z)\psi \cdot e^{cz}.
\end{align*}
Thus, we obtain the symmetry of $\partial_z^2 + c\partial_z$ and therefore
that of $\mathcal{A}$ in $L^2(\R, e^{cz}dz)$.
\end{proof}

Since the interface $\Ga_t$ moves at a constant speed
$c \sqrt{K}$, instead of \eqref{eq:2.5-Phi} or \eqref{eq:fluct}, we consider the
fluctuation of $\rho^{N,K}$ along with the moving interface $\Ga_t$.  Namely, 
we observe $\rho^{N,K}$ on a moving coordinate, that is,
\begin{equation}  \label{eq:fluct-UB}
\Phi^{N,K}(t,x) := N^{d/2} \{\rho^{N,K}(t,x_1+c\sqrt{K}t,\ux)-
U_0(\sqrt{K}x_1)\}.
\end{equation}
Note that $U_0(\sqrt{K}x_1) =u^K(t,x_1+c\sqrt{K}t,\ux)$.
Then, the SPDE \eqref{eq:2.3-1}, with $n=2$ for simplicity,
for $\Phi\equiv \Phi^{N,K}$ is modified as
\begin{align}  \label{eq:2.3-1-UB} 
\partial_t \Phi(t,x) = \De& \Phi(t,x)  +Kf'(U_0(\sqrt{K}x_1)) \Phi(t,x)
 + c \sqrt{K} \partial_{x_1}\Phi^N \\
& + KN^{-d/2} \tfrac12 f''(U_0(\sqrt{K}x_1)) \Phi(t,x)^2   \notag \\
&+ 
\nabla\cdot \big( {g_1(U_0(\sqrt{K}x_1))}\dot{{\mathbb W}}(t,x) \big)
 + \sqrt{K}g_2(U_0(\sqrt{K}x_1)) \dot{W}(t,x).  \notag
\end{align}
The term $c \sqrt{K} \partial_{x_1}\Phi^N$ is added and new in the right-hand side.

In fact, since $c\sqrt{K}t$ is inside of $\rho^N=\rho^{N,K}$
in \eqref{eq:fluct-UB}, from \eqref{eq:2.4-1},
the Boltzmann-Gibbs principle \eqref{eq:LE+T} is modified as
\begin{align} \label{eq:LE+T-UB}
K & N^{d/2} \Big(\bar c_{[N(x+c\sqrt{K}t e_1)]}(\eta^N(s)) -f(U_0(\sqrt{K}x_1))
 - c(\partial_zU_0)(\sqrt{K}x_1)\Big) \\
 & \hskip 10mm +c\sqrt{K} N^{d/2} \partial_{x_1}\rho^N   \notag\\
 &\sim K N^{d/2} \big( f(\rho^N(s, x+c\sqrt{K}t e_1))-f(U_0(\sqrt{K}x_1))\big)   
 + c\sqrt{K} \partial_{x_1}\Phi^N(t,x),  \notag 
\end{align}
noting \eqref{eq:fluct-UB} and
$\partial_{x_1}(U_0(\sqrt{K}x_1)) = \sqrt{K}(\partial_zU_0)(\sqrt{K}x_1)$,
where $e_1$ is the $x_1$-directed unit vector.
This leads to \eqref{eq:2.3-1-UB} by making Taylor expansion up to
the second order term.

Then, under the stretching \eqref{eq:2.stretch}, $\widetilde{\Psi}$ satisfies
the following SPDE in law
\begin{align} \label{eq:2.2-2-UB} 
\partial_t \widetilde{\Psi}(t,z,\ux) =(- & K \mathcal{A}_z + \De_{\ux})
\widetilde{\Psi}(t,z,\ux)  + cK\partial_{x_1}\widetilde{\Psi} \\
& + KN^{-d/2} \tfrac12 f''(U_0(z)) \widetilde{\Psi}(t,z,\ux)^2  \notag \\
&+ K^{3/4}\partial_z \big(  {g_1(U_0(z))} \dot{W}^1(t,z,\ux) \big)+ K^{1/4}
{g_1(U_0(z))} \nabla_{\ux} \cdot \dot{{\mathbb W}}^2(t,z,\ux)  \notag \\
& + K^{3/4} {g_2(U_0(z))} \dot{W}(t,z,\ux), \notag
\end{align}
where $\mathcal{A}_z = - \partial_z^2 - f'(U_0(z))$ is the same as before, 
except for the change from $\bar v^K(z)$ to $U_0(z)$.
The term $cK\partial_{x_1}\widetilde{\Psi}$ is added to the SPDE \eqref{eq:2.2-2}
or \eqref{eq:2.2-2-P}.

Finally, under the scaling \eqref{eq:2.2-4}, we get the same SPDE 
\eqref{eq:2.2-5} or \eqref{eq:2.2-5-P}, but now $\mathcal{A}_z$ defined in
\eqref{eq:2.2-3} with $U_0$ instead of $\bar v^K$ is replaced by
$$
\mathcal{A}_z = - \partial_z^2 - f'(U_0(z)) - c\partial_z.
$$
This is the operator considered in Lemma \ref{lem:2.7}.

Therefore, in the unbalanced case, we expect similar results as in 
Sections \ref{sec:2.2} and \ref{sec:2.3} for the fluctuation along with the
moving interface $\Ga_t$ except that the constants $c_*$ and $c$ are
modified, since $L^2(\R)$ is replaced by a weighted $L^2$-space as
in Lemma \ref{lem:2.7}.

\appendix

\section{Derivation of the SPDE \eqref{eq:2.3-1} from Glauber-Kawasaki dynamics}
\label{Appendix:A}

We consider the Glauber-Kawasaki dynamics $\eta^{N,K}(t)$ 
on a discrete torus $\T_N^d=(\Z/N\Z)^d\equiv \{1,2,\ldots,N\}^d$ of size $N$,
which is a Markov process on $\mathcal{X}_N$ 
with generator $L_N = N^2 L_K + K L_G$ defined below; see 
\cite{F18}, \cite{F23}, \cite{FMST}.

To explain the operators $L_K$ and $L_G$, let us introduce some notation.
The configuration space of particles on $\T_N^d$ with exclusion rule
is defined by $\mathcal{X}_N := \{0,1\}^{\T_N^d}$.
For $\eta=(\eta_p)_{p\in \T_N^d} \in \mathcal{X}_N$ and $p,q\in \T_N^d$
such that $|p-q|=1$, $\eta^{p,q} \in \mathcal{X}_N$ denotes the configuration
obtained from $\eta$ by exchanging the values of $\eta_p$ and $\eta_q$.
For $\eta \in \mathcal{X}_N$ and $p\in \T_N^d$, $\eta^{p} \in \mathcal{X}_N$
denotes that obtained from $\eta$ by flipping the value of $\eta_p$
to $1-\eta_p$.  We consider the operators $\pi_{p,q}$ and $\pi_p$, which act on
a function $G=G(\eta)$ on $\mathcal{X}_N$, each defined by
$$
\pi_{p,q}G(\eta) = G(\eta^{p,q})-G(\eta),
\quad
\pi_{p}G(\eta) = G(\eta^{p})-G(\eta).
$$

Let the jump rates (exchange rates) $c_{p,q}(\eta)>0$ in
Kawasaki part and the flip rates $c_p(\eta)>0$ in Glauber part be given.
These are functions on the configuration space $\mathcal{X} = \{0,1\}^{\Z^d}$
on the whole lattice $\Z^d$, and can be regarded as functions on $\mathcal{X}_N$
for sufficiently large $N$ by assuming the finite-rage property of these functions.
Then, the operators $L_K$ and $L_G$, which act on a function $G$ 
on $\mathcal{X}_N$, are defined by
\begin{align*}
& L_KG(\eta) = \frac12 \sum_{p,q\in \T_N^d:|p-q|=1} c_{p,q}(\eta) \pi_{p,q}G(\eta), \\
& L_GG(\eta) = \sum_{p\in \T_N^d} c_{p}(\eta) \pi_{p}G(\eta).
\end{align*}
We assume the translation-invariance for
$c_{p,q}$ and $c_p$.  In addition, we assume that $c_{p,q}(\eta)$ does not
depend on $\{\eta_p,\eta_q\}$, which implies the reversibility of $L_K$ under the
Bernoulli measures $\{\nu_\rho^N\}_{\rho\in [0,1]}$ on $\mathcal{X}_N$
or $\{\nu_\rho\}_{\rho\in [0,1]}$ on $\mathcal{X}$ with mean $\rho$,
and $f(u):= E^{\nu_u}[\bar c_p]$, $u\in [0,1]$ (see 
\eqref{eq:barc} for $\bar c_p$ and note that $f(u)$ does not depend on $p$ 
due to the translation-invariance of $c_p$) satisfies the bistability and the balance
conditions stated in Section \ref{sec:2.1} with $0<\rho_-<\rho_*<\rho_+<1$.
See Example 4.1 of \cite{F18} for some examples of $f(u)$ obtained from $c_p(\eta)$.

The Glauber-Kawasaki dynamics is a Markov process $\eta^{N,K}(t), t\ge 0$ 
on $\mathcal{X}_N$ generated by
$$
L_N=N^2L_K+KL_G.
$$
We denote $\eta^N(t)$ for $\eta^{N,K}(t)$, especially when $K=K(N)$.

A macroscopically scaled empirical measure (mass distribution) is associated 
with each configuration $\eta \in \mathcal{X}_N$ of the particles by 
\begin{equation} \label{eq:A.1-Q}
\rho^N(dx;\eta) := \frac1{N^d}\sum_{p\in \T_N^d} \eta_p \de_{\frac{p}N}(dx), 
\quad x \in \T^d,
\end{equation}
and thus, one can define for $\eta^{N,K}(t)$:
\begin{equation} \label{eq:A.2-Q}
\rho^{N}(t,dx)\equiv \rho^{N,K}(t,dx) := \rho^N(dx;\eta^{N,K}(t)), \quad 
t\ge 0, \; x\in \T^d.
\end{equation}
In other words, for a test function $\fa=\fa(x)\in C^\infty(\T^d)$, we have
\begin{equation} \label{eq:A.3-Q}
\lan \rho^N(t),\fa\ran = \frac1{N^d} \sum_{p\in \T_N^d} \eta_p^N(t) 
\fa\left(\tfrac{p}N\right),
\end{equation}
where $\lan\rho,\fa\ran$ denotes the integral of $\fa$ with respect to 
the measure $\rho=\rho(dx)$.

Instead of the random measure $\rho^N(t,dx)$ on $\T^d$, one can consider
the scaled particle density field $\rho^{N,K}(t,x)$ defined as a step function
\begin{equation} \label{eq:A.4-Q}
\rho^{N,K}(t,x) := \sum_{p\in \T_N^d} \eta_p^N(t) 1_{B(p/N,1/N)}(x), 
\quad x \in \T^d,
\end{equation}
where $B(p/N,1/N)$ is the box in $\T^d$ with center $p/N$ and side length $1/N$.
Then, $\lan \rho^{N,K}(t),\fa\ran$ is given by \eqref{eq:A.3-Q} with 
$\fa(p/N)$ replaced by $\bar \fa(p/N) := N^d \int_{B(p/N,1/N)} \fa(x)dx$,
which behaves as $\bar \fa(p/N) = \fa(p/N)+O(1/N)$.

In Sections \ref{sec:1} and \ref{sec:1.4}, we discussed based on the particle
density field $\rho^{N,K}(t,x)$.  But, in the following, for simplicity,
we discuss based on the empirical measure $\rho^N(t,dx)$.  We can also
consider the polylinear approximation of $\{\eta_p^N(t)\}$ located at
$p/N$ as in Section 3.2 of \cite{FS}, which is continuous in the spatial
variable $x\in \T^d$.

By applying Dynkin's formula for \eqref{eq:A.3-Q}, we see that
\begin{equation} \label{eq:1.pi-deri}
\lan \rho^N(t),\fa\ran = \lan \rho^N(0),\fa\ran + \int_0^t b^N(\eta^N(s),\fa)ds
+ M_t^N(\fa),
\end{equation}
where $b^N(\eta,\fa)=L_N\lan\rho^N,\fa\ran$ and $M_t^N(\fa)$
is a martingale with the (predictable) quadratic variation
$$
\frac{d}{dt} \lan M^N(\fa)\ran_t = \Ga^N(\eta^N(t),\fa)
$$
and $\Ga^N$ is the so-called carr\'e du champs defined by
$$
\Ga^N(\eta,\fa) := L_N \lan\rho^N,\fa\ran^2
- 2 \lan\rho^N,\fa\ran L_N\lan\rho^N,\fa\ran.
$$
One can decompose $b^N$ and $\Ga^N$ as
\begin{align*}
& b^N(\eta,\fa) = b_K^N(\eta,\fa) + b_G^N(\eta,\fa),\\
& \Ga^N(\eta,\fa) = \Ga_K^N(\eta,\fa) + \Ga_G^N(\eta,\fa),
\end{align*}
where $b_K^N(\eta,\fa)= N^2L_K\lan\rho^N,\fa\ran$,  $b_G^N(\eta,\fa)
= K L_G\lan\rho^N,\fa\ran$, and $\Ga_K^N(\eta,\fa), \Ga_G^N(\eta,\fa)$ are
defined as $\Ga^N(\eta,\fa)$ with $L_N$ replacing by $N^2L_K$, $KL_G$,
respectively.

For the Kawasaki part, we obtain the following lemma.

\begin{lem}\label{lem:1.1}
We have
\begin{align}  \label{eq:b}
b_K^N(\eta,\fa) & = \frac{N^2}{2N^d} \sum_{p,q\in\T_N^d: |p-q|=1} 
c_{p,q}(\eta)(\eta_p-\eta_q) 
\left(\fa\left(\tfrac{q}N\right)  - \fa\left(\tfrac{p}N\right) \right), \\
\label{eq:Ga}
\Ga_K^N(\eta,\fa) & = \frac{N^2}{2N^{2d}} \sum_{p,q\in\T_N^d: |p-q|=1} 
c_{p,q}(\eta) (\eta_p-\eta_q)^2
 \left(\fa\left(\tfrac{q}N\right)  - \fa\left(\tfrac{p}N\right) \right)^2.
\end{align}
In particular, when $c_{p,q}\equiv 1$ called simple Kawasaki dynamics
or simple exclusion process,
one can rewrite \eqref{eq:b} as
\begin{align}  \label{eq:b-2}
b_K^N(\eta,\fa) & = \frac1{N^d} \sum_{p\in\T_N^d} \eta_p \, \De^N
\fa\left(\tfrac{p}N\right)   \equiv \lan\rho^N,\De^N\fa\ran,
\intertext{where}  \notag
\De^N\fa\left(\tfrac{p}N\right) &= N^2 \sum_{q\in \T_N^d:|q-p|=1}
\left(\fa\left(\tfrac{q}N\right)  - \fa\left(\tfrac{p}N\right) \right).
\end{align}
\end{lem}

\begin{proof}
See \cite{F18}, Lemma 1.2 (for any dimension $d$, but with $c_{p,q}\equiv1$)
and Lemma 2.2 (for $d=1$).  The identities \eqref{eq:b} and \eqref{eq:Ga}
follow from the calculation in p.427 respectively p.428 in \cite{F18} noting
that (2.5) is modified by changing $\frac1N$ to $\frac1{N^d}$.
\end{proof}

On the other hand, for the Glauber part, we obtain

\begin{lem}  \label{lem:1.2}
We have
\begin{align}  \label{eq:b-G}
b_G^N(\eta,\fa) & = \frac{K}{N^d} \sum_{p\in\T_N^d} 
\bar c_{p}(\eta)\fa\left(\tfrac{p}N\right), \\
\label{eq:Ga-G}
\Ga_G^N(\eta,\fa) & = \frac{K}{N^{2d}} \sum_{p\in\T_N^d} 
c_{p}(\eta) \fa\left(\tfrac{p}N\right)^2.
\end{align}
where
\begin{align}  \label{eq:barc}
\bar c_p(\eta) := c_p(\eta) (1-2\eta_p) = c_p(\eta) \Big(1_{\{\eta_p=0\}}
- 1_{\{\eta_p=1\}}\Big).
\end{align}
\end{lem}

\begin{proof}
For $b_G^N(\eta,\fa)$, one can compute as
\begin{align*}
b_G^N(\eta,\fa) & = K L_G\lan\rho^N,\fa\ran
= \frac{K}{N^d} \sum_{p\in\T_N^d} 
L_G \eta_p \; \fa\left(\tfrac{p}N\right) \\
& = \frac{K}{N^d} \sum_{p\in\T_N^d} 
c_{p}(\eta)((\eta^p)_p - \eta_p) \fa\left(\tfrac{p}N\right)
= \frac{K}{N^d} \sum_{p\in\T_N^d} 
c_{p}(\eta)(1-2\eta_p)\fa\left(\tfrac{p}N\right), 
\end{align*}
since $(\eta^p)_p = 1-\eta_p$.  This shows \eqref{eq:b-G}.

For $\Ga_G^N(\eta,\fa)$, we have 
\begin{align*}
K L_G \lan \rho^N,\fa\ran^2 
& = \frac{K}{N^{2d}} \sum_{p, q\in\T_N^d} 
L_G(\eta_p \eta_q) \fa\left(\tfrac{p}N\right) \fa\left(\tfrac{q}N\right).
\end{align*}
However, by $L_G\eta_p= \bar c_p(\eta)$ shown above,
when $p\not=q$,
$$
L_G(\eta_p \eta_q)= \bar c_p(\eta) \eta_q + \bar c_q(\eta) \eta_p,
$$
and, when $p=q$, since $\eta_p^2 = \eta_p$,
$$
L_G(\eta_p^2)= \bar c_p(\eta).
$$
 Thus,
\begin{align*}
\Ga_G^N(\eta,\fa)
& = K L_G \lan\rho^N,\fa\ran^2
- 2 \lan\rho^N,\fa\ran K L_G\lan\rho^N,\fa\ran \\
& = \frac{K}{N^{2d}} \Big\{\sum_{p\in\T_N^d} 
\bar c_{p}(\eta) \fa\left(\tfrac{p}N\right)^2
+ \sum_{p\not= q\in\T_N^d} 
\big(\bar c_{p}(\eta) \eta_q + \bar c_q(\eta) \eta_p\big)
\fa\left(\tfrac{p}N\right) \fa\left(\tfrac{q}N\right) \\
& \hskip 30mm
- 2  \frac{K}{N^{2d}} \Big(\sum_{p\in\T_N^d} 
\eta_p \fa\left(\tfrac{p}N\right)\Big)
\Big( \sum_{q\in\T_N^d} 
\bar c_q(\eta) \fa\left(\tfrac{q}N\right) \Big)\Big\}  \\
& = \frac{K}{N^{2d}} \sum_{p\in\T_N^d} 
\big(\bar c_{p}(\eta) -2 \eta_p \bar c_{p}(\eta) \big)
\fa\left(\tfrac{p}N\right)^2.
\end{align*}
Noting that $\bar c_{p}(\eta) (1-2 \eta_p) = c_{p}(\eta)$, we obtain
\eqref{eq:Ga-G}.
\end{proof}

We now consider the fluctuation of the empirical measure (or the particle
density field) around the solution $u^K(t,x)$ of the hydrodynamic equation
\begin{equation}  \label{eq:fluct}
\Phi^N(t,dx) := N^{d/2} \{\rho^N(t,dx)-u^K(t,x)dx\},
\end{equation}
that is,
\begin{equation}  \label{eq:A.13-P}
\Phi^N(t,\fa) = N^{d/2} \Big\{ N^{-d} \sum_{p\in \T_N^d} \eta_p^N(t) 
\fa\left(\tfrac{p}N\right) - \lan u^K(t), \fa\ran \Big\},
\end{equation}
see (1.8) and Section 2.8 of \cite{F18}, and \cite{FLS}.
For the particle system with general jump rates $c_{p,q}$ (called non-gradient type),
the equation for $u^K(t,x)$ is written as
\begin{equation}  \label{eq:A.13-Q}
\partial_t u^K= \nabla\cdot D(u^K)\nabla u^K+ K f(u^K), \quad x \in \T^d,
\end{equation}
with the diffusion matrix $D(u), u\in [0,1]$; see (1.11) of \cite{F23}.
The equation contains the large parameter $K$.

For simplicity, we consider the case that the Kawasaki part is simple,
i.e., $c_{p,q}(\eta)\equiv 1$.  In this case, $D(u)$ is an identity matrix
and the nonlinear PDE \eqref{eq:A.13-Q} for $u^K(t,x)$ has a simple form
\begin{equation}  \label{eq:PDE}
\partial_t u^K= \De u^K + K f(u^K), \quad x \in \T^d,
\end{equation}
which is the same equation as \eqref{eq:AC}.  The fluctuation field 
$\Phi^N$ in \eqref{eq:fluct}
was defined in \eqref{eq:2.5-Phi}  for the scaled particle density field
$\rho^{N,K}(t,x)$ instead of $\rho^N(t,dx)$.

We further consider a simple stationary situation in the PDE \eqref{eq:PDE},
that is, $u^K(t,x)=u^K(x)$, and therefore it satisfies
\begin{align} \label{eq:stAC}
\De u^K + K f(u^K) =0, \quad x \in \T^d,
\end{align}
in particular, taking $x_1$-direction, $u^K(x) = v^K(x_1)$; recall \eqref{eq:baru}.

By Dynkin's formula \eqref{eq:1.pi-deri}, Lemmas \ref{lem:1.1}
(with $c_{p,q}\equiv 1$) and \ref{lem:1.2}  and using \eqref{eq:stAC},
also interpreting $\lan u^K,\fa\ran$ in \eqref{eq:A.13-P} in the empirical
sense,  we have
\begin{align} \label{eq:A.17-P}
\Phi^N(t,\fa) - \Phi^N(0,\fa) =
\int_0^t \Phi^N(s,\De^N\fa)ds + B_G(t)+M^N(t).
\end{align}
Here $M^N(t)$ is a martingale with quadratic variation
$\lan M^N\ran_t =Q_K(t)+Q_G(t)$ and
\begin{align*}
B_G(t)& = \frac{K}{N^{d/2}} \int_0^t \sum_{p\in\T_N^d} 
\Big(\bar c_{p}(\eta^N(s))-f(u^K(\tfrac{p}N))\Big) \fa\left(\tfrac{p}N\right) ds,\\
Q_K(t)&= \frac1{2N^d}\int_0^t \sum_{p,q\in\T_N^d:|p-q|=1} 
(\eta_p^N(s)-\eta_q^N(s))^2
(\nabla^N\fa(\tfrac{q}N,\tfrac{p}N))^2 ds, \\
Q_G(t) & = \frac{K}{N^{d}} \int_0^t\sum_{p\in\T_N^d} 
c_{p}(\eta^N(s)) \fa\left(\tfrac{p}N\right)^2ds,
\end{align*}
where $\nabla^N\fa(\tfrac{q}N,\tfrac{p}N) = N 
\big(\fa(\tfrac{q}N) - \fa(\tfrac{p}N)\big)$.

We now assume the validity of the higher-order Boltzmann-Gibbs principle
(cf.\ \cite{FLS}), that is, a combination of the averaging under the local ergodicity
for the microscopic functions and its asymptotic expansion.  Then, we would have
the following replacement for the microscopic function in $B_G(t)$.
\begin{align} \label{eq:LE+T}
& K N^{d/2}  \big(\bar c_{p}(\eta^N(s)) -f(u^K)\big)  \\
& \sim K N^{d/2} \big( f(\rho^N(s, \tfrac{p}N))-f(u^K)\big)  \notag \\
& \sim K N^{d/2} \Big\{ f'(u^K) (\rho^N(s, \tfrac{p}N) - u^K)
+ \tfrac12 f''(u^K) (\rho^N(s, \tfrac{p}N) - u^K)^2   \notag  \\
& \hskip 50mm
+ \tfrac16 f'''(u^K) (\rho^N(s, \tfrac{p}N) - u^K)^3
\Big\}  \notag \\
& \sim K \Big\{ f'(u^K) \Phi^N(s, \tfrac{p}N)
+ \tfrac12 f''(u^K) N^{-d/2} \Phi^N(s, \tfrac{p}N)^2
+ \tfrac16 f'''(u^K) N^{-d} \Phi^N(s, \tfrac{p}N)^3  \Big\}.  \notag 
\end{align}
Therefore, we expect to have
\begin{align*}
B_G(t)
& \sim K \int_0^t \lan f'(u^K) \Phi^N(s)
+ \tfrac12 f''(u^K) N^{-d/2} (\Phi^N(s))^2
+ \tfrac16 f'''(u^K) N^{-d} (\Phi^N(s))^3,\fa \ran ds.
\end{align*}
This gives the term $KF_n^{N}(u^K(x),\Phi(t,x))$ (with $n=3$)
in the SPDE \eqref{eq:2.3-1}.
The discrete Laplacian $\De^N$ in \eqref{eq:A.17-P} is replaced by
the continuous $\De$ in \eqref{eq:2.3-1}.

For $M^N(t)$, by the local ergodicity, 
its quadratic variation $Q_K(t)+Q_G(t)$ behaves for large $N$ as 
\begin{align}  \label{eq:quadraticV}
t \int_{\T^d} 2 \chi(u^K(x)) |\nabla\fa(x)|^2 dx+
 t K \int_{\T^d} \lan c_0\ran (u^K(x)) \fa^2(x) dx,
\end{align}
where $\chi(u) := u(1-u) = \frac12E^{\nu_u}[(\eta_p-\eta_q)^2]$
and $\lan c_0\ran (u) = E^{\nu_u}[c_0]$.
Thus, we obtain the noise terms in the SPDE \eqref{eq:2.3-1} with
$g_1(u) = \sqrt{2\chi(u)}$ and
$g_2(u) = \sqrt{\lan c_0\ran(u)}$, $u\in (0,1).$
Indeed, the third term with the space-time Gaussian white noise
$\dot{\mathbb W}(t)$ on the right-hand side of \eqref{eq:2.3-1}
in the integrated form in $t$ has the covariance
\begin{align*}
E[\lan \nabla\cdot (g_1(u^K(x)){\mathbb W}(t)),\fa\ran^2]
= E[\lan {\mathbb W}(t); g_1(u^K(x))\nabla\fa\ran^2]
= t \|g_1(u^K)\nabla\fa\|_{L^2(\T^d;\R^d)}^2,
\end{align*}
and this coincides with the first term of \eqref{eq:quadraticV}.
The covariance of the fourth  term with the space-time Gaussian white noise
$\dot{W}(t)$ in \eqref{eq:2.3-1} in the integrated form is given by
\begin{align*}
E[\lan \sqrt{K} g_2(u^K(x)) W(t),\fa\ran^2]
= t K\|g_2(u^K)\fa\|_{L^2(\T^d)}^2,
\end{align*}
and this is the same as the second term of \eqref{eq:quadraticV}.

In this way, one can derive the SPDE \eqref{eq:2.3-1}.

\section{Results of Carr and Pego and their applications}
\label{Appendix:B}

The aim of this appendix is to construct the (unique) periodic profile, the solution 
$v^\e(x)$ of \eqref{eq:1-CP} with two transition layers, to demonstrate relevant
properties, and also to show a certain convergence of the semigroup generated 
by the stretched and linearized operator of the Allen-Cahn equation around 
$v^\e(x)$.   Here we consider only the one-dimensional case so that $x\in \T$.
We follow the development from Carr and Pego \cite{CP} on 
metastable patterns in one-dimensional Allen-Cahn equations, 
adapted to our setting.  Taking $K=\e^{-2}$, $v^\e(x)=v^K(x)$, which is
given by \eqref{eq:baru}, and its graph is found in Figure \ref{Figure2}.

\subsection{Results from Carr and Pego \cite{CP}}
\label{sec:B.1}

Let $f\in C^\infty(\R)$ be the function introduced at the beginning of
Section \ref{sec:2.1}.  For simplicity, we take $\rho_\pm=\pm 1$ so that
$\rho_*\in (-1,1)$ and the balance condition is written as 
$\int_{-1}^1 f(u) du=0$.  Note that the sign of $f$ is opposite in \cite{CP}; 
we write $-f$ for $f$ in \cite{CP}.  The Allen-Cahn equation is 
$u_t = \e^2 u_{xx}+f(u), x \in (0,1)$.
Note $\e^2= 1/K$, following the PDE convention.  The potential $V$ 
corresponding to $f$ is defined by $V'(u)=-f(u)$ and $V(\pm 1)=V'(\pm1)=0$;
note that $V=F$ in \cite{CP}.

\begin{figure}[h]
\centering
\includegraphics[height=30mm]{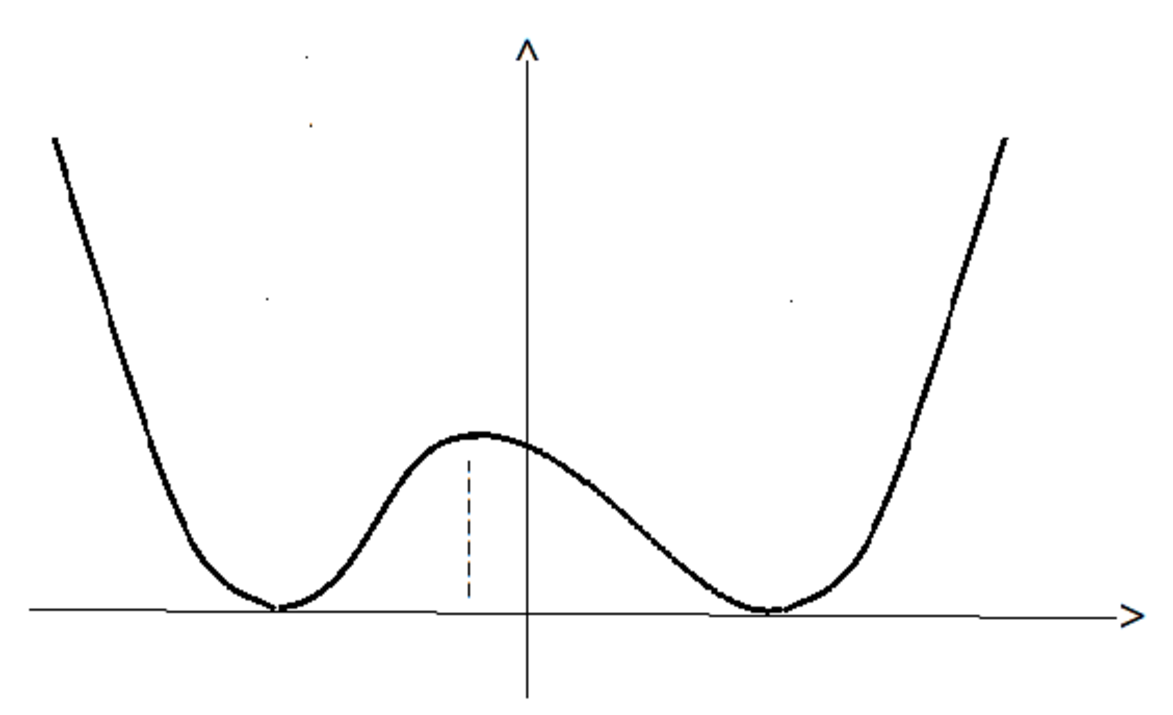}
\vskip -33mm
{\scriptsize \hskip 1mm $V$}
\vskip 22.3mm
{\scriptsize \hskip -6mm $-1$ \hskip 4mm $\rho_*$ 
\hskip 9mm $1$}
\caption{Potential function $V$}  
  \label{Figure4}
\end{figure}

We only consider the case where the number of the transition layers in
the profile is $N=2$, but the case of $N\in 2\N$ can be 
discussed similarly.  Let $v(x) \equiv v^\e(x), x\in \T$ be the solution of
\eqref{eq:baru}, that is, 
\begin{align}  \label{eq:1-CP}
\e^2 v_{xx} +f(v)=0, \quad x\in \T,
\end{align}
satisfying $N \equiv \sharp\{x\in \T; v(x)=\rho_*\}=2$.
Such a $v$ exists uniquely except for translation; see Section \ref{sec:construction}.
To fix ideas, let us normalize it where $v(0)=\rho_*, v_x(0)<0$.
We will view $v$ with respect to `base height' $\rho_*$, instead of $0$ in \cite{CP},
especially since the increasing/decreasing property of $V$ changes at $\rho_*$.
In Figure \ref{Figure2}, $\{h_1, h_2\} = \{x\in \T; v^\e(x)=\rho_*\}$ are the locations
of the $N=2$ layers; recall $h_1=0$.

Carr and Pego \cite{CP} discussed on the interval $[0,1]$ under the Neumann 
boundary condition.  However, concerning the spectral analysis of the linearized
operator, their results are directly applicable in our periodic setting.
Indeed, let us consider $\tilde v \equiv \tilde v^\e$ obtained by shifting
$v \equiv v^\e$ to the left by $m_1$; see Figure \ref{Figure2}.  Then, 
$\tilde v^\e$ satisfies the Neumann condition at $x=0$ and $1$, that is,
$\tilde v^\e_x(0)= \tilde v^\e_x(1)=0$.
Furthermore, the profile $u^{\tilde h}$ with $\tilde h = \{\tilde h_1:= m_1, \tilde h_2:=
1-m_1\}$ (recall $h_2=2m_1$) 
defined by (2.2) or in p.\ 561 of \cite{CP} coincides with $\tilde v^\e$
(although they take $0$ as the `base height'):

\begin{lem}  \label{lem:uh-ue-CP}
$u^{\tilde h}(x) = \tilde v^\e(x), x\in [0,1].$
\end{lem}

\begin{proof}
Since $\tilde v^\e(x)$ satisfies the equation \eqref{eq:1-CP} and
$\tilde v^\e(\tilde h_1) = \tilde v^\e(\tilde h_2) = \rho_*$, by the definition of
$\phi(x,\ell,\pm 1)$ given by (2.1) of \cite{CP} changing the Dirichlet condition
$0$ to $\rho_*$ (also $f$ to $-f$), i.e.\ the solution of
$$
\e^2 \phi_{xx} + f(\phi) =0, \quad \phi(-\tfrac12\ell) = \phi(\tfrac12\ell)=\rho_*,
$$
 (note that these functions are defined for all $x\in \R$ including the outside
 of the interval $[-\tfrac12\ell, \tfrac12\ell]$), we see that
\begin{align}\label{eq:tildeu-CP}
\tilde v^\e(x) = \phi(x-\tilde m_1,2\tilde h_1,-1) 
= \phi(x-\tilde m_2,\tilde h_2-\tilde h_1,+1)
\end{align}
for all $x\in \R$, where $\tilde m_1=0, \tilde m_2= m_2-m_1$.
Recall that ``$\pm1$'' means that the profile $\phi$ takes values greater or less
than $\rho_*$ for $|x|<\frac12\ell$.  Therefore, the weight $\chi$ in the
definition (2.2) in \cite{CP} of $u^{\tilde h}$ to piece these functions 
does not play any role in our setting and we see that
$u^{\tilde h}(x) = \tilde v^\e(x), x\in [0,1].$
\end{proof}

Consider the Sturm-Liouville operator
\begin{align}\label{eq:Lh-CP}
L^\e w := -\e^2 w_{xx} - f'(v^\e) w,
\end{align}
on $\T$.  This is the linearized operator of the Allen-Cahn equation around $v^\e$.
Then, the eigenvalues of $L^\e$ are real and simple, $\{\la_1<\la_2<\cdots\}$ 
and these are the same as those of $L^{\tilde h}$ defined in (3.3) of \cite{CP},
since $u^{\tilde h}$ is a shift of $v^\e$ by Lemma \ref{lem:uh-ue-CP}.  Hence,
by Theorem 4.1 and its Corollary 2 of \cite{CP} recalling $N=2$ in our case, we have
\begin{align}  \label{eq:2-CP}
0=\la_1 < \la_2 \le C e^{-c/\e} \quad \text{ and } \quad \la_3\ge \La_1,
\end{align}
for all $0<\e < \e_0$ and some $\e_0>0$, and $C, c, \La_1>0$ uniformly in $\e$.
See also \cite{WY} for some detailed analysis in the case of $f(u)=u-u^3$.

Note that the eigenvalues of $L^\e$ are all nonnegative in our setting.  Indeed,
$v^\e$ is a local minimizer of the energy functional 
$E[v] = \int_\T \big( \frac{\e^2}2 v_x^2+V(v)\big) dx$,
under the periodic boundary condition,  and this implies $\lan L^\e w,w\ran_{L^2(\T)}
= \frac{d^2}{d\de^2} E[v^\e+\de w] \big|_{\de=0} \ge 0$.  Note also that
\eqref{eq:1-CP} is the Euler-Lagrange equation for the local minimizer and
its solution with $N=2$ is unique except for translation by Proposition
\ref{prop:B.11}.  The function $u^{\tilde h}$ in \cite{CP} is 
a metastable point dynamically,
i.e.\ under the time-evolution determined by the Allen-Cahn equation (1.1) of
\cite{CP} subject to the Neumann boundary condition, while $v^\e$ is stable.
Thus, for the operator $L^\e$ defined by \eqref{eq:Lh-CP} with 
$u^{\tilde h}$ instead of $v^\e$, the minimal eigenvalue $\la_1$ can be
negative as in \cite{CP}.

Define the functions $\t_j^\e(x), x\in \T, j=1,2$ by
\begin{align}  \label{eq:tau-j-CP}
\t_j^\e(x) = - \ga^j(x) v_x^\e(x), \quad x\in \T,
\end{align}
where $v_x^\e= \partial_x v^\e$ and $\ga^j \in C^\infty(\T)$ such that
\begin{align*}
\ga^j(x) = \left\{
\begin{aligned}
0,& \quad x \notin [m_j,m_{j+1}], \\
1,& \quad x \in [m_j+2\e,m_{j+1}-2\e], 
\end{aligned}
\right.
\end{align*}
for $j=1,2$, where $m_1, m_2$ are given as in Figure \ref{Figure2} and
$m_3=m_1+1$ (i.e.\ $m_3=m_1$ mod $1$); see Section 2.4 of \cite{CP} for more
details for $\ga^j$.  Note that $v^\e$ is the shift of $u^{\tilde h}$, i.e.\ 
$v^\e(x) = u^{\tilde h}(x-m_1)$ by Lemma \ref{lem:uh-ue-CP},
and we can apply the results of \cite{CP} for $u^{\tilde h}$.
The function $\t_j^\e(x)$ corresponds to the shift of
the $j$th layer, i.e., the one on $[m_1,m_2]$ for $j=1$ and on $[m_2, m_3]$
for $j=2$ and is an almost $0$-eigenfunction of $L^\e$.

Let $\psi_j\in L^2\equiv L^2(\T)$, $j=1,2$ be the normalized eigenfunctions
of $L^\e$ corresponding to the eigenvalues $\la_j$, and set
$\mathcal{S}_\psi = \,$span$\{\psi_1, \psi_2\}$ and
$\mathcal{S}_\t = \,$span$\{\t_1^\e, \t_2^\e\}$; see Section 4.2 of \cite{CP}.
Let $\pi$ and $\pi_\t$ be the orthogonal projections on  $L^2$ to
$\mathcal{S}_\psi$ and $\mathcal{S}_\t$, respectively.  

Note that $\mathcal{S}_\t$ is explicitly defined, while $\mathcal{S}_\psi$ is unclear.
Note also that all these depend on $\e$.

\begin{rem}
It is easy to see that $\psi_1(x) = e^\e(x) := v_x^\e(x)/\|v_x^\e\|_{L^2(\T)}$,
since $e^\e$ satisfies $L^\e e^\e=0$ which follows by differentiating
\eqref{eq:1-CP}.  The second eigenfunction is given asymptotically by
$\tilde e^\e(x) := |e^\e(x)|$ for small $\e>0$, since $\| \tilde e^\e\|_{L^2(\T)}=1,
\lan \tilde e^\e, e^\e\ran =0$, $L^\e \tilde e^\e(x)=0$ holds except for
$x$ such that $\tilde e^\e(x)=0$, and at such $x$, $e^\e(x)$ becomes nearly flat
as $\e\downarrow 0$.  One can say that $\t_1^\e$ and $\t_2^\e$ correspond to
$(e^\e+|e^\e|)/2$ and $(e^\e-|e^\e|)/2$, respectively.
\end{rem}

We now state that 
the semigroup $e^{-tKL^\e}w \equiv e^{-tL^\e/\e^2}w$ on $L^2$ is projected to
the two-dimensional space $\mathcal{S}_\t \equiv \mathcal{S}_\t^\e$
for $t>0$ in the following sense.

\begin{lem}  \label{lem:1.1-CP}
For some $C, c>0$, we have
\begin{align*}
\|e^{-tKL^\e}w - \pi_\t w\| \le C \Big\{(t+1) e^{-c/\e} + e^{-tc/\e^2}\Big\} \|w\|
\end{align*}
for every $t\ge 0$ and $w\in L^2\equiv L^2(\T)$, where 
$\|\cdot\|=\|\cdot\|_{L^2(\T)}$ and 
$K=1/\e^2$.  In particular, the right hand side converges to
$0$ as $\e\downarrow 0$ when $t>0$.
\end{lem}

\begin{proof}
By Lemma 4.3 of \cite{CP} and noting $\la_1=0$, we have
\begin{align} \label{eq:3-CP}
\lan w, L^\e w\ran \ge \|w\|^2 \la_3(1-\cos^2\eta),
\end{align}
where $\lan\cdot,\cdot\ran$ is the inner product of $L^2$ and
\begin{align*}
\cos\eta= \sup\big\{\lan w,\psi\ran/\| w\|\|\psi\|; \psi \in 
\mathcal{S}_\psi \setminus \{0\}\big\}.
\end{align*}
By Lemma 4.5 of \cite{CP} (doubly called Lemma 4.4), for all $\t \in \mathcal{S}_\t$,
\begin{align} \label{eq:4-CP}
\|(I-\pi)\t\|\le \mu(h) \|\t\|
\end{align}
where
\begin{align} \label{eq:5-CP}
\mu(h)=g_2(h)/\la_3 \quad \text{ and } \quad 
0\le g_2(h) \le C \max_k \b^k \le C' e^{-c/\e},
\end{align}
for some $C, C', c>0$ 
by Lemma 4.4 and Corollary 2 of \cite{CP}; see \cite{CP} for the definition
of $g_2(h)$ and $\{\b^k\}$.  Moreover, by Lemma 4.6 of \cite{CP},
\begin{align} \label{eq:6-CP}
\cos^2\eta  \le (\cos^2\th+\mu^2(h))/(1-\mu^2(h)),
\end{align}
where
\begin{align*}
\cos\th= \sup\big\{\lan w,\t \ran/\| w\|\|\t\|; \t \in 
\mathcal{S}_\t \setminus \{0\}\big\}.
\end{align*}
By \eqref{eq:3-CP} and \eqref{eq:6-CP}, when $w\in L^2$ satisfies 
$\lan w,\t\ran =0$ for every $t \in \mathcal{S}_\t$
(so that $\cos\th=0$), for some $c>0$, we have
\begin{align*}
\lan w, L^\e w\ran & \ge \|w\|^2 \la_3 \big(1-\mu^2(h)/(1-\mu^2(h))\big)  \\
& \ge c \|w\|^2.
\end{align*}

Thus, in
\begin{align*}
\|e^{-tKL^\e}w - \pi_\t w\| 
\le \|e^{-tKL^\e} \pi_\t w - \pi_\t w\|  + \|e^{-tKL^\e}(I-\pi_\t)w\|,
\end{align*}
the second term is bounded by
\begin{align*}
\|e^{-tKL^\e}(I-\pi_\t)w\|  \le e^{-tKc} \|(I-\pi_\t)w\| \le e^{-tc/\e^2} \|w\|,
\end{align*}
since $(I-\pi_\t)w \perp \mathcal{S}_\t$.

For the first term, since $\la_1=0$ and 
$$
0\le K\la_2 \le \tfrac{C}{\e^2}e^{-c/\e} \le C' e^{-c/2\e},
$$
we see
$$
\|e^{-tKL^\e} w - \pi w\| \le \Big\{ \Big(1-e^{-tC'e^{-c/2\e}}\Big) 
+ e^{-tK\la_3}\Big\} \|w\|.
$$
Therefore, taking $\pi_\t w$ instead of $w$, the first term is bounded as
$$
\|e^{-tKL^\e} \pi_\t w - \pi_\t w\|   \le \Big\{ tC'e^{-c/2\e}
+ e^{-t\la_3/\e^2}\Big\} \|w\| + \|\pi \pi_\t w -\pi_\t w\|,
$$
by using $1-e^{-x}\le x, x\in \R$ and $\|\pi_\t w\| \le \|w\|$.
However, by \eqref{eq:4-CP} and \eqref{eq:5-CP},
$$
\|\pi \pi_\t w -\pi_\t w\| \le \mu(h) \|\pi_\t w\| \le \mu(h) \|w\|
\le Ce^{-c/\e}\|w\|,
$$
verifying the conclusion.
\end{proof}

\subsection{Applications to our setting}  \label{subsec:B.2}

As in Figure \ref{Figure2}, our profile has two transition layers (interfaces), 
a decreasing one in the region $[-m_2,m_1]$ centered at $h_1=0$,
and an increasing one in the region $[m_1,m_2]$ centered at $h_2$.
One of our goals will be to show that,
as $\e\downarrow 0$, both transitions become sharp and concentrate 
near their centers; cf.\ \eqref{eq:baru-CP} and \eqref{eq:u-baru-CP} below.
Another goal will be to show the convergence of the semigroup, 
Proposition \ref{thm:3-6-0-CP} already stated as Proposition
\ref{prop:3-6-0-CP}, and a uniform bound on the derivative of
the semigroup, Lemma \ref{lem:3.4}.

Let us recall some notation:  $x\in \T= [-1/2,1/2)$,
$z := \sqrt{K} x \in \sqrt{K}\T = [-\sqrt{K}/2, \sqrt{K}/2)$ is 
a stretched variable, $v^\e(x), x\in \T$ is the solution of \eqref{eq:1-CP},
\begin{align*}
& \bar v^\e(z) := v^\e(\e z)= v^\e(z/\sqrt{K}),
\quad z \in \e^{-1}\T = \sqrt{K}\T,
\intertext{is a stretched profile, and}
& \mathcal{A} \equiv \mathcal{A}_z^K := -\partial_z^2 - f'(\bar v^\e(z))
\end{align*}
is the stretched Sturm-Liouville operator on $\sqrt{K}\T$.

To apply Lemma \ref{lem:1.1-CP} in our setting, we prepare the following
lemma which is shown by a simple change of variables.

\begin{lem} \label{lem:2.1-a-CP}
For $F=F(z)$ on $\sqrt{K}\T$, consider $\tilde u(t,z)
:= e^{-tK\mathcal{A}} F(z)$ and set $u(t,x) := \tilde u(t, \sqrt{K}x)$,
$x\in \T$.  Then, we have
$$
u(t, x) = e^{-t K L^\e} \check{F}(x), \quad x \in \T,
$$
that is
$$
e^{-tK\mathcal{A}} F(z) = e^{-t K L^\e} \check{F}(z/\sqrt{K}), 
\quad z \in \sqrt{K}\T,
$$
where $L^\e$ is given in \eqref{eq:Lh-CP} and $\check{F}(x) := F(\sqrt{K}x)$.
\end{lem}

\begin{proof}
First, we see $u(0,x) = \tilde u(0,\sqrt{K}x) = F(\sqrt{K}x)
= \check{F}(x)$.  Then, since $\tilde u$ satisfies the equation
$\partial_t \tilde u(t,z)= -K \mathcal{A}_{z} \tilde u(t,z)$,
we see
\begin{align*}
\partial_t u(t,x)
& = (\partial_t \tilde u)(t,\sqrt{K} x) \\
& = K \Big( (\partial_z^2 \tilde u)(t,\sqrt{K} x)
 + f'(\tilde v^\e(\sqrt{K} x)) \tilde u(t,\sqrt{K} x)  \Big) \\
& = K \Big( \tfrac1K \partial_x^2 u(t,x)
 + f'(v^\e(x))  u(t, x)  \Big) \\
 & = -K L^\e u(t, x),
\end{align*}
showing the conclusion.
\end{proof}

Recall the standing wave solution $U_0(z), z\in \R$ defined by 
\eqref{eq:stand-wave} here with $\rho_\pm = \pm 1$.
We took $\rho_*$ as the height for $U_0$ at $z=0$: $U_0(0)= \rho_*$.
Recall Lemma \ref{lem:decay-U0} for the exponential decay property of 
$U_0(z)$ as $|z|\to\infty$.

Then we show the following proposition, already stated in Proposition
\ref{prop:3-6-0-CP}.

\begin{prop}  \label{thm:3-6-0-CP}
For every $t>0$ and $G\in L^2(\R) \cap L^1(\R)$, we have
\begin{align*} 
\lim_{\e\downarrow 0} \| e^{-tK\mathcal{A}} G(z) 
- \lan G, e\ran_{L^2(\R)} e(z)\|_{L^2(\e^{-1}\T)} =0,
\end{align*}
with the first $G$ interpreted as $G|_{\e^{-1}\T}$,
where $K=\e^{-2}$ and
\begin{align*}
e(z) := U_0'(-z) /\|U_0'\|_{L^2(\R)}, \quad z\in \R.
\end{align*}
\end{prop}

First, we prepare estimates for $v^\e$ stronger than \cite{CP}.  In fact,
our $v^\e$ is special compared to $u^h$ in \cite{CP} in the sense that
it is a local minimizer of $E[v]$ and we have better estimates than, 
for example, Proposition 2.2 ($H,\de_1,\de_2$ are fixed and taken 
independently of $\e$), Lemma 7.2 ($H,\de$ are fixed), Lemma 8.2 
(only near the layers, i.e., for $|\frac{x-h_j}\e| \le H$) in \cite{CP}.

Define the function $\hat v^\e(x)$ from $U_0$ (same as $\hat v^K$ in
\eqref{eq:2.hatv}) as
\begin{align}  \label{eq:baru-CP}
\hat v(x) \equiv \hat v^\e(x) = \left\{ 
\begin{aligned}
U_0(-\tfrac{x}\e), & \quad x \in [0,m_1], \\
U_0(\tfrac{x-h_2}\e), & \quad x \in [m_1,m_2], \\
U_0(\tfrac{1-x}\e), & \quad x \in [m_2,1].
\end{aligned}
\right.
\end{align}
Recall $f'(\pm 1)<0$ and this implies $V''(\pm 1)>0$.

\begin{lem}\label{prop:3.6-CP}
For some $C>0$, we have
\begin{align} \label{eq:upm1-CP}
-1 < v^\e(m_1)\le -1+ C \sqrt{\e}, \quad 1- C \sqrt{\e} \le v^\e(m_2)<1.
\end{align}
Moreover, we have
\begin{align} \label{eq:u-baru-CP}
\begin{aligned}
& 0\le v^\e(x)- \hat v^\e(x) \le C\sqrt{\e}, \quad x\in [0,h_2].\\
& 0\le \hat v^\e(x)- v^\e(x) \le C\sqrt{\e}, \quad x\in [h_2,1].
\end{aligned}
\end{align}
\end{lem}

\begin{proof}
For \eqref{eq:upm1-CP}, it is sufficient to prove the first inequality, since the 
second is similar.  

\vskip 1mm
{\it Step 1.}  Note that the function $v^\e(x), x\in [0,h_2]$ is a minimizer of 
the functional
\begin{align}  \label{eq:Energy-CP}
E[v] = \int_0^{h_2}\Big( \tfrac{\e^2}2 v_x^2 + V(v)\Big) dx
\end{align}
under the condition $v(0)=v(h_2)=\rho_*$ and $v\le \rho_*$ on $[0,h_2]$.
By the construction in Section \ref{sec:construction}, $v^\e$ is symmetric 
about $m_1$ and is the unique solution of the Euler-Lagrange equation 
\eqref{eq:1-CP} with $N=2$ layers on $\T$.

Since $h_2=2 m_1$, $\hat v = \hat v^\e$ defined by \eqref{eq:baru-CP}
is also symmetric under the reflection at $m_1$ with $N=2$ layers.
In particular, it is continuous.  It is not in $C^1$ at $x=m_1$ but one can 
smear/modify it with a small energy cost.  
We test the energy $E[v]$ by taking $v=\hat v$.
Then, since $\hat v$ is symmetric and satisfies the ODE $\hat v_x
= \sqrt{2V(\hat v)/\e^2}$ on $[m_1,h_2]$ by Lemma \ref{lem:2.1-CP} 
below with $e=e(x)=0$ (by letting $x\to-\infty$ in \eqref{eq:e-CP}), we have
\begin{align*}
E[\hat v] & = 2 \int_{m_1}^{h_2}\Big( \tfrac{\e^2}2 (\hat v_x)^2 + V(\hat v)\Big) dx \\
& = 2\e \int_{m_1}^{h_2} \hat v_x \sqrt{2V(\hat v)} dx \\
& = 2\e \int_{U_0(-m_1/\e)}^{\rho_*} \sqrt{2V(v)} dv \le C_0 \e,
\end{align*}
where $C_0 = 2\int_{-1}^{\rho_*} \sqrt{2V(v)} dv$; recall $m_1-h_2=-m_1$.  
In particular, since $v^\e$
is the minimizer, we obtain $E[v^\e] \le C_0\e$.

\vskip 1mm
{\it Step 2.}  
Now let us assume that the upper bound in \eqref{eq:upm1-CP} for $v^\e(m_1)$ 
does not hold, i.e.\ $v^\e(m_1) > -1+C\sqrt{\e}$ holds.
Then, since $v^\e$ is symmetric under the reflection at $m_1$,
$v^\e\le \rho_*$, it is increasing on $[m_1,h_2]$ and $V$ is increasing on
$[-1,\rho_*]$ (see Figure \ref{Figure4}), we have
\begin{align*}
E[v^\e] & \ge 2 \int_{m_1}^{h_2} V(v^\e(x))dx \\
& \ge 2m_1 V(-1+C\sqrt{\e}) \ge c_0 C^2 \e,
\end{align*}
for some $c_0>0$, since $V''(\pm 1)>0$.  So, if $C$ satisfies
$c_0 C^2 >C_0$, we have a contradiction.  Thus, we obtain the upper bound 
in \eqref{eq:upm1-CP} for $v^\e(m_1)$ with $C=\sqrt{C_0/c_0}$,
as well as the lower bound for $v^\e(m_2)$.
The bound $v^\e(m_2)<1$ follows from $v^\e(m_2) \le \hat v^\e(m_2)<1$, as shown
below in Step 3.  This proves \eqref{eq:upm1-CP}.

\vskip 1mm
{\it Step 3.}  
For \eqref{eq:u-baru-CP}, we show it only for $x\in [h_2,m_2]$.
The other regions are similar.  Set $v_1(x) = \hat v^\e(x+h_2)$ and 
$v_2(x) = v^\e(x+h_2)$, $x\in [0,m_2-h_2]$.  Let us compute, by using
Lemma \ref{lem:2.1-CP} below and noting that $v_1(x), v_2(x)>\rho_*$ 
(when $x\not=0$), that
\begin{align}  \label{eq:2.6-B}
\partial_x \tfrac12 (v_1(x)-v_2(x))^2
& = (v_1(x)-v_2(x)) \partial_x (v_1(x)-v_2(x))  \\
& = (v_1(x)-v_2(x)) \big( \sqrt{2V(v_1(x))/\e^2} - \sqrt{2(V(v_2(x)) +e)/\e^2}\big).
\notag
\end{align}
Here, by $v_x^\e(m_2)=0$ and \eqref{eq:upm1-CP}, for some $C_e>0$
\begin{align} \label{eq:2.7-B}
e= \tfrac{\e^2}2 \big((v_2)_x(m_2-h_2)\big)^2 -V(v_2(m_2-h_2)) 
= -V(v^\e(m_2)) \ge -C_e \e.
\end{align}

We first note that $v_1(x) > v_2(x), x\in (0,m_2-h_2]$.
(We can assume this also at $x=0$ by taking $v_1(0)=\rho_*+\de > v_2(0)=\rho_*$,
and then letting $\de\downarrow 0$.)  In fact, since $V(v_1)>V(v_2)$ 
for $\rho_*<v_1<v_2<1$, if $v_1(x)<v_2(x)$, from \eqref{eq:2.6-B} noting
$e\le 0$, we have
\begin{align*}
\partial_x \tfrac12 (v_1(x)-v_2(x))^2 <0.
\end{align*}
This means that $ (v_1(x)-v_2(x))^2$ is decreasing if $v_1(x)<v_2(x)$.  
Therefore, once $v_1=v_2$ happens, they cannot move to the side of $v_1<v_2$,
since moving to that side means that $ (v_1(x)-v_2(x))^2$ increases.

\vskip 1mm
\noindent
{\it Stage 1}.  At this stage, we consider for $x \ge 0$ such that 
$v_2(x) \le 1-C_1 \e^\b$ holds with some $C_1>0$ and $\b>0$ (we will take
$\b=1/2$ and $C_1$ large enough later for
\eqref{eq:EF-CP}).  Then, setting $E:= -e >0$ (recall \eqref{eq:2.7-B} for $e$)
and noting that $\rho_* \le v_2(x)<v_1(x)<1$ implies
$0\le V(v_1(x)) < V(v_2(x)) \le  V(\rho_*)$, we have
\begin{align}  \label{eq:F-F-CP}
&  \partial_x (v_1(x)-v_2(x))  \\
& = \sqrt{2V(v_1(x))/\e^2} - \sqrt{2(V(v_2(x)) -E)/\e^2}  \notag \\
& = \sqrt{2/\e^2} ( \sqrt{V(v_1(x))} - \sqrt{V(v_2(x))})  
+
\sqrt{2/\e^2}  \sqrt{V(v_2(x))} \Big( 1- \sqrt{ 1- \tfrac{E}{V(v_2(x))} } \Big) 
\notag \\
& \le - \sqrt{2/\e^2}C_2 (v_1(x)-v_2(x)) + C_3 \e^{-\b}  
\notag \\
& = -\tfrac{C_4}\e (v_1(x)-v_2(x)) +C_3 \e^{-\b},   \notag
\end{align}
where
$$
C_2 := - \inf_{v\in [\rho_*,1]} (\sqrt{V(v)})'
= \sup_{v\in [\rho_*,1]} - (\sqrt{V(v)})'.
$$
The derivation of the inequality in \eqref{eq:F-F-CP} is explained below.
Here, $C_2<\infty$ follows from
\begin{align*}
- (\sqrt{V(v)})' \le C, \quad v\in [\rho_*,1],
\end{align*}
since $- (\sqrt{V(v)})' = \tfrac12 (V(v))^{-1/2} (-V'(v)) \in C^\infty([\rho_*,1))$
and as $v\uparrow 1$, $V(v)=C(v-1)^2+O((v-1)^3), C>0,$ so that 
$-(\sqrt{V(v)})'  \to \sqrt{C}$.

In the above estimate \eqref{eq:F-F-CP}, 
to derive the second term in the fourth line, we use
$$
1-\sqrt{1-x} \le x, \quad x\in [0,1]
$$
noting that
\begin{align}  \label{eq:EF-CP}
\tfrac{E}{V(v_2(x))}\le \tfrac{C_e\e}{C_5 C_1^2 \e^{2\b}} <1,
\end{align}
from \eqref{eq:2.7-B}
by taking $\b=1/2$ and $C_1>0$ large enough, for every $0<\e \le 1$.
Thus, again by \eqref{eq:2.7-B}, the second term in the third line is bounded by
$$
\sqrt{2/\e^2}  \tfrac{E}{\sqrt{V(v_2(x))}}
\le \sqrt{2/\e^2} \tfrac{C_e\e}{\sqrt{C_5} C_1 \e^{\b}} 
$$
and we get the second term in the fourth line.

Therefore, since $v_1(x)-v_2(x)>0$ and we took $\b=1/2$,
\begin{align*}
(v_1(x)&-v_2(x)) \big(\sqrt{2V(v_1(x))/\e^2} - \sqrt{2(V(v_2(x)) -E)/\e^2} \big) \\
& \le -\tfrac{C_4}\e (v_1(x)-v_2(x))^2 +C_3 \e^{-1/2}(v_1(x)-v_2(x)) \\
& \le -\tfrac{C_4}{2\e}  (v_1(x)-v_2(x))^2 +C_6,
\end{align*}
for some $C_6>0$.  Thus, $g(x) := \frac12 (v_1(x)-v_2(x))^2$ satisfies the inequality
$g'(x) \le -C_7g(x) + C_6$ with $C_7\equiv C_7^\e = \tfrac{C_4}{2\e}$.  Since
$$
(e^{C_7x}g(x))' = C_7e^{C_7 x}g(x) + e^{C_7 x} g'(x)
\le C_6 e^{C_7x},
$$
and $g(0)=0$,
we get
$$
g(x) \le e^{-C_7x} \int_{0}^x C_6 e^{C_7y} dy \le \tfrac{C_6}{C_7},
$$
as long as $v_2(x) \le 1-C_1\e^{1/2}$.  Thus, for such $x$, we have
for $C_8=4C_6/C_4>0$,
$$
(v_1(x)-v_2(x))^2 \le  C_8\e.
$$

\vskip 1mm
\noindent
{\it Stage 2}.  For $x>0$ such that $v_2(x) \ge 1- C_1 \e^{\b}$ with $\b=1/2$, 
we automatically have  $0< v_1(x)-v_2(x) \le C_1 \e^{1/2}$, since $v_1(x)<1$.

Summarizing these two stages, we finally obtain
$$
0< v_1(x)-v_2(x) \le C \e^{1/2}
$$ 
as long as $v_2$ is increasing, that is, for $x\in (0,m_2-h_2]$.
This shows \eqref{eq:u-baru-CP} for $x\in [h_2,m_2]$.
\end{proof}

The following holds for all $x\in \T$; compare with Lemma 8.2 of \cite{CP} 
which is limited near the layers.  As $\hat v^\e$ is not differentiable at
$m_1$ and $m_2$, in the next lemma, $\hat v_x^\e$ should be understood
as left and right derivatives at these points.

\begin{lem}  \label{prop:3.7-CP}
We have
$$
|v_x^\e-\hat v_x^\e| \le C/\sqrt{\e}, \quad x \in \T.
$$
\end{lem}

\begin{proof}
In the estimate \eqref{eq:F-F-CP} with $\b=1/2$, 
estimating $|\partial_x (v_1(x)-v_2(x))|$ 
and then using \eqref{eq:u-baru-CP}, we get
\begin{align*}
|\hat v_x^\e(x) - v_x^\e(x)|
\le \tfrac{C_4}\e |\hat v^\e(x) - v^\e(x)| + C_3 \e^{-1/2}
\le C \e^{-1/2}
\end{align*}
on $[h_2,m_2]$.  The other regions are similar.
\end{proof}

We now come to the proof of Proposition \ref{thm:3-6-0-CP}.

\begin{proof}[Proof of Proposition \ref{thm:3-6-0-CP}]
Recall the space $\mathcal{S}_\tau$ in Section \ref{sec:B.1}.  
Let $\hat \t_1^\e(x) \equiv \t_1^\e(x)/\|\t_1^\e\|_{L^2(\T)}$ be 
the normalization of $\t_1^\e(x)= -\ga^1(x) v_x^\e(x)$.  Similarly, 
let $\hat \t_2^\e(x) \equiv \t_2^\e(x)/ \|\t_2^\e\|_{L^2(\T)}$ where
$\t_2^\e(x)= -\ga^2(x) v_x^\e(x)$ as in \eqref{eq:tau-j-CP}.  
Recall that $\check G$ on $\T$ is defined from $G\in L^2(\R)\cap L^1(\R)$
restricted on $\e^{-1}\T$ in Lemma \ref{lem:2.1-a-CP} and
the projection $\pi_\tau$ with respect to $\hat \t_1$ and $\hat \t_2$ 
is given by
\begin{align} \label{eq:B.18-Q}
\pi_\t \check{G} = \lan \check{G},\hat \t_1^\e\ran_{L^2(\T)}  \hat \t_1^\e
+ \lan \check{G},\hat \t_2^\e\ran_{L^2(\T)}  \hat \t_2^\e.
\end{align}

By Lemma \ref{lem:2.1-a-CP} and then by Lemma \ref{lem:1.1-CP},
we have
\begin{align*} 
\| e^{-tK\mathcal{A}} G(z) 
- \pi_\t \check{G}(\e z)\|_{L^2(\e^{-1}\T)}
= \e^{-1/2}\|e^{-t K L^\e} \check{G}(x) - \pi_\t \check{G}(x)\|_{L^2(\T)}
\to 0
\end{align*}
as $\e\downarrow 0$ for $t>0$, noting
$\|\check{G}\|_{L^2(\T)}= \e^{1/2} \|G\|_{L^2(\e^{-1}\T)}
\le \e^{1/2} \|G\|_{L^2(\R)}$.
To replace $\pi_\t \check{G}(\e z)$
with $\lan G,e\ran_{L^2(\R)}e(z)$ (recall $e(z)$ in
Proposition \ref{thm:3-6-0-CP}), we prove
\begin{align}  \label{eq:3.8-CP}
\lim_{\e\downarrow 0} \| \pi_\t \check{G}(\e z) - \lan G,e\ran_{L^2(\R)}e(z)\|_{L^2(\e^{-1}\T)} =0.
\end{align}

By \eqref{eq:B.18-Q}, the norm in \eqref{eq:3.8-CP} is bounded by
\begin{align}  \label{eq:G-tau12-CP}
\| \lan \check{G},\hat \t_1^\e\ran_{L^2(\T)}\hat \t_1^\e(\e z)\|_{L^2(\e^{-1}\T)}
+  \| \lan \check{G},\hat \t_2^\e\ran_{L^2(\T)}  \hat \t_2^\e(\e z) - \lan G, e\ran_{L^2(\R)}e(z)\|_{L^2(\e^{-1}\T)}.
\end{align}
Note that the origin `$x=\e z=0$' belongs to $[m_2, m_3]$ (in mode $1$), which 
covers the support of $\t_2$, and that the support of $\t_1$ is in $[m_1, m_2]$;
see Figure \ref{Figure2}.
So, we expect the first layer, when scaled by $1/\e\equiv\sqrt{K}$, not to 
contribute much in terms of the projection given the estimates on 
$e(z) = U_0'(-z) /\|U_0'\|_{L^2(\R)}$ far away from the origin.

First, let us consider the second norm in \eqref{eq:G-tau12-CP}.
Proposition 2.3 of \cite{CP} shows that
\begin{align}  \label{eq:Prop2.3-CP}
\|\t_2^\e\|_{L^2(\T)} = (A_0+o(1)) \e^{-1/2},
\end{align}
as $\e\downarrow 0$, 
where $A_0=S_\infty^{1/2} = \| U_0'\|_{L^2(\R)}$ (see (8.6) and p.535 of \cite{CP}).
Then, $\hat \t_2^\e(x) = \t_2^\e(x)/ \|\t_2^\e\|_{L^2(\T)}$ is equal to
\begin{align*}
\hat \t_2^\e(x) 
& =- \sqrt{\e} (A_0+o(1))^{-1} \ga^2(x)v_x^\e(x)  \\
& = \sqrt{\e} (A_0+o(1))^{-1} \ga^2(x) 
\Big( \tfrac1\e U_0'(-x/\e) +R^\e(x) \Big),
\end{align*}
where the error term is estimated as $|R^\e(x)| \le \tfrac{C}{\sqrt{\e}}$
by Lemma \ref{prop:3.7-CP}.  In particular, from the assumption $G\in L^1(\R)$,
\begin{align}  \label{eq:G-tau-CP}
\lan \check{G},\hat \t_2^\e\ran_{L^2(\T)}  
& =\sqrt{\e} (A_0+o(1))^{-1} \int_{-m_2}^{m_1} \ga^2(x) G(x/\e)  
\Big( \tfrac1\e U_0'(-x/\e) + R^\e(x) \Big) dx  \\
& = \sqrt{\e} (A_0+o(1))^{-1} \int_{-m_2/\e}^{m_1/\e} \ga^2(\e z)G(z)
\big( U_0'(-z) + \e R^\e(\e z) \big) dz    \notag  \\
& =  \sqrt{\e} \Big(A_0^{-1} \int_\R G(z) U_0'(-z) dz + o(1) \Big).   \notag
\end{align}
For the last line, the contribution of the error term $R^\e$ is $O(\e^{1/2})$, since
$$
\int_{-m_2/\e}^{m_1/\e} \big|\ga^2(\e z)G(z) \, \e R^\e(\e z) \big| dz 
\le C \e^{1/2} \|G\|_{L^1(\R)},
$$
and, for the other term, the error to remove $\ga^2(\e z)$ and then to replace
the integral by that on $\R$ is bounded by
\begin{align*}
\Big(\int_{-\infty}^{-m_2/\e+2} +\int_{m_1/\e-2}^\infty\Big)
|G(z)U_0'(-z)| \, dz = O(e^{-c/\e}),
\end{align*}
for some $c>0$, since $G\in L^1(\R)$, $U_0'$ decays exponentially fast in $|z|$ 
(see Lemma \ref{lem:decay-U0}) and
$$
1_{[-m_2/\e+2, m_1/\e-2]} (z) \le \ga^2(\e z)
\le 1_{[-m_2/\e, m_1/\e]} (z).
$$

The square of the second norm in \eqref{eq:G-tau12-CP}, after expansion, 
is equal to
\begin{align}  \label{eq:expand-CP}
& \lan \check{G},\hat \t_2^\e\ran_{L^2(\T)}^2 
\|\hat \t_2^\e(\e z)\|_{L^2(\e^{-1}\T)}^2
-2 \lan G,e\ran_{L^2(\R)}  
\lan \check{G},\hat \t_2^\e\ran_{L^2(\T)}
\lan \hat \t_2^\e(\e z), e \ran_{L^2(\e^{-1}\T)}   \\
& \hskip 20mm 
+ \lan G,e\ran_{L^2(\R)}^2 \|e\|_{L^2(\e^{-1}\T)}^2.   \notag
\end{align}
The first term in \eqref{eq:expand-CP} is computed as
\begin{align*}
\lan \check{G},\hat \t_2^\e\ran_{L^2(\R)}^2 
\e^{-1} \|\hat \t_2^\e\|_{L^2(\T)}^2 
& = \e \big( A_0^{-1} \lan U_0'(-\cdot),G\ran_{L^2(\R)}  +o(1)\big)^2  \e^{-1}\\
& = \lan G,e\ran_{L^2(\R)}^2+ o(1),
\end{align*}
by noting $\|\hat \t_2^\e\|_{L^2(\T)}= 1$ and
\eqref{eq:G-tau-CP}.  The second term is rewritten as 
\begin{align*}
& -2 \lan G,e\ran_{L^2(\R)}  
\sqrt{\e} \big( A_0^{-1} \lan G, U_0'(-\cdot)\ran_{L^2(\R)}  +o(1)\big) \\
& \hskip 20mm \times
\e^{-1} \sqrt{\e} \big( A_0^{-1} \lan e, U_0'(-\cdot)\ran_{L^2(\R)}  +o(1)\big)
  \\
&\;\;  = -2 \lan G,e\ran_{L^2(\R)}^2 + o(1),
\end{align*}
where we used \eqref{eq:G-tau-CP} for $G$ and also took $G=e$ by rewriting
$\lan \hat \t_2^\e(\e z), e \ran_{L^2(\e^{-1}\T)}
= \e^{-1} \lan \hat \t_2^\e, \check{e} \ran_{L^2(\T)}$.
The third term behaves like $\lan G,e\ran_{L^2(\R)}^2 + o(1)$,
Therefore, we see that \eqref{eq:expand-CP} so that
the second norm in \eqref{eq:G-tau12-CP}
converges to $0$ as $\e\downarrow 0$.

Next, we consider the first norm in \eqref{eq:G-tau12-CP}.
One can make a similar calculation for $\t_1^\e(x):= -\ga^1(x) v_x^\e(x)$ 
and obtain 
\begin{align*}
\lan \check{G},\hat \t_1^\e\ran_{L^2(\T)}  
& =- \sqrt{\e} (A_0+o(1))^{-1} \int_{m_1}^{m_2} \ga^1(x) G(x/\e)  
\Big( \tfrac1\e U_0'(\tfrac{x-h_2}\e) + R^\e(x) \Big) dx.
\end{align*}
So, noting that $\|U_0'\|_{L^\infty(\R)}<\infty$ and $|R^\e(x)|\le \frac{C}{\sqrt{\e}}$,
we have
\begin{align*}
|\lan \check{G},\hat \t_1^\e\ran_{L^2(\T)}|
 \le \frac{C'}{\sqrt{\e}} \int_{m_1}^{m_2} |G(x/\e)|dx
 \le C' \sqrt{\e} \int_{m_1/\e}^\infty |G(z)|dz.
\end{align*}
Therefore, noting that $\|\hat \t_1^\e(\e z)\|_{L^2(\e^{-1}\T)}=\e^{-1/2}$
and $G\in L^1(\R)$,
we conclude that the first norm in \eqref{eq:G-tau12-CP} also converges to $0$
as $\e\downarrow 0$.

Thus, \eqref{eq:3.8-CP} is shown.  The proof of the proposition is complete.
\end{proof}

Next, we give a simple proof of Lemma \ref{lem:3.4} assuming that
$\partial_z^2 H \in L^2(\R\times\T^{d-1})$.  We prepare a lemma.

\begin{lem}  \label{lem:B.12}
If $H\in L^2(\R\times\T^{d-1})$ satisfies 
$\partial_z^2 H \in L^2(\R\times\T^{d-1})$,
then we have
\begin{align*}
&\sup_{0\le t \le T} \sup_{K\ge 1} \|T_t^KH\|_
{L^2(\sqrt{K}\T\times \T^{d-1})}  <\infty,
\intertext{and}
& \sup_{0\le t \le T} \sup_{K\ge 1} \|\partial_z^2 T_t^KH\|_
{L^2(\sqrt{K}\T\times \T^{d-1})}  <\infty.
\end{align*}
\end{lem}

\begin{proof}
The first bound follows from the contraction property of $e^{-tK\mathcal{A}^K}$
on $L^2(\sqrt{K}\T)$ and $e^{t\De_{\ux}}$ on $L^2(\T^{d-1})$. 

To show
the second, recalling $\mathcal{A}=-\partial_z^2- f'(\bar v^K(z))$ and noting 
$f'(\bar v^K(z))$ is bounded, we have $\mathcal{A}H\in 
L^2(\sqrt{K}\T\times \T^{d-1})$ and its norm is bounded in $K$, from 
our assumption $H, \partial_z^2 H \in L^2(\R\times\T^{d-1})$. 
Then, since $\mathcal{A} e^{tK\mathcal{A}} = e^{tK\mathcal{A}}\mathcal{A}$,
\begin{align*}
\partial_z^2 T_t^KH
& = - \mathcal{A}T_t^KH - f'(\bar v^K(z))T_t^KH \\
& = - T_t^K\mathcal{A}H - f'(\bar v^K(z))T_t^KH.
\end{align*}
However, since $T_t^K$ is a contraction and the norm of $\mathcal{A}H$
in $L^2(\sqrt{K}\T\times \T^{d-1})$ is bounded in $K$, the norm of the first 
term in $L^2(\sqrt{K}\T\times \T^{d-1})$ is bounded in $K$.  For the second 
term, note the boundedness of $f'(\bar v^K(z))$ and the boundedness of 
the norm of $T_t^KH$ in $L^2(\sqrt{K}\T\times \T^{d-1})$ in $K$.  
This shows the conclusion.
\end{proof}

Now, Lemma \ref{lem:3.4} is easily shown by the integration by parts:
\begin{align*}
\|\partial_z T_t^KH\|_{L^2(\sqrt{K}\T\times \T^{d-1})}^2
& = \int_{\sqrt{K}\T\times \T^{d-1}} (\partial_z T_t^KH)^2 dzd\uy \\
& = - \int_{\sqrt{K}\T\times \T^{d-1}} T_t^KH \cdot \partial_z^2 T_t^KH dzd\uy \\
& \le \|T_t^KH\|_{L^2(\sqrt{K}\T\times \T^{d-1})} \|\partial_z^2 T_t^KH\|_{L^2(\sqrt{K}\T\times \T^{d-1})},
\end{align*}
which is bounded in $K\ge 1$ by Lemma \ref{lem:B.12}.

\subsection{Construction of the solution of \eqref{eq:1-CP} on $\T$}
\label{sec:construction}

We assume $f\in C^\infty(\R)$ has three zeros $\pm 1$ (stable) and 
$\rho_*\in (-1,1)$ 
(unstable), but the balance condition is unnecessary in this section.

We first consider \eqref{eq:1-CP} under the Dirichlet boundary
condition $v(0)=v(\ell) = \rho_*$ and denote the solutions by $\phi(x,\ell,\pm 1)$,
$x\in [0,\ell]$, satisfying (i) $\phi(x,\ell,+ 1)>\rho_*$ for $x\in (0,\ell)$
or (ii) $\phi(x,\ell,- 1)<\rho_*$ for $x\in (0,\ell)$; see (2.1) of \cite{CP}.
We change the Dirichlet boundary value  $0$ in \cite{CP} to $\rho_*$,
as we noted before.

Recall the following lemma; cf.\ (7.2) of \cite{CP}.

\begin{lem} \label{lem:2.1-CP}
Let $v$ be a solution of \eqref{eq:1-CP}.  Then,
\begin{align}\label{eq:e-CP}
e(x):= \tfrac{\e^2}2 (v_x(x))^2-V(v(x))
\end{align}
is constant, i.e., $e(x)\equiv e$.  In particular, $v(x)$ satisfies the first order
ODE:
\begin{align}\label{eq:1.u'e}
v_x(x) = \pm \sqrt{2(V(v(x))+ e)/\e^2},
\end{align}
where $\pm$ is determined by the sign of $v_x(x)$. 
Note that $V(v(x))+e = \tfrac{\e^2}2 (v_x(x))^2 \ge 0$.
\end{lem}

The solution of \eqref{eq:1-CP} has a symmetry:

\begin{lem} \label{lem:2.3-CP}
If $v$ is a solution of \eqref{eq:1-CP} satisfying $v_x(x_0)=0$
at some $x_0$, then we have $v(x) = v(2x_0-x)$.
\end{lem}

\begin{proof}
Set $\tilde v(x) := v(2x_0-x)$.  Then it satisfies the same ODE \eqref{eq:1-CP} 
as $\e^2\tilde v_{xx}(x) = \e^2v_{xx}(2x_0-x) = -f(v(2x_0-x)) = -f(\tilde v(x))$, and also
$\tilde v(x_0)= v(x_0)$ and $\tilde v_x(x_0) = - v_x(x_0)=0$. The conclusion
follows by the uniqueness of the solution of the ODE.
\end{proof}

We also note a simple comparison theorem for ODEs:

\begin{lem} \label{lem:2.2-CP}
Let $v_1(x), v_2(x), x\ge 0$ be solutions of the ODEs $v_1'=b_1(v_1)$ and 
$v_2'=b_2(v_2)$ for $x>0$, where we write $v_i'$ for $(v_i)_x$, $i=1,2$.
If $b_1(v)\ge b_2(v), v\in \R$ and $v_1(0)\ge v_2(0)$,
then we have $v_1(x)\ge v_2(x), x\ge 0$.
\end{lem}

To construct the solution $v$ of \eqref{eq:1-CP} on $\T$, 
we first determine $\ell\in (0,1)$ such that
\begin{align}\label{eq:1-consistent}
\phi_x(\ell-,\ell,-1) = \phi_x(0+,1-\ell,+1),
\end{align}
so that $v$ defined by \eqref{eq:2.4-CP} below is $C^1$.

\begin{lem} \label{lem:2.4-CP}
Such an $\ell=\ell^*\in (0,1)$ uniquely exists.
\end{lem}

\begin{proof}
To show this, consider two solutions $v_1(x), v_2(x), x\ge 0$ of \eqref{eq:1-CP} 
such that $v_1(0)=v_2(0)=\rho_*$ and $v_1'(0+)\ge v_2'(0+)\ge 0$.
Then, for $e_i=e_i(x)$ defined by \eqref{eq:e-CP} from $v_i$, $i=1,2$, we have
$e_1\ge e_2$, since this holds at $x=0+$.  

Therefore, since $\sqrt{2(V(v)+ e_1)/\e^2} \ge \sqrt{2(V(v)+ e_2)/\e^2}, v\in \R$
on the right-hand side of
\eqref{eq:1.u'e}, we obtain $v_1(x)\ge v_2(x)$ and then $v_1'(x)\ge v_2'(x)$
as long as $v_1'(x), v_2'(x)\ge 0$ by Lemma \ref{lem:2.2-CP}.

Also noting Lemma \ref{lem:2.3-CP} (after $v_2'(x)\le 0$ occurs and then
$v_1'(x)\le 0$), from the above argument, under $v_1(0)=v_2(0)=\rho_*$,
we see that $v_1'(0+) \ge v_2'(0+)\ge 0$ implies $\ell_1\ge \ell_2$,
where $\ell_i := \inf\{x>0; v_i(x)=\rho_*\}$.  Since $\ell$ depends continuously
on $v_x(0+)$, we see by the inverse function theorem that $\ell$ is an 
increasing continuous function of $v_x(0+)$.

The right-hand side of \eqref{eq:1-consistent} is decreasing in $\ell$ and
takes values in $(0, \phi_x(0+,1,+1))$, while the left-hand side of 
\eqref{eq:1-consistent} 
is increasing in $\ell$ and takes values in $(0, - \phi_x(0+,1,-1))$ (one can show
this similarly, though $\phi(\cdot,\cdot,-1)\le \rho_*$);
note that the left-hand side is equal to $-\phi_x(0+,\ell,-1)$ due to
the symmetry of $\phi$
under the reflection at $\ell/2$ (by Lemma \ref{lem:2.3-CP}).
Thus, by the intermediate value theorem, we can uniquely choose $\ell=\ell^*$
such that \eqref{eq:1-consistent} holds and this completes the proof.
\end{proof}

Define $v(x), x\in \T,$ by piecing $\phi(\cdot,\cdot,\pm 1)$ as
\begin{align}  \label{eq:2.4-CP}
v(x) = \left\{
\begin{aligned}
\phi(x,\ell,-1), & \quad x\in [0,\ell], \\
\phi(x-\ell,1-\ell,+1), & \quad x\in [\ell,1],
\end{aligned}
\right.
\end{align}
with $\ell=\ell^*$ determined by Lemma \ref{lem:2.4-CP}.

\begin{prop}  \label{prop:B.11}
$v(x), x\in \T$ is a solution of \eqref{eq:1-CP} on $\T$ with $N=2$,
and it is unique except for translation.
\end{prop}

\begin{proof}
In fact, $v\in C^1$ by \eqref{eq:1-consistent} (especially at $x=\ell^*$.
For $x=0$, note Lemma \ref{lem:2.3-CP}).  Moreover, it is $C^2$
and satisfies \eqref{eq:1-CP} also at $x=0$ and $x=\ell^*$.  Indeed,
for every $\ell\in (0,1)$,
$\phi(\cdot,\ell,-1)$ is convex and $\phi(\cdot,\ell,+1)$ is concave:
$$
\phi_{xx}(x,\ell,-1) >0, \; \phi_{xx}(x,\ell,+1) <0, 
\quad x \in (0,\ell),
$$
by the equation \eqref{eq:1-CP} noting $f(\phi(x,\ell,-1))<0$ or 
$f(\phi(x,\ell,+1))>0$, respectively, and 
$$
\phi_{xx}(0+,\ell,\pm 1) = \phi_{xx}(\ell-,\ell,\pm 1) =0, 
$$
by letting $x\downarrow 0$ and $x\uparrow \ell$ in the equation \eqref{eq:1-CP}
noting $f(\rho_*)=0$.
Therefore, $v(x)$ constructed as above is $C^2$ and satisfies  \eqref{eq:1-CP}
for all $x\in \T$, including $x=0$ and $\ell^*$.  Indeed, by the above observation
taking $\ell=\ell^*$ with $-1$ and $\ell=1-\ell^*$ with $+1$,
we have $v_{xx}(x) =0$ and $f(v(x)) = f(\rho_*)=0$ at $x=0$ and $\ell^*$. 

The uniqueness of $v$ except for translation follows from the uniqueness of $\ell^*$.
\end{proof}

\section*{Acknowledgements}

The author thanks Sunder Sethuraman for his comments for Section 
\ref{Appendix:B}.  He also thanks Hendrik Weber for a useful discussion on
singular SPDEs.  This work was supported in part by the 
International Scientists Project of BJNSF, No.\ IS23007.

\end{document}